\documentclass[11pt,letterpaper]{amsart}
\usepackage{libertine}
\usepackage[libertine,libaltvw]{newtxmath}
\usepackage[scr=rsfso]{mathalfa}
\usepackage{t1enc}
\usepackage[english]{babel}
\usepackage[unicode]{hyperref}

\numberwithin{equation}{section}
\theoremstyle{plain}
\newtheorem{thm}[equation]{Theorem}
\newtheorem{lem}[equation]{Lemma}
\newtheorem{pro}[equation]{Proposition}
\newtheorem{cor}[equation]{Corollary}
\theoremstyle{definition}
\newtheorem{df}[equation]{Definition}
\newtheorem{exa}[equation]{Example}
\newtheorem{rem}[equation]{Remark}
\newtheorem*{open}{Open problem}
\let\OldSubsection\subsection
\renewcommand{\subsection}{\setcounter{subsection}{\value{equation}}\stepcounter{equation}\OldSubsection}
\newcommand*{\Nbb}{\mathbb{N}}
\newcommand*{\Zbb}{\mathbb{Z}}
\newcommand*{\Rbb}{\mathbb{R}}
\newcommand*{\Bcal}{{\mathcal B}}
\newcommand*{\Ccal}{{\mathcal C}}
\newcommand*{\Hcal}{\mathcal{H}}
\newcommand*{\Wcal}{\mathcal{W}}
\providecommand*{\coloneq}{:=}
\newcommand*{\eps}{\varepsilon}
\newcommand*{\Om}{\Omega}
\newcommand*{\vphi}{\varphi}
\newcommand*{\dOm}{{\partial\Omega}}
\newcommand*{\dd}{{\mathrm d}}
\newcommand*{\loc}{{\mathrm{loc}}}
\newcommand*{\PI}{{\mathrm{PI}}}
\newcommand*{\wk}{{\textup{weak-}}}
\newcommand*{\dbl}{{\textup{dbl}}}
\newcommand*{\emb}{\hookrightarrow}
\DeclareMathOperator*{\esssup}{ess\,sup}
\DeclareMathOperator*{\essinf}{ess\,inf}
\DeclareMathOperator{\dist}{dist}
\DeclareMathOperator{\rad}{rad}
\DeclareMathOperator{\Id}{Id}
\DeclareMathOperator{\Lip}{Lip}
\DeclareMathOperator{\LIP}{LIP}
\DeclareMathOperator{\diam}{diam}
\DeclareMathOperator{\Ext}{Ext}
\newcommand*{\symdiff}{{\mathbin{\triangle}}}
\begin{document}
\begin{abstract}
Trace classes of Sobolev-type functions in metric spaces are subject of this paper. In particular, functions on domains whose boundary has an upper codimension-$\theta$ bound are considered. Based on a Poincar\'e inequality, existence of a Borel measurable trace is proven whenever the power of integrability of the ``gradient'' exceeds $\theta$. The trace $T$ is shown to be a compact operator mapping a Sobolev-type space on a domain into a Besov space on the boundary. Sufficient conditions for $T$ to be surjective are found and counterexamples showing that surjectivity may fail are also provided. The case when the exponent of integrability of the ``gradient'' is equal to $\theta$, i.e., the codimension of the boundary, is also discussed. Under some additional assumptions, the trace lies in $L^\theta$ on the boundary then. Essential sharpness of these extra assumptions is illustrated by an example.
\end{abstract}
\subjclass[2010]{Primary 46E35; Secondary 26B30, 26B35, 30L99, 31E05, 46E30.}
\keywords{trace, extension, Besov space, Haj\l asz space, Newtonian space, metric measure space, Hausdorff codimension-$\theta$ measure, Poincar\'e inequality, upper gradient, Sobolev function.}
\title[Trace and extension theorems in metric spaces]{Trace and extension theorems\\for Sobolev-type functions in metric spaces}
\author{Luk\'{a}\v{s} Mal\'{y}}
\thanks{The research of the author was supported by the Knut and Alice Wallenberg Foundation (Sweden).}
\address{Department of Mathematical Sciences\\%
         University of Cincinnati\\%
         Cincinnati, OH 45221-0025\\%
         U.S.A.}
\email{lukas.maly@matfyz.cz}
\date{April 20, 2017}
\maketitle
\section{Introduction and Overview}
Over the past two decades, analysis in metric measure spaces (and non-linear potential theory, in particular) has attracted a lot of attention, e.g., \cite{BjoBjo,BjoBjoSha,Haj,HajKos,HKST,KinMar}. See also \cite[Chapter 22]{HeKiMa} and references therein. 
The metric space setting provides a wide framework to study partial differential equations and, specifically, boundary value problems. These seek to find solutions to an equation in a domain, subject to a prescribed boundary condition. Most thoroughly studied problems deal with the Dirichlet boundary condition, where a trace of the solution is prescribed, and with the Neumann condition, where the normal derivative of the solution at the boundary is given. 

Both the Dirichlet and the Neumann problems in a domain in a metric space can be formulated as minimization of a certain energy functional. While the Dirichlet problem works with an extension of the prescribed data, see, e.g.,~\cite{BjoBjo}, the energy functional for the Neumann problem contains an integral of the trace of the solution, see~\cite{MalSha}. Given a domain $\Om$, it is therefore natural to ask what kind of boundary data can be extended to a Sobolev-type function in $\Om$, and conversely, what boundedness properties the trace operator exhibits when mapping a class of Sobolev-type functions in $\Om$ to some function class on $\dOm$.

Such questions were first studied in the Euclidean setting by Gagliardo~\cite{Gag}, who proved that there is a bounded surjective linear trace operator $T: W^{1,1}(\Rbb^{n+1}_+) \to L^1(\Rbb^n)$ with a non-linear right inverse (while the non-linearity is not an artifact of the proof, as can be seen from~\cite{Pee}). Moreover, he proved that $T: W^{1,p}(\Rbb^{n+1}_+) \to B^{1-1/p}_{p,p}(\Rbb^n)$, where $B^{1-1/p}_{p,p}(\Rbb^n)$ stands for a Besov space (cf.\@ Section~\ref{sec:besov-emb} below), is bounded for every $p>1$ and there exists a bounded linear extension operator that acts as a right inverse of $T$. The results can be easily adapted to domains with Lipschitz boundary in $\Rbb^n$. Later, Besov~\cite{Bes} followed up with investigating how Besov functions in a (half)space can be restricted to a hyperplane.

In the Euclidean setting, the trace theorems have been substantially generalized by Jonsson--Wallin~\cite{JonWal}, see also \cite{JonWalBk}, who showed that $T: W^{1,p}(\Rbb^n) \to B^{1-\theta/p}_{p,p}(F)$ is a bounded linear surjection provided that $F$ is a compact set that is Ahlfors codimension-$\theta$ regular (cf.\@ \eqref{eq:H-lower-massbound} and \eqref{eq:H-upper-massbound} below) and $p>\theta$. However, they did not address the case $p=\theta > 1$.

In the metric setting, \cite{LahSha} proved that the trace of a function of bounded variation on $\Om$ (which arise as relaxation of Newton--Sobolev $N^{1,1}(\Om)$ functions) lies in $L^1(\dOm)$ provided that $\Om$ is a domain of finite perimeter that admits a $1$-Poincar\'e inequality and its boundary is codimension-$1$ regular. An extension operator for $L^1$ boundary data was constructed in \cite{Vit} for Lipschitz domains in 
Carnot--Carath\'eodory spaces and in \cite{MalShaSni} for domains with codimension-$1$ regular boundary in general metric spaces.

Recent paper~\cite{SakSot} discusses traces of Besov, Triebel--Lizorkin, and Haj\l asz--Sobolev functions to porous Ahlfors regular closed subsets of the ambient metric space in case that the Haj\l asz gradient is summable in sufficiently high power. The method there is based on hyperbolic fillings of the metric space, cf.\@ \cite{BonSak,Sot}. The paper~\cite{SakSot} also shows possible relaxation of Ahlfors regularity by replacing it with the Ahlfors codimension-$\theta$ regularity. 

The goal of the present paper is to study trace and extension theorems for Sobolev-type functions in domains in a metric measure space with a lower and/or upper codimension bound (possibly unequal) for the boundary, including the case when the upper codimension bound equals the integrability exponent of the gradient. Unlike the trace theorems, the extension theorems make no use of any kind of Poincar\'e inequality. This gives new results also for weighted classical Sobolev functions in the Euclidean setting.

Let us now state the main theorems of this paper.
\begin{thm}
\label{thm:main-John-pPI}
Let $(X, \dd, \mu)$ be a metric measure space. Let $\Om \subset X$ be a John domain whose closure is compact. Assume that $\mu\lfloor_{\Om}$ is doubling and admits a $p$-Poincar\'e inequality for some $p> 1$. (If $\Om$ is uniform, then $p=1$ is also allowed.) Let $\dOm$ be endowed with Ahlfors codimension-$\theta$ regular measure $\Hcal$ for some $\theta < p$. Then, there is a bounded linear surjective trace operator $T: N^{1,p}(\Om, d\mu) \to B^{1-\theta/p}_{p,p}(\dOm, d\Hcal)$, which satisfies 
\begin{equation}
  \label{eq:defoftrace}
  \lim_{r\to0^+}\fint_{B(z,r)\cap\Om} |u(x)-Tu(z)|\, d\mu(x)=0, \quad \text{for $\Hcal$-a.e.\@ }z\in\dOm.
\end{equation}
Moreover, there exists a bounded linear extension operator $E: B^{1-\theta/p}_{p,p}(\dOm, d\Hcal) \to N^{1,p}(\Om, d\mu)$ that acts as a right inverse of $T$, i.e., $T\circ E = \Id$ on $B^{1-\theta/p}_{p,p}(\dOm, d\Hcal)$.
\end{thm}
Existence and boundedness of the trace operator is established by Theorems~\ref{thm:TraceJohn} (for John domains) and~\ref{thm:traceUniform} (for uniform domains). The linear extension operator is constructed in Section~\ref{sec:Besov-extension}, see Lemmata~\ref{lem:Lp-est_Whitney} and \ref{lem:trace-Lebesgue}, and Proposition~\ref{pro:extnBounds}. 
\begin{thm}
\label{thm:main-general}
Let $(X, \dd, \mu)$ be a metric measure space. Let $\Om\subset X$ be a domain such that $\mu \lfloor_{\Om}$ is doubling. Let $F \subset \overline{\Om}$ be a bounded set, endowed with a measure $\nu$ that satisfies codimension-$\theta$ upper bound, i.e., $\nu(B(z,r) \cap F) \lesssim \mu(B(z,r)\cap \Om)/r^\theta$ for every $z \in F$, $r \le 2 \diam F$, and some $\theta>0$. Then, for every $p>\theta$ with $p\ge 1$, there is a bounded linear trace operator $T: M^{1,p}(\Om, d\mu) \to B^{1-\theta/p}_{p, \infty}(F, d\nu)$. However, the trace operator need not be surjective and there can be a function $u \in M^{1,p}(\Om, d\mu)$ such that $Tu \in B^{1-\theta/p}_{p, \infty}(F, d\nu) \setminus \bigcup_{\alpha>1-\theta/p} B^{\alpha}_{p, \infty}(F, d\nu)$.

If, in addition, $\Om$ supports a $p$-Poincar\'e inequality, then $T: N^{1,p}(\Om, d\mu) \to B^{1-\theta/p}_{p,\infty}(F, d\nu)$.
\end{thm}
The assertion follows from Theorem~\ref{thm:TraceOpB}, which is however formulated for $F = \dOm$; nevertheless, the proof goes through even in the setting of Theorem~\ref{thm:main-general}. The fact that $T$ need not be surjective is shown by Examples~\ref{exa:T-not-surjective} and~\ref{exa:T-not-surjective-uniform}. However, it is unclear whether the target space $B^{1-\theta/p}_{p,\infty}(F)$ is optimal under these mild assumptions.
\begin{open}
Under the assumptions of Theorem~\ref{thm:main-general}, is $T: M^{1,p}(\Om, d\mu) \to B^{1-\theta/p}_{p,p}(F, d\nu)$ bounded? Note however that even if $T$ were a bounded mapping into $B^{1-\theta/p}_{p,p}(F, d\nu)$, then it still could lack surjectivity as is seen in Examples~\ref{exa:T-not-surjective} and \ref{exa:T-not-surjective-uniform}.
\end{open}
\begin{thm}
Let $(X, \dd, \mu)$ be a metric measure space. Let $\Om \subset X$ be a bounded domain that admits a $\theta$-Poincar\'e inequality for some $\theta \ge 1$. Assume that $\mu\lfloor_{\Om}$ is doubling and that $\dOm$ is endowed with an Ahlfors codimension-$\theta$ regular measure $\Hcal$. Let $w_\eps(x) = \log(2 \diam \Om/\dist(x,\dOm))^{\theta + \eps}$ for some $\eps>0$. Then, there is a bounded linear trace operator $T: N^{1,\theta}(\Om, w_\eps\, d\mu) \to L^\theta(\dOm, d\Hcal)$, which satisfies \eqref{eq:defoftrace}.

Moreover, there exists a bounded non-linear extension operator $E: L^p(\dOm, d\Hcal) \to N^{1,p}(\Om, d\mu)$ that acts as a right inverse of $T$, i.e., $T\circ E = \Id$ on $L^p(\dOm, d\Hcal)$.

The choice of weight $w_\eps$ with an arbitrary $\eps>0$ is essentially sharp when $\theta > 1$ since for every $-\eps<0$ there exists $u \in N^{1,p}(\Om, w_{-\eps}\,d\mu)$ such that $Tu \equiv \infty$ in $\dOm$, i.e., \eqref{eq:defoftrace} fails.
\end{thm}
Existence and boundedness of the trace operator from the weighted Newtonian space follows from Theorem~\ref{thm:TraceOp-theta=p}. The non-linear extension operator is constructed in Section~\ref{sec:lp-extension}. Essential optimality of the weight is proven in Proposition~\ref{pro:T-intoLp-sharpness}, see also Example~\ref{exa:T-to-Lp-fail}. Observe that the behavior of the trace operator is substantially different when $\Hcal\lfloor_{\dOm}$ is codimension~$1$ regular since no weight is needed then and $L^1(\dOm, d\Hcal)$ is the trace class of $N^{1,1}(\Om, d\mu)$, see~\cite{LahSha,MalShaSni}.

The trace results stated for $N^{1,1}(\Om, d\mu)$ functions under the assumption of $1$-Poincar\'e inequality can be further generalized to functions of bounded variations. Details of such a generalization lie outside of the scope of the main interest and have therefore been omitted.

The paper is divided into 8 sections. We start off by preliminaries for the analysis in metric spaces and Sobolev-type functions considered in this paper. Section~\ref{sec:besov-emb} is then devoted to establishing elementary properties (continuous and compact embeddings, in particular) of Besov spaces in the metric setting. Trace theorems for general domains, whose boundary has an upper codimension-$\theta$ bound, are the subject of Section~\ref{sec:tracesGen}. Sharp trace theorems for John and uniform domains that are compactly embedded in $X$ are proven in Section~\ref{sec:tracesJohn}. A linear extension operator, i.e., a right inverse of the trace operator, for Besov data is constructed in Section~\ref{sec:Besov-extension}. A non-linear extension operator for $L^\theta$ boundary data is then constructed in Section~\ref{sec:lp-extension}. Finally, several examples that show sharpness of the hypothesis of the presented trace theorems are provided in Section~\ref{sec:examples}.
\section{Preliminaries \& Sobolev-type functions in the metric setting}
Throughout the paper, $(X, \dd, \mu)$ will be a metric space endowed with a $\sigma$-finite Borel regular measure that is non-trivial in the sense that $\mu(B) \in (0, \infty)$ for every ball $B \subset X$. We do not a priori assume that $X$ is (locally) complete, nor that $\mu$ is doubling in $X$. However, given a domain $\Om$, we will require that $\mu\lfloor_{\Om}$ is doubling and non-trivial, i.e., there is $c_\dbl \ge 1$ such that
\[
  0 < \mu(B(x, 2r) \cap \Om) \le c_\dbl \mu(B(x, r) \cap \Om) < \infty
\]
for every $x\in \overline{\Om}$ and every $r>0$, where $B(x,r)$ denotes an open ball of radius $r$ centered at $x$.
Then, \cite[Lemma~3.3]{BjoBjo} yields that there are $C_s>0$ and $s>0$ such that
\[
  \frac{\mu(B(x,r) \cap \Om)}{\mu(B(y,R) \cap \Om)} \ge C_s \biggl(\frac{r}{R} \biggr)^s
\]
for all $0<r\le R$, $y\in \overline{\Om}$, and $x\in B(y,R) \cap \overline{\Om}$. In particular, if $\Om$ is bounded, then $\mu\lfloor_\Om$ has a lower mass bound $\mu(B(x,r) \cap \Om) \ge c_s r^s$ since
\begin{equation}
  \label{eq:ball-mu-vs-rads}
  \frac{\mu(B(x,r) \cap \Om)}{r^s} \ge C_s \frac{\mu(B(x, 2 \diam \Om) \cap \Om)}{(2 \diam \Om)^s} = \frac{\mu(\Om)}{(2 \diam \Om)^s}\eqcolon c_s\,,\quad x\in \overline{\Om} .
\end{equation}

In the rest of the paper, $f_E$ denotes the \emph{integral mean} of a function $f\in L^0(E)$ over a measurable set $E \subset X$ of finite positive measure, defined as
\[
  f_E = \fint_E f\,d\mu = \frac{1}{\mu(E)} \int_E f\,d\mu
\]
whenever the integral on the right-hand side exists, not necessarily finite though. Given an open ball $B=B(x,r) \subset X$ and $\lambda>0$, the symbol $\lambda B$ denotes the inflated ball $B(x, \lambda r)$.

Throughout the paper, $C$ represents various constants and its precise value is not of interest. Moreover, its value may differ in each occurrence. Given expressions $a$ and $b$, we write $a \lesssim b$ if there is a constant $C>0$ such that $a\le C b$. We say that $a$ and $b$ are \emph{comparable}, denoted by $a\approx b$, if $a \lesssim b$ and $b\lesssim a$ at the same time.

Given a set $F \subset \overline{\Om}$ endowed with a $\sigma$-finite Borel regular measure $\Hcal$, we say that $\Hcal$ satisfies a \emph{lower codimension-$\vartheta$ bound} with some $\vartheta > 0$ if
\begin{equation}
 \label{eq:H-lower-massbound}
 \Hcal(B(x,r)\cap F) \ge c_\vartheta \frac{\mu(B(x,r)\cap \Om)}{r^\vartheta}, \quad\text{for all } x\in F, 0<r<2\diam F.
\end{equation}
Analogously, we say that $\Hcal$ satisfies a \emph{upper codimension-$\theta$ bound} with some $\theta > 0$ if
\begin{equation}
 \label{eq:H-upper-massbound}
 \Hcal(B(x,r)\cap F)\le c_\theta \frac{\mu(B(x,r)\cap \Om)}{r^\theta}, \quad\text{for all } x\in F, r>0.
\end{equation}
Note that the trace theorems established in Sections~\ref{sec:tracesGen} and \ref{sec:tracesJohn} make use only of the upper bound~\eqref{eq:H-upper-massbound} with $0<\theta < s$, whereas the extension theorems in Sections~\ref{sec:Besov-extension} and~\ref{sec:lp-extension} assume both \eqref{eq:H-lower-massbound} and~\eqref{eq:H-upper-massbound} with $0 < \vartheta \le \theta < s$. In the rest of the paper, the conditions \eqref{eq:H-lower-massbound} and \eqref{eq:H-upper-massbound} will be referred to with $F = \dOm$, while the notation $\Hcal$ is suggestively used to promote the measure's relation to the codimension-$\theta$ (and codimension-$\vartheta$) Hausdorff measure, see~\eqref{eq:df-Hcodim} and Lemma~\ref{lem:codim-bounds-vs-Hcodim} below.

If both \eqref{eq:H-lower-massbound} and \eqref{eq:H-upper-massbound} hold true with $0< \vartheta = \theta < s$, then $F$ (as well as $\Hcal$) will be called \emph{Ahlfors codimension-$\theta$ regular}.

Observe that \eqref{eq:H-lower-massbound} with \eqref{eq:ball-mu-vs-rads} imply that $\Hcal$ has a lower mass bound $\Hcal(B(x,r) \cap F) \gtrsim r^{s-\vartheta}$. On the other hand, \eqref{eq:H-upper-massbound} guarantees that $\Hcal \lfloor_{F}$ is doubling provided that $\mu\lfloor_{\Om}$ is doubling. If $\Om$ (and hence $F$) is bounded, then $\Hcal \lfloor_{F}$ has a lower mass bound $\Hcal(B(x,r) \cap F) \gtrsim r^{\alpha}$ with some $\alpha \ge s-\theta$ by \eqref{eq:H-upper-massbound}. However, the lower bound $s-\theta$ need not be optimal (as the optimal $\alpha$ might be smaller).

The following lemma shows how \eqref{eq:H-lower-massbound} and \eqref{eq:H-upper-massbound} relate $\Hcal(F)$ to the measure of an exterior shell of $F$ in $\Om$.
\begin{lem}
\label{lem:shell-measure}
Let $\Om \subset X$ be a domain with $\mu\lfloor_\Om$ doubling. Let $F \subset \overline{\Om}$ be endowed with a measure $\Hcal$ and set $\Om_R = \{ x\in \Om : \dist(x, F) < R\}$ for $R>0$.
\begin{enumerate}
	\item \label{it:shellmeas1} If $F$ has a lower codimension-$\vartheta$ bound, then $\mu(\Om_R) \lesssim \Hcal(F) R^\vartheta$.
	\item \label{it:shellmeas2} If $F$ has an upper codimension-$\theta$ bound, then $\mu(\Om_R) \gtrsim \Hcal(F) R^\theta$.
\end{enumerate}
\end{lem}
It will become apparent from the proof that it would suffice to assume the relation between $\mu$ and $\Hcal$ to hold only for all $z \in E$, where $E$ is a dense subset of $F$, and all $r_k = 2^{-k} \diam(\Om)$, $k \in \Nbb_0$.
\begin{proof}
  Since $\Om_R = \bigcup_{z\in F} (B(z,R) \cap \Om)$ and $\mu\lfloor_\Om$ is doubling, we can apply the simple 5-covering lemma to obtain a collection $\{z_k\}_k \subset F$ such that $\Om_R \subset \bigcup_{k} (B(z_k, 5R) \cap \Om)$ while the balls $\{B(z_k, R)\}_k$ are pairwise disjoint. In case of \ref{it:shellmeas1}, we have that
  \[
    \mu(\Om_R) \le \sum_k \mu(B(z_k, 5R) \cap \Om) \lesssim \sum_k \mu(B(z_k, R) \cap \Om) \lesssim \sum_k R^\vartheta \Hcal(B(z_k, R) \cap F) \le R^\vartheta \Hcal(F).
  \]
On the other hand, in case of \ref{it:shellmeas2}, we obtain that $\Hcal\lfloor_{F}$ is doubling and hence
  \begin{align*}
    \mu(\Om_R) & \ge \sum_k \mu(B(z_k, R) \cap \Om) \gtrsim \sum_k R^\theta \Hcal(B(z_k, R) \cap F) \\
    & \gtrsim \sum_k R^\theta \Hcal(B(z_k, 5R) \cap F) \ge R^\theta \Hcal\biggl(\bigcup_k B(z_k, 5R) \cap F\biggr) = R^\theta \Hcal(F).
    \qedhere
  \end{align*}
\end{proof}
Given $E\subset \overline{\Om}$ and $\theta>0$, we define its \emph{co-dimension $\theta$ Hausdorff measure} $\Hcal^{\textup{co-}\theta}(E)$ by
\begin{equation}
  \label{eq:df-Hcodim}
  \Hcal^{\textup{co-}\theta}(E) = \lim_{\delta\to 0^+}\ \inf\biggl\{\sum_i\frac{\mu(B_i \cap \overline{\Om})}{\rad(B_i)^{\theta}}: B_i\text{ balls in }\overline{\Om}, \rad(B_i)<\delta, E\subset \bigcup_i B_i\biggr\}\,.
\end{equation}
More generally, it is possible (and natural) to define $\Hcal^{\textup{co-}\theta}(E)$ for a set $E \subset X$ by considering the limit of infima over $\sum_i \mu(B_i)/\rad(B_i)^\theta$ instead. If $\mu$ is globally doubling and $\Om$ satisfies the measure density condition, i.e., $\mu(B\cap \Om) \gtrsim \mu(B)$ for every ball $B$ centered in $\Om$, then such a measure would be comparable with \eqref{eq:df-Hcodim} whenever $E \subset \overline{\Om}$.
\begin{lem}
\label{lem:codim-bounds-vs-Hcodim}
Let $\Om \subset X$ be a domain with $\mu\lfloor_\Om$ doubling. Let $F \subset \overline{\Om}$ be endowed with a measure $\Hcal$.
\begin{enumerate}
	\item \label{it:Hcodim1} If $F$ has a lower codimension-$\vartheta$ bound, then $\Hcal^{\textup{co-}\vartheta}(E) \lesssim \Hcal(E)$ for every Borel set $E \subset F$.
	\item \label{it:Hcodim2} If $F$ has an upper codimension-$\theta$ bound, then $\Hcal^{\textup{co-}\theta}(E) \gtrsim \Hcal(E)$ for every Borel set $E \subset F$.
\end{enumerate}
\end{lem}
\begin{proof}
\ref{it:Hcodim1} Let $\{B_i\}_i$ be a cover of $E$ by balls of radius at most $\delta$. As $\mu\lfloor_{\Om}$ is doubling, we can use the simple Vitali $5$-covering lemma to extract a collection of pairwise disjoint balls that will cover $E$ after being inflated by the factor of 5. Then,
\begin{equation}
 \label{eq:H-vs-codimbound}
 \sum_i \frac{\mu(B_i \cap \Om)}{\rad(B_i)^\vartheta} \approx \sum_j \frac{\mu(5B_{i_j} \cap \Om)}{\rad(B_{i_j})^\vartheta} \approx \sum_j \frac{\mu(B_{i_j} \cap \Om)}{\rad(B_{i_j})^\vartheta} \lesssim \sum_j \Hcal(B_{i_j} \cap F) = \Hcal\biggl(\bigcup_j B_{i_j} \cap F\biggr)\,.
\end{equation}
In particular,
\[
  \Hcal^{\textup{co-}\vartheta}(E) \le \lim_{\delta \to 0^+} \inf\biggl\{\Hcal \biggl(\bigcup_i B_{i} \cap F\biggr): B_i\text{ balls in }\overline{\Om}, \rad(B_i)<\delta, E\subset \bigcup_i B_i\biggr\} = \Hcal(E)
\]
since $\Hcal$ is outer regular by~\cite[Theorem~1.10]{Mat}.

\ref{it:Hcodim2} Note that $\Hcal$ is doubling. Analogously as in \eqref{eq:H-vs-codimbound}, we have
\[
  \sum_i \frac{\mu(B_i \cap \Om)}{\rad(B_i)^\theta} \gtrsim \sum_j \Hcal\bigl( B_{i_j} \cap F\bigr) \approx \sum_j \Hcal\bigl( 5 B_{i_j} \cap F\bigr)
\ge \Hcal\biggl(\bigcup_j 5 B_{i_j} \cap F\biggr) \ge \Hcal(E)
\]
whenever $E \subset \bigcup_i B_i$. Thus, $\Hcal^{\textup{co-}\theta}(E) \gtrsim \Hcal(E)$.
\end{proof}
\subsection{Newtonian functions}
\label{subsec:newtonian}
In the metric setting, there need not be any natural distributional derivative structure that would allow us to define Sobolev functions similarly as in the Euclidean setting. Shanmugalingam~\cite{Sha} pioneered an approach to Sobolev-type functions in metric spaces via (weak) upper gradients. The interested reader can be referred to~\cite{BjoBjo,HKST}.

The notion of (weak) upper gradients (first defined in~\cite{HeiKos}, named ``\emph{very weak gradients}''), corresponds to $|\nabla u|$.
A Borel function $g: X \to [0,\infty]$ is an \emph{upper gradient} of $u: X\to \Rbb\cup\{\pm\infty\}$ if
\begin{equation}
  \label{eq:def-ug}
  \bigl|u(\gamma(b))-u(\gamma(a))\bigr| \le \int_{\gamma}g \, ds
\end{equation}
holds for all rectifiable curves $\gamma: [a,b] \to X$
whenever $u(\gamma(a))$ and $u(\gamma(b))$ are both finite and $\int_{\gamma} g \,ds=\infty$ otherwise.

If a function $u: X \to \Rbb$ is locally Lipschitz, then
\[
  \Lip f(x) = \limsup_{y\to x}\frac{|f(y)-f(x)|}{\dd(y,x)}, \quad x \in X,
\]
is an upper gradient of $u$, see for example \cite{Hei01}.

For $p\ge1$, one defines the \emph{Newtonian space} $N^{1,p}(X) = \{ u \in L^p(X): \|u\|_{N^{1,p}(X)}<\infty\}$, where
\[
  \|u\|_{N^{1,p}(X)}^p = \|u\|_{L^p(X)}^p + \inf\bigl\{ \|g\|_{L^p(X)}^p: \text{$g$ is an upper gradient of $u$}\bigr\}\,.
\]
For every $u \in N^{1,p}(X)$, the infimum in the definition of the $N^{1,p}(X)$ norm is attained by a unique \emph{minimal $p$-weak upper gradient} (which is minimal both pointwise and normwise). The phrase ``$p$-weak'' refers to the fact that \eqref{eq:def-ug} may fail for a certain negligible number of rectifiable curves.

Given a domain $\Om \subset X$, the class $N^{1,p}(\Om)$ is defined as above with $\Om$ being the underlying metric space (whose metric measure structure is inherited from $X$).
\subsection{Haj\l asz functions}
\label{subsec:hajlasz}
Another approach to Sobolev-type functions was introduced in~\cite{Haj96}. Given $u: X \to \Rbb \cup \{\pm \infty\}$, we say that $g: X \to [0, \infty]$ is a \emph{Haj\l asz ($\alpha$-fractional) gradient} if there is $E \subset X$ with $\mu(E) = 0$ such that
\[
  |u(x) - u(y)| \le \dd(x,y)^\alpha (g(x) + g(y))\quad \text{for every }x,y \in X \setminus E,
\]
where $\alpha \in (0,1]$. The phrase ``$\alpha$-fractional'' is typically dropped in case $\alpha = 1$.
Compared to the classical Sobolev spaces in $\Rbb^n$, Haj\l asz gradients for $\alpha=1$ correspond to $c_n M|\nabla u|$, where $M$ is the Hardy--Littlewood maximal operator.

For $p>0$ and $\alpha \in (0,1]$, one defines the \emph{Haj\l asz space} $M^{\alpha,p}(X) = \{ u \in L^p(X): \|u\|_{M^{\alpha,p}(X)}<\infty\}$, where
\[
  \|u\|_{M^{\alpha,p}(X)}^p = \|u\|_{L^p(X)}^p + \inf\bigl\{ \|g\|_{L^p(X)}^p: \text{$g$ is an $\alpha$-fractional Haj\l asz gradient of $u$}\bigr\}\,.
\]
The infimum is attained for $p>1$, but the minimal Haj\l asz gradient is not pointwise minimal. For a domain $\Om \subset X$, the space $M^{\alpha,p}(\Om)$ is defined as above with $\Om$ being the underlying metric measure space. Haj\l asz gradients are $p$-weak upper gradients, cf.\@ \cite{JSYY}, whence $M^{1,p}(\Om) \emb N^{1,p}(\Om)$ for all $p\ge1$. The inclusion may be proper unless $\Om$ admits a $p$-Poincar\'e inequality and $p>1$. 

More information about Haj\l asz functions and the motivating ideas can be found in \cite{Haj}.
\subsection{Functions with a Poincar\'e inequality}
\label{subsec:PI}
Let $p \in [1 , \infty)$ and $q\in (0 ,p]$. Then, the space $P^{1,p}_q(X)$ consists of those functions $u \in L^p(X)$ for which there is a constant $\lambda \ge 1$ and a non-negative $g\in L^p(X)$ such that
\begin{equation}
  \label{eq:def-P1pq}
  \fint_B |u-u_B|\,d\mu \le \rad(B) \biggl( \fint_{\lambda B} g^q\,d\mu \biggr)^{1/q}
\end{equation}
holds true for all balls $B \subset X$.

A non-negative function $g \in L^0(X)$ will be called a \emph{$q$-PI gradient} of $u \in L^1_\loc(X)$, if the couple $(u,g)$ satisfies \eqref{eq:def-P1pq} for some fixed $\lambda \ge 1$.

It was observed already in \cite{Haj96} that $M^{1,p}(X) \subset P^{1,p}_1(X)$ and there is an absolute constant $c$ such that $cg$ is a $1$-PI gradient of $u\in M^{1,p}(X)$ with $\lambda = 1$, whenever $g$ is a Haj\l asz gradient of $u$. In fact, it follows from~\cite[Theorems~8.7 and~9.3]{Haj} that if $\mu$ is doubling, then there is $q \in (0,1)$ such that $M^{1,p}(X) = P^{1,p}_q(X)$ with a universal dilation factor $\lambda \ge 1$ for every $p\ge 1$. Corollary~\ref{cor:Hajlasz-infPI} below shows that Haj\l asz functions can be also characterized by an infimal Poincar\'e inequality \eqref{eq:inf-PI}.

The interested reader can refer to \cite{Haj,HajKos}, where functions with a Poincar\'e inequality have been thoroughly studied.
\begin{rem}
If $u \in P^{1,p}_q(\Om)$ for some domain $\Om \subset X$ such that $\mu\lfloor_\Om$ is doubling, then \eqref{eq:def-P1pq} holds true not only for balls with center in $\Om$, but also for balls whose center lies at $\dOm$ possibly at a cost of multiplying the $q$-PI gradient $g$ by a constant factor. This can be shown by the dominated convergence theorem as follows. If $B = B(z, r)$ with $z \in \dOm$, then there is a sequence of points $\{x_n\}_{n=1}^\infty \subset \Om$ such that $\dd(x_n, z) < 2^{-n} r$. Then, $\chi_{B_n} \to \chi_B$ a.e.\@ in $\Om$, where $B_n = B(x_n, (1-2^{-n})r)$. The doubling condition ensures that $\mu(B_n \cap \Om) \approx \mu(B \cap \Om)$ with constants independent of $n$. Thus,
\begin{align*}
  \fint_{B\cap \Om} |u& -u_{B\cap \Om}|\,d\mu \approx \fint_{B\cap \Om}\fint_{B\cap \Om} |u(x)-u(y)|\,d\mu(x)d\mu(y) \\
  & \approx \lim_{n\to\infty} \frac{\mu(B_n \cap \Om)^2}{\mu(B \cap \Om)^2} \fint_{B_n\cap \Om}\fint_{B_n\cap \Om} |u(x)-u(y)|\,d\mu(x)d\mu(y) \\
  & \lesssim \lim_{n\to\infty} (1-2^{-n}) r \biggl(\fint_{\lambda B_n \cap \Om} g^q\,d\mu\biggr)^{1/q} \approx r \biggl(\fint_{\lambda B \cap \Om} g^q\,d\mu\biggr)^{1/q}\,.
\end{align*}
\end{rem}
\begin{df}
\label{df:PI}
We say that a metric measure space $(X, \dd, \mu)$ admits a \emph{$q$-Poincar\'e inequality} if there is a constant $C_\PI > 0$ and a universal dilation factor $\lambda \ge 1$ such that \eqref{eq:def-P1pq} holds true for the couple of functions $(u, C_\PI g)$ whenever $g\in L^0(X)$ is an upper gradient of $u\in L^1_\loc(X)$.
\end{df}
In particular, if $X$ admits a $p$-Poincar\'e inequality, then $N^{1,p}(X) \subset P^{1,p}_p(X)$. It was shown in Keith--Zhong~\cite{KeiZho} that spaces that admit $p$-Poincar\'e inequality, where $p>1$, undergo a self-improvement so that they, in fact, admit a $(p-\eps)$-Poincar\'e inequality for some $\eps>0$ provided that $X$ is complete and $\mu$ doubling. Moreover, one still obtains that $N^{1,p}(X) \subset P^{1,p}_{p-\eps}(X)$ also in case $X$ is merely locally complete. In view of Corollary~\ref{cor:P1pq-equality} below, we therefore see that $N^{1,p}(\Om) \subset P^{1,p}_1(\Om)$ for every $p\ge 1$ whenever $\Om \subset X$ is a domain with a $p$-Poincar\'e inequality, $\mu\lfloor_{\Om}$ is doubling, and $\Om$ as a metric space is complete.
\begin{lem}
\label{lem:inf-PI}
Assume that $\mu$ doubling. Let $0< q < p<\infty$. Then, $u \in P^{1,p}_{q}(X)$ if and only if there is a non-negative function $h \in L^p(X)$ such that
\begin{equation}
  \label{eq:inf-PI}
  \fint_{B} |u-u_B|\,d\mu \le \rad(B) \essinf_{x\in B} h(x)
\end{equation}
for every ball $B \subset X$.
\end{lem}
\begin{proof}
Obviously, \eqref{eq:inf-PI} implies \eqref{eq:def-P1pq} since $\essinf_B h = \bigl(\essinf_B h^q\bigr)^{1/q} \le \bigl(\fint_B h^q\,d\nu\bigr)^{1/q}$.

Let now $u \in P^{1,p}_q(X)$ with a $q$-PI gradient $g\in L^p(X)$ for some $\lambda \ge 1$. For every ball $B \subset X$ and every point $x \in B$, we have $h(x) \coloneq  (M g^q)^{1/q}(x) \ge (\fint_{\lambda B} g^q\,d\mu)^{1/q}$, where $M$ is the non-centered Hardy--Littlewood maximal operator. Thus,
\[
  \fint_B |u-u_B|\,d\mu \le \rad(B) \inf_{x\in B} h(x).
\]
As $\mu$ is doubling, $M: L^1(X) \to \text{weak-}L^1(X)$ is bounded due to Coiffman--Weiss~\cite[Theorem~III.2.1]{CoiWei}. Obviously, $M: L^\infty(X) \to L^\infty(X)$ is bounded. The Marcinkiewicz interpolation theorem then yields that $M: L^{p/q}(X) \to L^{p/q}(X)$. Therefore, $\| h \|_{L^p(X)}=\|(Mg^q)^{1/q}\|_{L^p(X)} \le c_{p/q} \|g\|_{L^p(X)}$. 
\end{proof}
One can deduce from the proof above that for every $u \in P^{1,p}_p(\Om)$, there is $h \in \wk L^p(\Om)$ such that the couple $(u,h)$ satisfies the infimal Poincar\'e inequality \eqref{eq:inf-PI}.
\begin{cor}
\label{cor:P1pq-equality}
Assume that $\mu$ is doubling. Let $0<r \le q < p$. Then, $P^{1,p}_{r}(X) = P^{1,p}_{q}(X) \subset P^{1,p}_{p}(X)$. Moreover, if $u\in P^{1,p}_q(X)$, then there is $h \in L^p(X)$ such that the couple $(u, h)$ satisfies \eqref{eq:def-P1pq} with $\lambda = 1$.
\end{cor}
\begin{proof}
The inclusion $P^{1,p}_{q}(X) \subset P^{1,p}_{p}(X)$ follows by the H\"older inequality applied to the right-hand side of \eqref{eq:def-P1pq}. Equality of $P^{1,p}_q(X)$ and $P^{1,p}_r(X)$ and existence of $h \in L^p(X)$ as desired are an immediate consequence of Lemma~\ref{lem:inf-PI}.
\end{proof}
\begin{cor}
\label{cor:Hajlasz-infPI}
Assume that $\mu$ is doubling. Then, $M^{1,p}(X) = P^{1,p}_q(X)$ for every $0<q<p$, where $1\le p < \infty$. Moreover, each $u \in M^{1,p}(X)$ has a Haj\l asz gradient $h\in L^p(X)$ such that $(u, h)$ satisfies the infimal Poincar\'e inequality \eqref{eq:inf-PI}.
\end{cor}
\begin{proof}
Follows from \cite[Theorem~9.3]{Haj}, Lemma~\ref{lem:inf-PI}, and Corollary~\ref{cor:P1pq-equality}.
\end{proof}
\section{Besov spaces in the metric setting and their embeddings}
\label{sec:besov-emb}
In this section, we will turn our attention to spaces of fractional smoothness, in particular, Besov spaces in the metric setting. Besov spaces have been thoroughly studied in the Euclidean setting (see \cite{Bes,JonWalBk,Tri2,Tri14} and references therein).
They made their first appearance in the metric setting in~\cite{BouPaj} and were explored further in~\cite{GogKosSha,GoKoZh,Han,HeiIhnTuo,KoYaZh,MulYan,Sot}.

The main goal of this section is to investigate continuous and compact embeddings of Besov spaces with norm based on generalization of modulus of continuity as in \cite{GogKosSha} under the very mild assumption that the underlying Borel regular measure is non-trivial and doubling. The results obtained in this section have been proven in some of the papers listed above under more restrictive hypothesis that the measure is Ahlfors regular. An extremely technical approach via frames was employed in \cite{Han,MulYan} to show the embeddings if the measure satisfies a reverse doubling condition.

The great advantage of using the Besov-type norm introduced in \cite{GogKosSha}, see \eqref{eq:Besov} below, is that it allows for very elementary proofs, where the main tools are the ordinary H\"older inequality and Sobolev-type embeddings for Haj\l asz functions. For our arguments, we will utilize also the zero-smoothness Besov spaces that have been studied only scarcely so far, cf.~\cite{MalShaSni}.

Eventually, we will make use of Besov spaces over $\dOm$ for a given domain $\Om \subset X$. In order to avoid confusion with the ambient metric space $X$ that is used in the rest of this paper, we will state the definitions and results in this section for a general metric space $Z = (Z, \dd, \nu)$, where $\nu$ is a doubling measure. By \cite[Lemma~3.3]{BjoBjo}, there are $Q>0$ and $c_Q > 0$ such that
\begin{equation}
  \label{eq:nu-lower-dimension}
   \frac{\nu(B(w,r))}{\nu(B(z,R))} \ge c_Q \biggl(\frac{r}{R} \biggr)^Q
\end{equation}
for all $0<r\le R <\infty$, $z\in Z$, and $w\in B(z,R)$. In particular, if $Z$ is bounded, then $\nu$ has a lower mass bound, i.e., there is $\tilde{c}_Q > 0$ such that
\begin{equation}
  \label{eq:nu-reg-lower-dimension}
   \nu(B(z,r)) \ge \tilde{c}_Q r^Q
\end{equation}
for every $z\in Z$ and $0<r < 2\diam Z$. In fact, one can choose $\tilde{c}_Q = c_Q \nu(Z)/(\diam Z)^Q$.

If $Z$ is connected, then it also satisfies the \emph{reverse doubling condition} by \cite[Corollary~3.8]{BjoBjo}, i.e., there are $\sigma > 0$ and $c_\sigma > 0$ such that
\begin{equation}
  \label{eq:nu-upper-dimension}
   \frac{\nu(B(y,r))}{\nu(B(z,R))} \le c_\sigma \biggl(\frac{r}{R} \biggr)^\sigma
\end{equation}
for all $0<r\le R \le 2 \diam Z$, $z\in Z$, and $y\in B(z,R)$. Unless explicitly stated otherwise, we will assume neither \eqref{eq:nu-reg-lower-dimension}, nor \eqref{eq:nu-upper-dimension} in this section.
\begin{df}
Fix $R>0$ and let $\alpha \in [0,1]$, $p\in [1, \infty)$, and $q \in (0, \infty]$. Then, the \emph{Besov space} $B_{p,q}^\alpha(Z)$ of smoothness $\alpha$ consists of $L^p$-functions of finite Besov (quasi)norm that is given by
\begin{equation}
  \label{eq:Besov}
  \|u\|_{B_{p,q}^\alpha(Z)} = \|u\|_{L^p(Z)} + \Biggl(\int_0^R \biggl(\int_Z \fint_{B(y, t)} \frac{|u(y) - u(z)|^p}{t^{\alpha p}}\,d\nu(z) \,d\nu(y)\biggr)^{q/p} \frac{dt}{t}\Biggr)^{1/q},
\end{equation}
with a standard modification in case $q=\infty$.
\end{df}
The function class $B^\alpha_{p,q}(Z)$ is in fact independent of the exact value of $R \in (0, \infty)$. Moreover, if $\alpha>0$, then we may also choose $R=\infty$ without changing the function class $B^\alpha_{p,q}(Z)$. However, $B^0_{p,q}(Z)$ consists only of (equivalence classes of) constant functions in case $R=\infty$, $q<\infty$ and $Z$ is bounded. Note also that $B^0_{p,\infty}(Z) = L^p(Z)$ regardless of the value of $R \in (0, \infty]$.

Considering the other extreme smoothness value, one can see that $B^1_{p,q}(Z)$ consists of constant functions if $q<\infty$. On the other hand, $M^{1,p}(Z) \subset B^1_{p,\infty}(Z) \subset KS^{1,p}(Z)$, where $M^{1,p}(Z)$ is the Haj\l{}asz space (see Paragraph~\ref{subsec:hajlasz} above) and $KS^{1,p}(Z)$ is the Korevaar--Schoen space introduced in~\cite{KorSch}. See Lemma~\ref{lem:HajlaszIsBesov} below for the former embedding, while the latter follows directly from the definition of $KS^{1,p}$, cf.\@ \cite{GogKosSha}.

If $Z$ supports a $p$-Poincar\'e inequality (see Definition~\ref{df:PI} above) for some $p>1$, then $M^{1,p}(Z) = B^1_{p,\infty}(Z) = KS^{1,p}(Z) = N^{1,p}(Z)$ by~\cite{KeiZho,KosMac}, where $N^{1,p}(Z)$ is the Newtonian space (see Paragraph~\ref{subsec:newtonian} above). For more information on the Haj\l{}asz, Korevaar--Schoen, and Newtonian spaces, see~\cite{BjoBjo,GogKosSha,Haj,HKST}.

The corresponding homogeneous seminorm will be denoted by $\dot{B}^\alpha_{p,q}(Z)$. The value of $q$ can be understood as a fine-tuning parameter of the smoothness for Besov functions of equal value of $\alpha \in [0,1)$.
By the Fubini theorem, the Besov norm defined by \eqref{eq:Besov} with $p=q$ and $R>0$ is equivalent to the Besov-type norm considered by Bourdon--Pajot~\cite{BouPaj}, given by
\begin{equation}
  \label{eq:BesovB}
  \|u \|_{\Bcal_p^\alpha(Z)} = \|u\|_{L^p(Z)} + \biggl(\int_Z \int_{B(w, R)} \frac{|u(w) - u(z)|^p}{\dd(w,z)^{\alpha p} \nu(B(w, \dd(w,z)))}\,d\nu(z) \,d\nu(w) \biggr)^{1/p}
\end{equation}
whenever $\alpha \in [0, 1)$ and $p \in [1, \infty)$, cf.\@ \cite[Theorem~5.2]{GogKosSha}. The setting of \cite{GogKosSha} has several additional standing assumptions, which however turn out to be superfluous for the particular result of equivalence of norms.
The corresponding homogeneous seminorm will be denoted by $\dot{\Bcal}_p^\alpha(Z)$.

The Bourdon--Pajot form of the Besov norm \eqref{eq:BesovB} corresponds very naturally to the Sobolev--Slobodeckij norm used in the classical trace theorems of Gagliardo~\cite{Gag} in the Euclidean setting. In fact, it allows for a cleaner exposition of proofs of trace theorems in John domains and hence will be used in Section~\ref{sec:tracesJohn}. Otherwise, we will be using the Gogatishvili--Koskela--Shan\-mu\-ga\-lin\-gam form of the Besov norm \eqref{eq:Besov}.

For the sake of brevity, we define 
\begin{equation}
  \label{eq:Eput-def}
  E_p(u, t) = \bigl(\int_Z \fint_{B(y,t)} |u(y) - u(z)|^p\,d\nu(z)\,d\nu(y)\bigr)^{1/p}, \qquad u\in L^p(Z),\ t\in (0, \infty).
\end{equation}
First, let us establish continuous embedding when varying the second exponent of the Besov norm.
\begin{lem}
\label{lem:increasing-q}
Let $\alpha \in [0,1)$, $p \in [1, \infty)$, and $q \in (0, \infty)$ be arbitrary. Then, $
  \|u\|_{{B}^\alpha_{p, \tilde{q}}(Z)} \lesssim \|u\|_{{B}^\alpha_{p, q}(Z)}$ for every $\tilde{q} \in [q, \infty]$.
\end{lem}
\begin{proof}
Let us first prove the inequality for $\tilde{q}=\infty$. Then,
\begin{align*}
 &\|u\|_{\dot{B}^\alpha_{p, \infty}(Z)}^q 
= \sup_{0<t<R} \frac{E_p(u, t)^q}{t^{\alpha q}}
\lesssim \sup_{0<t<R} \int_{t/2}^t \frac{ds}{s^{\alpha q+1}} E_p(u,t)^q \lesssim \sup_{0<t<R} \int_{t/2}^t \frac{E_p(u,2s)^q}{s^{\alpha q+1}}\,ds \\
& \qquad \lesssim \int_{0}^{2R} \frac{E_p(u,s)^q}{s^{\alpha q+1}}\,ds = \int_{0}^{R} \ldots\,ds + \int_{R}^{2R} \ldots\,ds \le \|u\|_{\dot{B}^\alpha_{p, q}(Z)}^q + \frac{C}{R^{\alpha q}} \|u\|_{L^p(Z)}^q \lesssim \|u\|_{{B}^\alpha_{p, q}(Z)}^q\,.
\end{align*}
Hence, $\|u\|_{{B}^\alpha_{p, \infty}(Z)} \lesssim \|u\|_{{B}^\alpha_{p, q}(Z)}$.

Given $\tilde{q} \in (q, \infty)$, we can estimate
\begin{align*}
  & \|u\|_{\dot{B}^\alpha_{p, \tilde{q}}(Z)}^{\tilde q} = \int_0^R \frac{E_p(u,t)^{\tilde q}}{t^{\alpha \tilde{q}+1}}\,dt = \int_0^R \frac{E_p(u,t)^{q}}{t^{\alpha q+1}} \cdot \biggl(\frac{E_p(u,t)}{t^{\alpha}}\biggr)^{\tilde{q}-q}\,dt \\
  & \qquad \le \biggl(\sup_{0<s\le R} \frac{E_p(u,s)}{s^{\alpha}}\biggr)^{\tilde{q}-q} \int_0^R \frac{E_p(u,t)^{q}}{t^{\alpha q+1}} \,dt \le 
 \|u\|_{\dot{B}^\alpha_{p, \infty}(Z)}^{\tilde{q}-q}\cdot \|u\|_{\dot{B}^\alpha_{p, q}(Z)}^{q} \lesssim \|u\|_{{B}^\alpha_{p, q}(Z)}^{\tilde{q}-q} \cdot \|u\|_{\dot{B}^\alpha_{p, q}(Z)}^{q}.
\end{align*}
In particular, $\|u\|_{\dot{B}^\alpha_{p, \tilde{q}}(Z)} \lesssim \|u\|_{{B}^\alpha_{p, q}(Z)}$ and hence $\|u\|_{{B}^\alpha_{p, \tilde{q}}(Z)} \lesssim \|u\|_{{B}^\alpha_{p, q}(Z)}$.
\end{proof}
A careful inspection of the proof above immediately yields that $\|u\|_{\dot{B}^\alpha_{p, \tilde{q}}(Z)} \le C \|u\|_{\dot{B}^\alpha_{p, q}(Z)}$ with $C>0$ independent of $\tilde{q} \in [q, \infty]$ holds true whenever $R > \diam Z$ or $R=\diam Z = \infty$. 

Next, we will look into simple interpolating properties of the Besov spaces.
\begin{lem}
\label{lem:BesovInterpolate}
Let $\alpha_j \in [0,1]$, $p_j \in [1, \infty)$, and $q_j \in (0, \infty]$, $j=0,1$. Let $\lambda\in [0,1]$ and define $\alpha$, $p$, and $q$ by the convex combinations
\[
  \alpha = (1-\lambda)\alpha_0 + \lambda \alpha_1, \qquad
  \frac{1}{p} = \frac{1-\lambda}{p_0} + \frac{\lambda}{p_1}, \qquad \text{and}\qquad
  \frac{1}{q} = \frac{1-\lambda}{q_0} + \frac{\lambda}{q_1},
\]
with standard modification if $q_0 = \infty$ or $q_1 = \infty$. Then, $\|u\|_{\dot{B}^\alpha_{p, q}(Z)} \le \|u\|_{\dot{B}^{\alpha_0}_{p_0, q_0}(Z)}^{1-\lambda} \cdot \|u\|_{\dot{B}^{\alpha_1}_{p_1, q_1}(Z)}^\lambda$.
\end{lem}
\begin{proof}
The conclusion holds trivially true when $\lambda = 0$ or $\lambda = 1$. Let us now assume that $\lambda \in (0,1)$. Suppose also that both $q_0$ and $q_1$ are finite. Then, the desired estimate follows from the H\"older inequality applied twice.
\begin{align}
\notag
 \|u\|_{\dot{B}^\alpha_{p, q}(Z)}^q & = \int_0^R \biggl( \int_Z \fint_{B(y,t)} |u(y) - u(z)|^{(1-\lambda)p + \lambda p}\,d\nu(z)\,d\nu(y) \biggr)^{q/p}\,\frac{dt}{t^{\alpha q + 1}} \\
\notag
& \le \int_0^R \biggl( \int_Z \fint_{B(y,t)} |u(y) - u(z)|^{p_0}\,d\nu(z)\,d\nu(y) \biggr)^{(1-\lambda)q/p_0} \cdot \\
\notag
& \phantom{\le \int_0^R} \cdot \biggl( \int_Z \fint_{B(y,t)} |u(y) - u(z)|^{p_1}\,d\nu(z)\,d\nu(y) \biggr)^{\lambda q/p_1}\,\frac{dt}{t^{\alpha q + 1}} \\
\label{eq:BesovInterpolation}
& = I \coloneq \int_0^R \biggl(\frac{E_{p_0}(u,t)}{t^{\alpha_0}}\biggr)^{(1-\lambda) q} \biggl(\frac{E_{p_1}(u,t)}{t^{\alpha_1}}\biggr)^{\lambda q} \frac{dt}{t} \\
\notag
& \le \biggl(\int_0^R \biggl(\frac{E_{p_0}(u,t)}{t^{\alpha_0}}\biggr)^{q_0} \frac{dt}{t}\biggr)^{(1-\lambda)q/q_0} \biggl(\int_0^R \biggl(\frac{E_{p_1}(u,t)}{t^{\alpha_1}}\biggr)^{q_1} \frac{dt}{t}\biggr)^{(1-\lambda)q/q_1} \\
\notag
& = \|u\|_{\dot{B}^{\alpha_0}_{p_0, q_0}(Z)}^{(1-\lambda)q} \cdot \|u\|_{\dot{B}^{\alpha_1}_{p_1, q_1}(Z)}^{\lambda q}\,.
\end{align}
If $q_1 < q_0 = \infty$, then $\lambda q = q_1$ and we would proceed from \eqref{eq:BesovInterpolation} as follows:
\[
I \le  \sup_{0<t<R} \biggl(\frac{E_{p_0}(u,t)}{t^{\alpha_0}} \biggr)^{(1-\lambda) q} \int_0^R \biggl(\frac{E_{p_1}(u,t)}{t^{\alpha_1}}\biggr)^{q_1} \frac{dt}{t} = \|u\|_{\dot{B}^{\alpha_0}_{p_0, q_0}(Z)}^{(1-\lambda)q} \cdot \|u\|_{\dot{B}^{\alpha_1}_{p_1, q_1}(Z)}^{\lambda q}.
\]
A similar argument can be used for $q_0 < q_1 = \infty$.

Finally, if $q=q_0=q_1=\infty$, then the H\"older inequality yields that
\begin{align*}
  \|u\|_{\dot{B}^\alpha_{p, \infty}(Z)} & = \sup_{0<t<R} \frac{E_{(1-\lambda)p + \lambda p}(u,t)}{t^{\alpha}} \le \sup_{0<t<R} \biggl(\frac{E_{p_0}(u,t)^{1-\lambda} }{t^{(1-\lambda)\alpha_0}}\cdot \frac{E_{p_1}(u,t)^{\lambda}}{t^{\lambda \alpha_1}}\biggr) \\
  & \le \biggl( \sup_{0<t<R} \frac{E_{p_0}(u,t)}{t^{\alpha_0}}\biggr)^{1-\lambda} \cdot \biggl( \sup_{0<t<R} \frac{E_{p_1}(u,t)}{t^{\alpha_1}}\biggr)^{\lambda}
  \le \|u\|_{\dot{B}^{\alpha_0}_{p_0, \infty}(Z)}^{1-\lambda} \cdot \|u\|_{\dot{B}^{\alpha_1}_{p_1, \infty}(Z)}^{\lambda}\,.
  \qedhere
\end{align*}
\end{proof}
\begin{cor}
\label{cor:BesovInterpolate}
Given all the parameters as in Lemma~\ref{lem:BesovInterpolate}, $\|u\|_{{B}^\alpha_{p, q}(Z)} \le \|u\|_{{B}^{\alpha_0}_{p_0, q_0}(Z)}^{1-\lambda}  \|u\|_{{B}^{\alpha_1}_{p_1, q_1}(Z)}^\lambda$\,.
\end{cor}
\begin{proof}
By the H\"older inequality, $\|u\|_{L^p(Z)} \le \|u\|_{L^{p_0}(Z)}^{1-\lambda} \|u\|_{L^{p_1}(Z)}^{\lambda}$.
Thus,
\begin{align*}
 \|u\|_{{B}^\alpha_{p, q}(Z)} & \le \|u\|_{L^{p_0}(Z)}^{1-\lambda} \|u\|_{L^{p_1}(Z)}^{\lambda} + \|u\|_{\dot{B}^{\alpha_0}_{p_0, q_0}(Z)}^{1-\lambda} \|u\|_{\dot{B}^{\alpha_1}_{p_1, q_1}(Z)}^\lambda \\
& \le \bigl(\|u\|_{L^{p_0}(Z)} + \|u\|_{\dot{B}^{\alpha_0}_{p_0, q_0}(Z)}\bigr)^{1-\lambda} \bigl(\|u\|_{L^{p_1}(Z)} + \|u\|_{\dot{B}^{\alpha_1}_{p_1, q_1}(Z)}\bigr)^{\lambda} = \|u\|_{{B}^{\alpha_0}_{p_0, q_0}(Z)}^{1-\lambda}  \|u\|_{{B}^{\alpha_1}_{p_1, q_1}(Z)}^\lambda
\end{align*}
by the elementary inequality $x^{1-\lambda} y^\lambda + z^{1-\lambda} w^\lambda \le (x+z)^{1-\lambda} (y+w)^\lambda$, which holds true for every quadruplet of numbers $x,y,z,w\ge 0$ and $\lambda \in (0,1)$. Indeed, by concavity of $t \mapsto t^\lambda$, we have that
\begin{align*}
  x^{1-\lambda} y^\lambda + z^{1-\lambda} w^\lambda
  = x \Bigl(\frac yx \Bigr)^\lambda +
    z \Bigl(\frac wz \Bigr)^\lambda
  & = (x+z) \biggl( \frac{x}{x+z} \Bigl(\frac yx \Bigr)^\lambda + \frac{z}{x+z} \Bigl(\frac wz \Bigr)^\lambda \biggr) \\
  & \le (x+z) \biggl( \frac{xy}{(x+z)x} + \frac{zw}{(x+z)z} \biggr)^\lambda = (x+z)^{1-\lambda} (y+w)^\lambda
\end{align*}
whenever $xz > 0$. If $xz=0$, then the elementary inequality is trivial.
\end{proof}
\begin{lem}
\label{lem:BesovLinftyInterpolate}
Assume that $\alpha \in [0,1]$, $p \in [1, \infty)$, and $q \in (0, \infty]$. Let $\lambda \in (0,1)$. Then,
\[
  \|u\|_{\dot{B}^{\lambda\alpha}_{p/\lambda, q/\lambda}(Z)} \le 2 \|u\|_{L^\infty(Z)}^{1-\lambda} \cdot \|u\|_{\dot{B}^{\alpha}_{p, q}(Z)}^\lambda\,.
\]
\end{lem}
\begin{proof}
Factoring out $\esssup_{y,z} |u(y)-u(z)|^{p(1/\lambda - 1)}$ from the innermost integral yields that
\begin{align*}
 \|u\|_{\dot{B}^{\lambda\alpha}_{p/\lambda, q/\lambda}(Z)}^{q/\lambda} & = \int_0^R \biggl( \int_Z \fint_{B(y,t)} |u(y) - u(z)|^{p/\lambda}\,d\nu(z)\,d\nu(y) \biggr)^{q/p}\,\frac{dt}{t^{\alpha q + 1}} \\
& \le 2^{q/\lambda} \|u\|_{L^\infty(Z)}^{(1-\lambda) q/\lambda} \int_0^R \biggl( \int_Z \fint_{B(y,t)} |u(y) - u(z)|^{p}\,d\nu(z)\,d\nu(y) \biggr)^{q/p}\,\frac{dt}{t^{\alpha q + 1}} 
\\
& = 2^{q/\lambda} \|u\|_{L^\infty(Z)}^{(1-\lambda) q/\lambda} \|u\|_{\dot{B}^{\alpha}_{p, q}(Z)}^q\,.
\qedhere
\end{align*}
\end{proof}
Note that similarly as in Corollary~\ref{cor:BesovInterpolate}, one obtains $\|u\|_{{B}^{\lambda\alpha}_{p/\lambda, q/\lambda}(Z)} \le 2 \|u\|_{L^\infty(Z)}^{1-\lambda} \cdot \|u\|_{{B}^{\alpha}_{p, q}(Z)}^\lambda$.
\begin{lem}
\label{lem:decreasingsmoothness}
Assume that $R<\infty$. Let $0\le \beta<\alpha \le 1$ and suppose that $p\in [1,\infty)$, and $q \in (0, \infty]$. Then, $\|u\|_{\dot{B}^\beta_{p,q}(Z)} \lesssim \|u\|_{\dot{B}^\alpha_{p,\infty}(Z)}$.
\end{lem}
\begin{proof}
If $q< \infty$, then
\[
  \|u\|_{\dot{B}^\beta_{p,q}(Z)}^q = \int_0^R E_p(u,t)^q \frac{dt}{t^{\beta q + 1}} = \int_0^R \biggl( \frac{E_p(u,t)}{t^\alpha} \biggr)^q \frac{dt}{t^{1-(\alpha-\beta)q}} \le \|u\|_{\dot{B}^\alpha_{p,\infty}(Z)}^q \frac{R^{(\alpha - \beta)q}}{(\alpha-\beta)q}\,.
\]
In case $q=\infty$, it suffices to notice that $E_p(u,t)/t^\beta \le R^{\alpha-\beta} E_p(u,t)/t^\alpha$ for every $0<t\le R$.
\end{proof}
\begin{rem}
\label{rem:decreasingsmoothness}
If $R=\infty$, then one can further use the estimate $E_p(u, t) \le C \|u\|_{L^p(Z)}$, where $C$ does not depend on $t>0$, to obtain the inequality $\|u\|_{{B}^\beta_{p,q}(Z)} \lesssim \|u\|_{{B}^\alpha_{p,\infty}(Z)}$ for every $0 \le \beta<\alpha \le 1$, $p\in [1,\infty)$, and $q \in (0, \infty]$.
\end{rem}
\begin{lem}
\label{lem:zerosmoothness}
Suppose that $u \in L^r(Z) \cap \dot{B}^\alpha_{s,\infty}(Z)$ for some $1\le s < r$ and $\alpha > 0$. Then, $u \in \dot{B}^0_{p,q}(Z)$ for every $p\in [s, r)$ and $q \in (0, \infty)$ provided that $R<\infty$ in \eqref{eq:Besov}.
\end{lem}
\begin{proof}
Case $p=s$ follows immediately from Lemma~\ref{lem:decreasingsmoothness}.

Fix $p \in (s,r)$ and $q\in(0, \infty)$. Define $\eta = \frac{s(r-p)}{r-s} \in (0,s)$. Observe also that
\[
  \frac{\nu(B(y,t))}{c_Q} \le \nu(B(z,t)) \le c_Q \,\nu(B(y,t))
\]
whenever $y,z \in Z$ with $\dd(y,z)< t$, where $t>0$ is arbitrary. For the sake of brevity, let us define $Z^2_t = \{(y,z) \in Z^2: \dd(y,z) < t\}$. Then,
\begin{align*}
  \|u\|_{\dot{B}^0_{p,q}(Z)}^q & 
  \le \int_0^R \biggl(\int_Z \fint_{B(y,t)} \bigl(|u(y)| + |u(z)|\bigr)^{p-\eta} |u(y) - u(z)|^\eta\,d\nu(z)\,d\nu(y)\biggr)^{q/p}\,\frac{dt}{t} \\
  & \le 2^{q(p-\eta)/p} \int_0^R \biggl(\int_Z \fint_{B(y,t)} \bigl(|u(y)|^{p-\eta} + |u(z)|^{p-\eta}\bigr) |u(y) - u(z)|^\eta\,d\nu(z)\,d\nu(y)\biggr)^{q/p}\,\frac{dt}{t} \\
 & \approx \int_0^R \biggl( \iint_{Z^2_t}  \frac{|u(y)|^{p-\eta}|u(y) - u(z)|^\eta}{\nu(B(y,t))}+\frac{|u(z)|^{p-\eta}|u(y) - u(z)|^\eta}{\nu(B(z,t))}\,d\nu^2(y,z) \biggr)^{q/p} \frac{dt}{t}\\
& \approx \int_0^R \biggl( \iint_{Z^2_t} \frac{|u(y)|^{p-\eta} |u(y) - u(z)|^\eta}{\nu(B(y,t))} \,d\nu^2(y,z) \biggr)^{q/p} \frac{dt}{t} \eqcolon I\,.
\end{align*}
Next, we apply the H\"older inequality for the inner double integral, which yields
\begin{align*}
  I & \le \int_0^R \biggl( \iint_{Z^2_t} \frac{|u(y)|^r}{\nu(B(y,t))}\,d\nu^2(y,z)\biggr)^{q(p-\eta)/pr} \\
  & \qquad \times \biggl( \iint_{Z^2_t} \frac{|u(y)-u(z)|^{\eta r/(r-p+\eta)}}{\nu(B(y,t))}\,d\nu^2(y,z)\biggr)^{q(r-p+\eta)/pr} \frac{dt}{t} \\
  & = \int_0^R \biggl(\int_Z |u(y)|^r \,d\nu(y)\biggr)^{q(p-s)/p(r-s)} \biggl(\int_Z \fint_{B(y,t)} |u(y)-u(z)|^s\,d\nu(y)\,d\nu(z) \biggr)^{q(r-p)/p(r-s)}\frac{dt}{t} \\
  & \le \|u\|_{L^r(Z)}^{q r(p-s)/p(r-s)} \|u\|_{\dot{B}^\alpha_{1,\infty}(Z)}^{qs(r-p)/p(r-s)} \int_0^R \frac{dt}{t^{1-\beta}} < \infty\,,
\end{align*}
where $\beta = \alpha qs(r-p)/p(r-s) > 0$. Hence, $\|u\|_{\dot{B}^0_{p,q}(Z)} \le C \|u\|_{L^r(Z)}^{r(p-s)/p(r-s)} \|u\|_{\dot{B}^\alpha_{s,\infty}(Z)}^{s(r-p)/p(r-s)}$.
\end{proof}
The next lemma shows that Besov functions of smoothness $\alpha$ are closely related to $\alpha$-fractional Haj\l asz functions. Using also Lemma~\ref{lem:HajlaszIsBesov} below, we deduce that $B^\alpha_{p,p}(Z) \subset M^{\alpha,p}(Z) \subset B^\alpha_{p,\infty}(Z)$ and the embeddings are continuous. This fact will be used to prove Sobolev-type embeddings for Besov functions.
\begin{lem}
\label{lem:BesovHaj}
Given $u \in \dot{B}^{\alpha}_{p,q}(Z)$ for some $\alpha\in(0,1)$ and $0<q\le p < \infty$, there is $g \in L^p(Z)$ and $E \subset Z$ with $\nu(E)=0$ such that
\begin{equation}
  \label{eq:BesovHaj}
  |u(y) - u(z)| \le \dd(y,z)^\alpha (g(y) + g(z)) \qquad\text{for all }y,z \in Z \setminus E.
\end{equation}
Moreover, $\|g\|_{L^p(Z)} \lesssim \|u\|_{\dot{B}^\alpha_{p,q}(Z)}$ provided that $R > \diam Z$ in \eqref{eq:Besov} (or, perhaps, $R=\diam Z=\infty$).
\end{lem}
\begin{proof}
Similarly as in \cite[Lemma~6.1]{GogKosSha}, setting $g(z) = 2^{Q+\alpha} c_Q^{-1} \sup_{r>0} \fint_{B(z,r)} |u(z) - u(y)|r^{-\alpha} \,d\nu(y)$, $z\in Z$, gives the desired function that satisfies \eqref{eq:BesovHaj}. Then, \cite[Lemma~6.1]{GogKosSha} and Lemma~\ref{lem:increasing-q} yield that $\|g\|_{L^p(Z)} \lesssim \|u\|_{\dot{B}^\alpha_{p,p}(Z)} \lesssim \|u\|_{\dot{B}^\alpha_{p,q}(Z)}$.
\end{proof}
\begin{cor}
\label{cor:Besov-standardembeddings}
Let $\alpha \in (0,1)$, $p \in [1, \infty)$, and $q \in (0,p]$. Let $G\subset Z$ be an arbitrary open ball. Then, there is $C>0$ such that:
\begin{enumerate}
	\item \label{it:BesovEmbIt1} If $\alpha p < Q$, then $\|u-u_G\|_{L^{p^*}(G)} \le C \|u\|_{\dot{B}^\alpha_{p,q}(Z)}$, where $p^* = pQ / (Q-\alpha p)$.
  \item \label{it:BesovEmbIt2} If $\alpha p = Q$, then $\|u-u_G\|_{\exp L(G)} \le C \|u\|_{\dot{B}^\alpha_{p,q}(Z)}$.
  \item \label{it:BesovEmbIt3} If $\alpha p > Q$, then $\|u-u_G\|_{L^{\infty}(G)} \le C \|u\|_{\dot{B}^\alpha_{p,q}(Z)}$ and $|u(y) - u(z)| \le C \dd(y,z)^{\kappa} \|u\|_{B^\alpha_{p,q}(Z)}$ for all $y,z \in G$, where $\kappa = \alpha - Q/p$.
\end{enumerate}
If \eqref{eq:nu-reg-lower-dimension} is satisfied, then we may choose $G = Z$ in all of the above even if $Z$ is unbounded, in which case $u_G = 0$.
\end{cor}
\begin{proof}
Given $u \in \dot{B}^\alpha_{p,q}(Z)$, there is $g\in L^p(Z)$ such that \eqref{eq:BesovHaj} is satisfied.

Fix a ball $G\subset Z$. Considering the snow-flaked metric $\tilde{d}(\cdot, \cdot) = \dd(\cdot, \cdot)^\alpha$, we obtain from \eqref{eq:nu-lower-dimension} that the metric space $(G, \tilde{d}, \nu\lfloor_{G})$ has lower mass bound with exponent $Q/\alpha$, i.e., $\nu(B_{\tilde{d}}(z,r)) \ge \tilde{c} r^{Q/\alpha}$ for every $z\in G$ and $r\le \diam_{\tilde d} G$, where $B_{\tilde{d}}(z,r) = \{y \in G: \tilde{d}(y,z)<r\}$ and $\tilde{c} = c_Q \nu(G)/\rad_{\dd}(G)^Q$.

Then, all three parts follow from \cite[Theorem~8.7]{Haj}, which would give $\|g\|_{L^p(Z)} \lesssim \|u\|_{\dot{B}^\alpha_{p,q}(Z)}$ on the right hand side of the inequality in \ref{it:BesovEmbIt1} and \ref{it:BesovEmbIt2}. As for \ref{it:BesovEmbIt3}, we obtain that
\[
  |u(y)-u(z)| \le C \tilde{d}(y,z)^{1-Q/\alpha p} \|g\|_{L^p(Z)} \le C \dd(y,z)^{\alpha - Q/p} \|u\|_{\dot{B}^\alpha_{p,q}(Z)}\quad \text{for every }y,z\in G\,.
\]
By inspecting the proof of \cite[Theorem~8.7]{Haj}, one sees that the constant $C$ really depends only on $p$, $Q/\alpha$, and $\tilde{c}$ in the lower mass bound condition of $G$. Hence, if \eqref{eq:nu-reg-lower-dimension} holds, then $\nu(B) \ge c_Q \rad_{\tilde{d}}(B)^{Q/\alpha}$ for every ball $B \subset Z$ and hence we may apply \cite[Theorem~8.7]{Haj} directly to the snow-flaked space $(Z, \tilde{d}, \nu)$.
\end{proof}
\begin{rem}
\label{rem:sharp-besov-emb}
The embedding as stated in Corollary~\ref{cor:Besov-standardembeddings}\,\ref{it:BesovEmbIt1} is not sharp in case both \eqref{eq:nu-reg-lower-dimension} and \eqref{eq:nu-upper-dimension} are satisfied. In fact, one can deduce from \cite[Theorem~4.4]{HeiIhnTuo} that
$\|u\|_{L^{p^*,q}(Z)} \lesssim \|u\|_{B^{\alpha}_{p,q}(Z)}$ for all $p\in[1, Q/\alpha)$ and $q \in [1, \infty]$, where $L^{p^*,q}(Z)$ stands for a Lorentz space (see~\cite[Section~IV.4]{BenSha}).
Observe however that \cite{HeiIhnTuo} assumes that $\nu$ is Ahlfors $Q$-regular, which it fails to be in case $\sigma < Q$. Nevertheless, 
the problem will be circumvented as soon as \cite[Lemma~3.4]{HeiIhnTuo} in the proof of \cite[Theorem~4.4]{HeiIhnTuo} is replaced by \cite[Theorem~8.7\,(1)]{Haj}.
\end{rem}
\begin{lem}[{cf.\@ \cite[Lemma~6.2]{GogKosSha}}]
\label{lem:HajlaszIsBesov}
Let $p\in[1, \infty)$ and $\alpha\in (0,1]$. Suppose that the couple of measurable functions $u, g$ with $g \in L^p(Z)$ satisfy \eqref{eq:BesovHaj}. Then, $\|u\|_{\dot{B}^\alpha_{p,\infty}(Z)} \lesssim \|g\|_{L^p(Z)}$. Consequently, $M^{\alpha,p}(Z) \emb B^{\alpha}_{p,\infty}(Z)$.
\end{lem}
\begin{proof}
Plugging in \eqref{eq:BesovHaj} into the definition of $\|u\|_{\dot{B}^\alpha_{p,\infty}}$ yields that
\begin{align*}
\|u\|_{\dot{B}^\alpha_{p,\infty}}^p & = \sup_{0<t<R}  \int_Z \fint_{B(y,t)} \frac{|u(y) - u(z)|^p}{t^{\alpha p}}\,d\nu(z)\,d\nu(y) 
\\
& \le \sup_{0<t<R}  \int_Z \fint_{B(y,t)} \frac{\dd(y,z)^{\alpha p}}{t^{\alpha p}} \bigl(g(y) + g(z)\bigr)^p\,d\nu(z)\,d\nu(y) \\
& \le 2^p \sup_{0<t<R} \int_Z \fint_{B(y,t)} \bigl(g(y)^p + g(z)^p\bigr)\,d\nu(z)\,d\nu(y) \approx \|g\|_{L^p(Z)}^p\,.
\qedhere
\end{align*}
\end{proof}
We are now ready to prove one of the main theorems of this section, dealing with continuous embeddings between Besov spaces over a measure space with lower mass bound. Recall also that the doubling condition \eqref{eq:nu-lower-dimension} implies the lower mass bound \eqref{eq:nu-reg-lower-dimension} whenever $Z$ is bounded.
\begin{thm}
\label{thm:Besov-embedding-openended}
Assume that \eqref{eq:nu-reg-lower-dimension} is satisfied. Suppose that $\alpha \in (0,1]$ and $p \in [1, Q/\alpha)$. Let $p^*$ be the \emph{Sobolev conjugate exponent}, i.e., $p^* = pQ / (Q-p\alpha)$. Given $\lambda \in (0,1)$, let $p_\lambda$ be given by the convex combination of reciprocals of $p$ and $p^*$, viz., $p_\lambda^{-1} = (1-\lambda) p^*{}^{-1} + \lambda p^{-1}$. 
\begin{enumerate}
  \item If $q \in (0, p]$, then $\|u\|_{B^{\lambda \alpha}_{p_\lambda, q'}}(Z) \lesssim \|u\|_{B^{\alpha}_{p, q}}(Z)$ for every $q' \in [q/\lambda, \infty]$.%
	\label{it:BesEmbedIt1}
  \item If $q \in (0, \infty]$, then $\|u\|_{B^{\lambda \alpha}_{p', q'}}(Z) \lesssim \|u\|_{B^{\alpha}_{p, q}}(Z)$ for every $p' \in [p, p_\lambda)$ and $q' \in (0, \infty]$.%
    \label{it:BesEmbedIt2}
\end{enumerate}
\end{thm}
Note that by the definition of $p^*$, we have $p_\lambda^{-1} = p^{-1} - (1-\lambda) \alpha Q^{-1} = p^*{}^{-1} + \lambda \alpha Q^{-1}$. Part \ref{it:BesEmbedIt1} holds true even for $\lambda = 1$ by Lemma~\ref{lem:increasing-q}. The embedding in \ref{it:BesEmbedIt2} holds true also for $\lambda = 0$ by Corollary~\ref{cor:BesovInterpolate} and Lemma~\ref{lem:zerosmoothness} (details on how to deal with $q>p$ will become clear from the proof below). If $\nu(Z)<\infty$, then the estimate in \ref{it:BesEmbedIt2} is satisfied also for $p'<p$ since $E_{p'}(u,t) \le \nu(Z)^{1-p'/p} E_{p(u,t)}$ by the H\"older inequality.
\begin{proof}
Assume that $q \le p$. Then, $\|u\|_{L^{p^*}(Z)} \lesssim \|u\|_{B^\alpha_{p,q}(Z)}$ by Corollary~\ref{cor:Besov-standardembeddings}\,\ref{it:BesovEmbIt1}. In particular,
$
 \|u\|_{B^0_{p^*, \infty}(Z)} \approx \|u\|_{L^{p^*}(Z)} \lesssim \|u\|_{B^{\alpha}_{p,q}(Z)}.
$
Hence, Corollary~\ref{cor:BesovInterpolate} yields that
\[
  \|u\|_{B^{\lambda \alpha}_{p_\lambda, q/\lambda}(Z)} \le \|u\|_{B^0_{p^*, \infty}(Z)}^{1-\lambda} \|u\|_{B^\alpha_{p, q}(Z)}^\lambda \lesssim \|u\|_{B^\alpha_{p, q}(Z)}\,.
\]
Proof of \ref{it:BesEmbedIt1} will then be completed by applying Lemma~\ref{lem:increasing-q}.

Let us now focus on \ref{it:BesEmbedIt2}. If $p'=p$, then the desired embedding follows immediately from Lemma~\ref{lem:decreasingsmoothness}. Assume that $p' \in (p, p_\lambda)$. Then, there are $\tilde{\alpha} \in (\lambda\alpha, \alpha)$ and $\tilde{\lambda} \in (\lambda \alpha/\tilde{\alpha}, 1)$ such that
\[
  \frac{1}{p'} = \frac{1-\tilde\lambda}{\tilde{p}^*} + \frac{\tilde\lambda}{p} \coloneq (1-\tilde\lambda) \biggl( \frac{1}{p} - \frac{\tilde{\alpha}}{Q}\biggr) + \frac{\tilde\lambda}{p} = \frac{1}{p} - \frac{(1-\tilde\lambda) \tilde{\alpha}}{Q}
  = \frac{1}{p_\lambda} + \frac{\tilde{\lambda}\tilde{\alpha} - \lambda\alpha}{Q}\,.
\]
Then, $u \in B^{\tilde{\alpha}}_{p,p}(Z)$ with $\|u\|_{B^{\tilde{\alpha}}_{p,p}(Z)} \lesssim \|u\|_{B^\alpha_{p,q}}$ by Lemmata~\ref{lem:increasing-q} and \ref{lem:decreasingsmoothness}. Next, we obtain from \ref{it:BesEmbedIt1} that
$\|u\|_{B^{\tilde\lambda \tilde\alpha}_{p',\infty}(Z)} \lesssim \|u\|_{B^{\tilde{\alpha}}_{p,p}(Z)}$. Finally, Lemma~\ref{lem:decreasingsmoothness} concludes the proof of \ref{it:BesEmbedIt2}.
\end{proof}
\begin{rem}
\label{rem:Besov-embedding-openended}
If $\alpha p = Q$, then the conclusion of Theorem~\ref{thm:Besov-embedding-openended}\,\ref{it:BesEmbedIt2} still holds true with $p^* = \infty$, i.e., $p_\lambda = p/\lambda$ and the proof goes through verbatim. If $\alpha p > Q$, then both parts of Theorem~\ref{thm:Besov-embedding-openended} hold true with $p^* = \infty$, i.e., $p_\lambda = p/\lambda$. One needs to use Corollary~\ref{cor:Besov-standardembeddings}\,\ref{it:BesovEmbIt3} and Lemma~\ref{lem:BesovLinftyInterpolate} in the proof of \ref{it:BesEmbedIt1}, and make sure that $\tilde{\alpha} p \neq Q$ in the proof of \ref{it:BesEmbedIt2}.
\end{rem}
If we assume that $\nu$ is not only doubling, but also satisfies the reverse doubling condition~\eqref{eq:nu-upper-dimension}, then one has a somewhat sharper embedding between Besov spaces in the critical case $p' = p_\lambda$ for all $q\in(0,\infty]$. This has been proven by a very different (and very technical) approach to Besov spaces via frames. Roughly speaking, \eqref{eq:Besov} generalizes the modulus of continuity definition of Besov spaces in $\Rbb^n$, whereas the frames correspond to the Fourier approach to Besov spaces in $\Rbb^n$.
\begin{thm}
\label{thm:Besov-embedding-closedend}
Assume that both \eqref{eq:nu-reg-lower-dimension} and \eqref{eq:nu-upper-dimension} are satisfied. Suppose that $\alpha p < Q$ and define the ``Sobolev conjugate'' exponent $p^* = pQ / (Q-p\alpha)$. If $\lambda \in (0,1)$, then $B^{\alpha}_{p, q}(Z) \subset B^{\lambda \alpha}_{p_\lambda, q}(Z)$, where $p_\lambda$ is given by the convex combination of reciprocals of $p^*$ and $p$, viz., $p_\lambda^{-1} = (1-\lambda) p^*{}^{-1} + \lambda p^{-1}$.
\end{thm}
\begin{proof}
The claim follows from \cite[Theorem~1.5\,(i)]{Han}, \cite[Theorem~4.1]{MulYan}, and Lemma~\ref{lem:increasing-q}.
\end{proof}
\begin{rem}
In the theory of Besov spaces in $\Rbb^n$, it is more common to state the embeddings between spaces of different smoothness in terms of their \emph{differential dimension}, which is defined as $\alpha - \frac{Q}{p}$ for $B^\alpha_{p,q}(Z)$, where $Q>0$ is from \eqref{eq:nu-lower-dimension}, cf.\@ \cite{Tri2,Tri14} and references therein. Embeddings of the form $B^{\alpha_0}_{p_0, q_0}(Z) \emb B^{\alpha_1}_{p_1, q_1}(Z)$, where $\alpha_0 - \frac{Q}{p_0} = \alpha_1 - \frac{Q}{p_1}$ and $\alpha_0 > \alpha_1$ are the objective of Theorems~\ref{thm:Besov-embedding-openended}\,\ref{it:BesEmbedIt1} and~\ref{thm:Besov-embedding-closedend}, whereas Theorem~\ref{thm:Besov-embedding-openended}\,\ref{it:BesEmbedIt2} deals with the case when $\alpha_0 - \frac{Q}{p_0} > \alpha_1 - \frac{Q}{p_1}$ and $\alpha_0 > \alpha_1$.
\end{rem}
Let us now aim our attention to proving compactness of the embeddings. First, we prove a technical lemma that will allow us to run a Rellich--Kondrachev-type argument in the proof of Theorem~\ref{thm:Besov0-cpt-embedding} below.
\begin{lem}
\label{lem:supsum-estimate}
Let $\{a_{j,k}\}_{j,k=1}^\infty \subset \Rbb^+$. Suppose that there is $K>0$ such that $\sup_{k} \sum_{j} a_{j,k} \le K$.
Then, for every $\eps>0$, there are $j_0 \in \Nbb$ and $I \subset \Nbb$ such that $\# I = \infty$ and $\sup_{k\in I} a_{j_0, k} \le \eps$.
\end{lem}
\begin{proof}
For the sake of contradiction, suppose that there is $\eps>0$ such that $I_j \coloneq \{k \in \Nbb: a_{j, k} \le \eps\}$ is finite for every $j\in\Nbb$. Let $N =  \lceil K / \eps \rceil$. Then,
$  E = \bigcap_{j=1}^N (\Nbb \setminus I_j) \neq \emptyset$.
Hence, if $k\in E$ and $1\le j \le N$, then $a_{j,k} > \eps$. Thus,
\[
  \sum_{j=1}^\infty a_{j,k} \ge \sum_{j=1}^N a_{j,k} >  \sum_{j=1}^N \eps = \biggl\lceil \frac{K}{\eps}\biggr\rceil \cdot \eps \ge K\,,
\]
which contradicts the hypothesis that $\sum_j a_{j,k} \le K$ for every $k\in\Nbb$.
\end{proof}
Next, will show the compact embedding of Besov spaces of zero smoothness, which will be later applied to prove compact embeddings for spaces of higher smoothness via continuous embeddings and elementary interpolating properties.
\begin{thm}
\label{thm:Besov0-cpt-embedding}
Assume that $Z$ is bounded. Let $p\in (1, \infty)$ and $q\in (0, \infty)$. Let $\{u_k\}_{k=1}^\infty$ be a sequence of functions with uniformly bounded $B^0_{p,q}(Z)$ norm. Then, there is a subsequence $\{u_{k_j}\}_{j=1}^\infty$ and a function $u \in L^p(Z)$ such that $u_{k_j} \to u$ in $L^p(Z)$.
\end{thm}
\begin{proof}
Let $\{u_k\}_{k=1}^\infty$ be bounded in $B^0_{p,q}(Z)$. Then, it is bounded in $L^p(Z)$, which is reflexive as $p>1$ and hence there is a subsequence (still denoted by $\{u_k\}_{k=1}^\infty$) that is weakly convergent in $L^p(Z)$. 
We will find a subsequence that is Cauchy (and hence convergent) with respect to the $L^p$ norm. 

Given $\delta>0$, let $A_\delta$ be the averaging operator defined for $f\in L^1_\loc(Z)$ by
\[
  A_\delta f(y) = \fint_{B(y,\delta)} f(z)\,d\nu(z), \quad y\in Z.
\]
For a fixed $\delta>0$, the weak convergence of $\{u_k\}_k$ results in the pointwise convergence $A_\delta u_k(y) \to A_\delta u(y)$ for every $y \in Z$ as $k\to \infty$. The lower mass bound~\eqref{eq:nu-reg-lower-dimension}, which is available as $Z$ is bounded and $\nu$ is doubling, implies that $\| A_\delta u_k\|_{L^\infty(Z)} \lesssim \delta^{-Q/p} \|u_k\|_{L^p(Z)} \le C$. The Lebesgue dominated convergence then yields that $A_\delta u_k \to A_\delta u$ in the norm of $L^p(Z)$.

By the triangle and the H\"older inequalities,
\begin{align*}
   \|A_\delta f - f\|_{L^p(Z)}^p = \int_Z |A_\delta f(y) - f(y)|^p\,d\nu(y)
   \le \int_Z  \biggl(\fint_{B(y,\delta)} |f(z) - f(y)| \,d\nu(z)\biggr)^p d\nu(y)\qquad& \\
    \le \int_Z \fint_{B(y,\delta)} |f(z) - f(y)|^p \,d\nu(z)\,d\nu(y) \le 
    E_p(f,\delta)^p &.
\end{align*}
Observe that
\[
  \sum_{j=1}^\infty E_p(f, 2^{-j})^q \approx \sum_{j=1}^\infty E_p(f, 2^{-j})^q \int_{2^{-j}}^{2^{1-j}} \frac{dt}{t} \lesssim \int_0^1 E_p(f, t)^q \frac{dt}{t} \lesssim \|f\|_{B^0_{p,q}(Z)}^q\,.
\]

Letting $a_{j,k} \coloneq E_p(u_k, 2^{-j})^q$, we see that $\sup_k \sum_j a_{j,k} \lesssim \sup_k \|u_k\|_{B^0_{p,q}(Z)}^q \le C$, which allows us to use Lemma~\ref{lem:supsum-estimate}. Thus, for every $n \in \Nbb$, we can find $j_n\ge 1$ and a countably infinite set $I_n \subset \Nbb$ so that $E_p(u_k, 2^{-j_n})<2^{-n}$ for all $k\in I_n$. Moreover, it is possible to ensure that the sequence $\{j_n\}_{n=1}^\infty$ is strictly increasing. Using a diagonal argument, we will now construct the desired subsequence $\{u_{k_m}\}_{m=1}^\infty$. Let $k_1 \in I_1$ be arbitrary. For $m>1$, we pick $k_m \in \bigcap_{n=1}^m I_n$ such that $k_m \ge k_{m-1}$. By this choice, we have
\[
  E_p(u_{k_m}, 2^{-j_n}) < 2^{-n} \qquad\text{whenever $m\ge n \ge 1$.}
\]

Let $\eps > 0$. Then, there is $n_\eps\in \Nbb$ such that $2^{-n_\eps} < \eps$. Let $\delta = 2^{-j_{n_\eps}}$. Since $\{A_\delta u_{k_m}\}_{m=1}^\infty$ is convergent in $L^p(Z)$, there is $m_0\in \Nbb$ such that $\| A_\delta u_{k_m} - A_\delta u_{k_{m'}}\|_{L^p(Z)} \le \eps$ whenever $m,m' \ge m_0$. Thus, for every $m, m' \ge \max\{n_\eps, m_0\}$, we have
\begin{align*}
  \|u_{k_m} - u_{k_{m'}}\|_{L^p(Z)} & \le \|u_{k_m} - A_\delta u_{k_{m}}\|_{L^p(Z)} + \|A_\delta u_{k_{m}} - A_\delta u_{k_{m'}}\|_{L^p(Z)} + \|A_\delta u_{k_{m'}} - u_{k_{m'}}\|_{L^p(Z)} \\
  & \le E_p(u_{k_m}, \delta) + \eps + E_p(u_{k_{m'}}, \delta) < 3\eps.
\end{align*}
Since $\eps>0$ is arbitrary, we have shown that $\{u_{k_m}\}_{m=1}^\infty$ is a Cauchy sequence in $L^p(Z)$ and hence $u_{k_m} \to u$ in $L^p(Z)$ as $m\to\infty$.
\end{proof}
\begin{cor}
\label{cor:Besov-cpt-embedding}
Assume that $Z$ is bounded. Let $p\ge 1$ and $\alpha\in(0,1]$. Let $\{u_k\}_{k=1}^\infty$ be a sequence of functions with uniformly bounded $B^\alpha_{p,\infty}(Z)$ norm. Then, there is a subsequence $\{u_{k_j}\}_{j=1}^\infty$ and a function $u \in L^s(Z)$ such that $u_{k_j} \to u$ in $L^s(Z)$ for every $s \in [1, p^*)$, where $p^* = pQ/(Q-\alpha p)$  if $\alpha p<Q$ and $p^* = \infty$ otherwise. Moreover, $s=\infty$ is also possible if $\alpha p > Q$.
\end{cor}
\begin{proof}
The claim follows from Lemma~\ref{lem:decreasingsmoothness} and Theorems~\ref{thm:Besov-embedding-openended} and~\ref{thm:Besov0-cpt-embedding} (see also Remark~\ref{rem:Besov-embedding-openended}). In view of Corollary~\ref{cor:Besov-standardembeddings}, the assertion in case $s=\infty$ for $\alpha p > Q$ follows from the compact embedding of H\"older continuous functions $\Ccal^{0,\kappa}(Z) \Subset \Ccal(Z) \subset L^\infty(Z)$.
\end{proof}
\begin{cor}
\label{cor:BesovBesov-cpt-embedding}
Under the assumptions of Corollary~\ref{cor:Besov-cpt-embedding}, $u_{k_j} \to u$ in $B^{\lambda \alpha}_{\tilde{p},q}$ for every $\lambda \in [0,1)$, $q \in (0,\infty)$, and $\tilde{p} \in [1, p_\lambda)$, where $p_\lambda^{-1} = (1-\lambda){p^*}^{-1} + \lambda p^{-1}$, where $p^* = pQ/(Q-\alpha p)$  if $\alpha p<Q$ and $p^* = \infty$ otherwise.
\end{cor}
\begin{proof}
Let $\lambda\in[0,1)$ and $\tilde{p}\in[1, p_\lambda)$ be given. Without loss of generality, we may assume that $\tilde{p} > p$ since the H\"older inequality will then guarantee norm convergence for smaller values of $\tilde{p}$. Then, we can find $\tilde\lambda \in (\lambda, 1)$ and $s \in (p, p^*)$ such that
\[
  \frac{1-\tilde{\lambda}}{p^*} + \frac{\tilde\lambda}{p} \eqcolon \frac{1}{p_{\tilde\lambda}} < \frac{1}{\tilde{p}} = \frac{1-\tilde{\lambda}}{s} + \frac{\tilde\lambda}{p} \,.
\]
Let $\{u_{k}\}_{k=1}^\infty$ be the $L^s(Z)$-norm convergent subsequence selected in Corollary~\ref{cor:Besov-cpt-embedding} and $u \in L^s(Z)$ be the limit function. Then,
\begin{align*}
  \|u_k - u\|_{B^{\lambda \alpha}_{\tilde{p},q}(Z)} \lesssim \|u_k - u\|_{B^{\tilde{\lambda} \alpha}_{\tilde{p},\infty}(Z)} & \lesssim \|u_k - u\|_{L^s(Z)}^{1-\tilde{\lambda}} \|u_k - u\|_{B^\alpha_{p,\infty}(Z)}^{\tilde{\lambda}} \\
  & \lesssim \|u_k - u\|_{L^s(Z)}^{1-\tilde{\lambda}} \sup_{j\ge 1}\|u_j\|_{B^\alpha_{p,\infty}(Z)}^{\tilde{\lambda}} \to 0\quad \text{as $k\to\infty$}
\end{align*}
by Lemma~\ref{lem:decreasingsmoothness} (see also Remark~\ref{rem:decreasingsmoothness}) and Corollary~\ref{cor:BesovInterpolate}.
\end{proof}
\section{Traces of functions with a Poincar\'e inequality: general domains}
\label{sec:tracesGen}
In this section, we will consider functions that satisfy a Poincar\'e inequality in a domain $\Om \subset X$, cf.\@ Paragraph~\ref{subsec:PI}, and we will show that such functions have a measurable trace on $\dOm$ in the sense of~\eqref{eq:defoftrace} provided that $\Hcal\lfloor_{\dOm}$ has an upper codimension-$\theta$ bound \eqref{eq:H-upper-massbound} and a $p$-PI gradient of the Sobolev-type function exhibits sufficiently high summability.

Even though the theorems and proofs in this section are stated for boundary traces of functions on a bounded domain $\Om$, the arguments will also work whenever $\dOm$ is replaced by an arbitrary bounded set $F \subset \overline{\Om}$ such that $\Hcal \lfloor_F$ has an upper codimension-$\theta$ bound. In such a case $\overline{\Om}$ may very well be unbounded (unlike $F$).

We will apply boundedness of a certain fractional maximal operator. For $\alpha \ge \theta$, we define the centered operator
\begin{equation}
\label{eq:FracMax}
  M_{\alpha, p} f(z) = \sup_{0<r<2 \diam \dOm} \biggl( r^\alpha \fint_{B(z,r) \cap \Om} |f|^p\,d\mu \biggr)^{1/p}, \quad z\in \dOm,
\end{equation}
which maps $L^p_\loc(\Om)$ into the space of lower semicontinuous functions on $\dOm$. Let us now establish its boundedness.
\begin{lem}
\label{lem:fracMax-bdd}
Let $p\in [1, \infty)$. Suppose that $\alpha \in [\theta, s)$. Then, $M_{\alpha, p}: L^p(\Om) \to \text{weak-}L^{\frac{p(s-\theta)}{s-\alpha}}(\dOm)$ is bounded. Moreover, $M_{\alpha, p}: \text{weak-}L^{ps/\alpha}(\Om) \to L^\infty(\dOm)$ is bounded.

If however $\alpha \in [s, \infty)$, then $M_{\alpha, p}: L^p(\Om) \to L^\infty(\dOm)$ is bounded.
\end{lem}
\begin{proof}
Let $f \in L^p(\Om)$ and suppose that $\alpha \in [\theta, s)$.

 Having fixed $\lambda>0$, define $E_\lambda = \{z \in \dOm: M_{\alpha, p} f(z) > \lambda\}$. For each $z \in E_\lambda$, there is a ball $B_z = B(z, r_z)$ such that $r_z^\alpha \fint_{B_z \cap \Om} |f|^p\,d\mu > \lambda^p$. Then, $E_\lambda \subset \bigcup_{z\in E_\lambda} B_z \cap \dOm$. Since $\mu\lfloor_\Om$ is doubling and the radii $r_z$ are bounded by $2\diam \dOm$, we can apply the simple Vitali 5-covering lemma to find pairwise disjoint  balls $B_k \coloneq B_{z_k}$, $k=1,2,\ldots$, for some choice of $\{z_k\} \subset E_\lambda$ such that $E_\lambda \subset \bigcup_{k} 5B_k \cap \dOm$. Then,
\[
  \Hcal(E_\lambda) \le \Hcal\biggl(\bigcup_k 5B_k \cap \dOm \biggr) \lesssim \sum_k \Hcal(5B_k \cap \dOm) \lesssim \sum_k \frac{\mu(5B_k \cap \Om)}{(5r_k)^\theta} \approx \sum_k \frac{\mu(B_k \cap \Om)}{r_k^\theta}\,.
\]
By the choice of balls $B_z$, we have
\[
  \frac{\int_{B_k} |f|^p\,d\mu}{\lambda^p} > \frac{\mu(B_k \cap \Om)}{r_k^\alpha} \ge C \biggl(\frac{\mu(B_k \cap \Om)}{r_k^\theta}\biggr)^\beta,
\]
for some $\beta \in [0, 1]$ that we will now determine.

Recall that the doubling condition of $\mu\lfloor_\Om$ gives a lower mass bound, i.e., $\mu(B(z,r) \cap \Om) \ge C r^s$ whenever $r < 2\diam \dOm$.
We strive for $\mu(B(z,r) \cap \Om)^{1-\beta} / r^{\alpha - \theta \beta} \ge C$ for every $r \in (0, 2\diam\dOm)$. This estimate will be satisfied whenever $s(1-\beta)-(\alpha - \theta\beta) \le 0$. Thus, $\beta \ge (s-\alpha)/(s-\theta)$.

As $\alpha \in [\theta, s)$, we have $\beta\coloneq (s-\alpha)/(s-\theta) \in (0, 1]$. That gives us that
\[
  \Hcal(E_\lambda) \lesssim \sum_k \biggl(\frac{\int_{B_k} |f|^p\,d\mu}{\lambda^p} \biggr)^{(s-\theta)/(s-\alpha)} \lesssim \biggl(\frac{\sum_k \int_{B_k} |f|^p\,d\mu}{\lambda^p} \biggr)^{(s-\theta)/(s-\alpha)} \le \biggl(\frac{\int_{\Om} |f|^p\,d\mu}{\lambda^p} \biggr)^{(s-\theta)/(s-\alpha)}.
\]
Therefore,
\[
  \|M_{\alpha, p}f\|_{\text{weak-}L^{p(s-\theta)/(s-\alpha)}(\dOm)} = \sup_{\lambda>0} \lambda \Hcal(E_\lambda)^{(s-\alpha)/p(s-\theta)} \le \|f\|_{L^p(\Om)}\,.
\]

The (generalized) H\"older inequality gives us that
\begin{align*}
  \biggl( r^\alpha \fint_{B(z,r)} |f|^p\,d\mu\biggr)^{1/p} & \le  \biggl(r^{\alpha}\,\frac{\|\, |f|^p \chi_{B(z,r)}\|_{\text{weak-}L^{s/\alpha}(\Om)}}{\mu(B(z,r)\cap\Om)^{\alpha/s}}\biggr)^{1/p} \\
  & = \biggl(\frac{r}{\mu(B(z,r)\cap\Om)^{1/s}}\biggr)^{\alpha/p} \|f \chi_{B(z,r)}\|_{\text{weak-}L^{ps/\alpha}(\Om)}\,.
\end{align*}
The quantity $r/\mu(B(z,r)\cap \Om)^{1/s}$ is bounded due to \eqref{eq:ball-mu-vs-rads}, and hence $M_{\alpha, p}f(z) \lesssim \|f\|_{\text{weak-}L^{ps/\alpha}(\Om)}$ for every $z\in\dOm$.

Assume now that $\alpha \in [s, \infty)$. Recall that $0<r<2\diam \dOm$. By \eqref{eq:ball-mu-vs-rads}, we see that 
\[
  \frac{r^\alpha}{\mu(B(z,r) \cap \Om)} = \frac{r^s }{ \mu(B(z,r) \cap \Om)} \cdot r^{\alpha-s}\lesssim \frac{(\diam \dOm)^s }{ \mu(B(z,2\diam \dOm) \cap \Om)} \cdot r^{\alpha-s} \le C.
\]
This yields that $M_{\alpha, p} f \le C \|f\|_{L^p(\Om)}$ everywhere on $\dOm$, and hence $M_{\alpha, p} f \in L^\infty(\dOm)$.
\end{proof}
\begin{rem}
\label{rem:fracMaxbdd}
If $\alpha \in [\theta, s)$, we immediately obtain that $M_{\alpha, p} f \in L^\infty(\dOm)$ whenever $f \in L^q(\Om)$ with $q\ge ps/\alpha$. On the other hand, if $p < q <ps / \alpha$, then the Marcinkiewicz interpolation theorem yields that $M_{\alpha, p}: L^q(\Om) \to L^{p(s-\theta)/(\frac {ps}{q}- \alpha)}(\dOm)$ is bounded.
Furthermore,
it follows from \cite[Theorems~IV.4.11 and~IV.6.14]{BenSha} that $\|M_{\theta, p} f\|_{L^p(\dOm)} \lesssim \|f \log(e+|f|)\|_{L^p(\Om)}$.
\end{rem}
Now, we are ready to prove the first trace theorem of this section, which shows existence and boundedness of a trace operator for functions with a $p$-PI gradient in $L^p(\Om)$ with $p>\theta$.
\begin{thm}
\label{thm:TraceOpB}
Assume that $\Om \subset X$ is a bounded domain such that $\mu\lfloor_\Om$ is doubling and $\Hcal\lfloor_{\dOm}$ has an upper codimension-$\theta$ bound. Let $u \in P^{1,p}_p(\Om)$ with some $p \in (\theta, \infty)$. Then, there exists a trace $Tu \in L^p(\dOm)$ that satisfies \eqref{eq:defoftrace}, i.e.,
\[
  \lim_{R \to 0} \fint_{B(z, R)\cap \Om} |u(x) - Tu(z)|\,d\mu(x) = 0\quad\text{for $\Hcal$-a.e.\@ } z\in\dOm.
\]
Moreover, $T: P^{1,p}_p(\Om) \to B^{1-\theta/p}_{p,\infty}(\dOm)$ is a bounded linear operator.
\end{thm}
In fact, the trace exhibits better integrability than merely $L^p(\dOm)$ due to the embeddings of Besov spaces and boundedness of the fractional maximal operator, see Corollary~\ref{cor:Tr-openendedtargets} below.

As noted in the introductory paragraph of this section, the proof below can be easily modified to obtain that $T: P^{1,p}_p(\Om) \to B^{1-\theta/p}_{p,\infty}(F)$ is bounded whenever $F \subset \overline\Om$ is bounded (while $\Om \subset X$ may be unbounded) and $\Hcal\lfloor_F$ has an upper codimension-$\theta$ bound.
\begin{proof}
Let $u\in P^{1,p}_p(\Om)$ be given. For $z \in \dOm$ and $r>0$, we define $T_r u(z) = \fint_{B(z,r) \cap \Om} u\,d\mu$. Let us now show that $T_r u \in \Ccal(\dOm)$ for every $r>0$. For the sake of brevity, let $B_z$ denote the set $B(z, r)\cap \Om$ whenever $z \in \dOm$ for some fixed $r>0$. As balls have finite measure, we have $\mu(B_w \symdiff B_z) \to 0$ and, in particular, $\mu(B_w) \to \mu(B_z)$ as $w \to z$, where $w,z \in \dOm$. Since we are interested in the behavior of $T_r u(w)$ as $w\to z$, we may also assume that $B_w \subset 2B_z$. Then, we can estimate
\begin{align*}
  |T_r u(z) - T_r u(w)| & \le  \biggl| \frac{1}{\mu(B_z)}\int_{B_z \setminus B_w} u\,d\mu - \frac{1}{\mu(B_w)} \int_{B_w \setminus B_z} \,u\,d\mu  + \biggl(\frac{1}{\mu(B_z)}-\frac{1}{\mu(B_w)}\biggr) \int_{B_w \cap B_z} u\,d\mu\biggr|\\
  & \le \biggl(\frac{1}{\mu(B_z)}+\frac{1}{\mu(B_w)}\biggr) \int_{B_z \symdiff B_w} |u|\,d\mu + \biggl|\frac{1}{\mu(B_z)}-\frac{1}{\mu(B_w)}\biggr| \int_{2B_z} |u|\,d\mu,
\end{align*}
which approaches $0$ as $w\to z$ since $u \in L^1(\Om)$.

As the next step, we will show that $\{T_{r_k}u\}_{k=1}^\infty$ is a Cauchy sequence in $L^p(\dOm)$ whenever $r_k \to 0$. For $z\in\dOm$ and $\rho > 0$, let $B_{z,\rho} = B(z,\rho)\cap \Om$ and $\Om_\rho = \{x \in \Om: \dist(x,\dOm)<\rho\}$. Suppose that $g\in L^p(\Om)$ is a $p$-PI gradient of $u$ with a dilation factor $\lambda\ge 1$, see \eqref{eq:def-P1pq}. Let $0<\frac{R}{2} \le r < R < 2 \diam \Om$. Then,
\begin{align*}
  \|T_R u &- T_r u\|_{L^p(\dOm)}^p  = \int_{\dOm} |u_{B_{z,R}}-u_{B_{z,r}}|^p \,d\Hcal(z) \lesssim \int_{\dOm} \biggl(\fint_{B_{z,R}}|u(x)-u_{B_{z,R}}|\,d\mu(x)\biggr)^p \,d\Hcal(z) \\
  & \lesssim \int_{\dOm} R^p \fint_{B_{z, \lambda R}} g(x)^p\,d\mu(x)\,d\Hcal(z) 
   \approx \int_{\Om_{\lambda R}} g(x)^p \int_{B(x,\lambda R)\cap\dOm} \frac{R^p}{\mu(B_{z,\lambda R})}\,d\Hcal(z)\,d\mu(x)
  \\
  & \lesssim \int_{\Om_{\lambda R}} g(x)^p \int_{B(x,\lambda R)\cap\dOm} \frac{R^{p-\theta} \,d\Hcal(z)}{\Hcal(B(z,\lambda R)\cap\dOm)}\,d\mu(x) \lesssim R^{p-\theta}\,\int_{\Om_{\lambda R}} g(x)^p\,d\mu(x),
\end{align*}
where in the last inequality we used that $\Hcal$ is  doubling while $B(x,\lambda R)\cap \dOm \subset B(z, 2\lambda R) \cap \dOm$ whenever $z \in B(x,\lambda R)$. 

Let now $0 < r < R \le 2\diam\Om$. Then, there is $N \in \Nbb$ such that $2^{-N}R \le r < 2^{1-N} R$. Thus,
\begin{align}
  \notag \| T_R u - T_r u\|_{L^p(\dOm)} & \le \|T_r u - T_{2^{-N} R}u\|_{L^p(\dOm)}+ \sum_{k=1}^N \|T_{2^{1-k} R} - T_{2^{-k} R}\|_{L^p(\dOm)} \\
  \label{eq:TRTr-est}
  & \lesssim \sum_{k=1}^N (2^{1-k}R)^{1-\theta/p} \|g\|_{L^p(\Om_{2^{1-k}\lambda R})} \lesssim R^{1-\theta/p} \|g\|_{L^p(\Om_{\lambda R})}.
\end{align}
We have hereby shown that $\{T_{r_k} u\}_{k=1}^\infty$ is a Cauchy sequence in $L^p(\dOm)$ whenever $r_k \to 0$. Thus, we can define the trace of $u$ as the $L^p$-limit, i.e.,  $Tu \coloneq \lim_{r\to0} T_r u \in L^p(\dOm)$. Moreover, it follows from \eqref{eq:TRTr-est} that $T_{2^{-k}} u(z) \to Tu(z)$ as $k\to \infty$ for $\Hcal$-a.e.\@ $z\in\dOm$. Consequently, $T_R u(z) \to Tu(z)$ as $R\to 0$ for $\Hcal$-a.e.\@ $z\in\dOm$ since $\mu\lfloor_\Om$ is doubling.

If $R = 2 \diam \Om$, then $T_R u \equiv u_\Om$. Hence, \[
  \|Tu - u_\Om\|_{L^p(\dOm)} = \lim_{r\to 0} \|T_r u - T_R u\|_{L^p(\dOm)} \lesssim (\diam \Om)^{1-\theta/p} \|g\|_{L^p(\Om)}.
\]
Let $E=\{z \in \dOm: M_{\theta,p} g(z)<\infty\text{ and }T_r u(z) \to Tu(z)\text{ as }r\to0\}$. It follows from Lemma~\ref{lem:fracMax-bdd} that $M_{\theta,p}: L^p(\Om) \to \wk L^p(\dOm)$ is bounded, whence $\Hcal(E) = 0$. For every $z \in \dOm \setminus E$ and $r>0$, we obtain then that
\begin{align}
  \notag
  \fint_{B_{z,r}} |u(x) & - Tu(z)|\,d\mu(x) \le \fint_{B_{z,r}} |u(x) - T_r u(z)|\,d\mu(x) + |T_ru(z) - Tu(z)| \\
  \label{eq:traceconvergence}
  & \le r \biggl(\fint_{B_{z,\lambda r}} g^p\,d\mu\biggr)^{1/p} + |T_ru(z) - Tu(z)| \le r^{1-\theta/p} M_{\theta,p} g(z) + |T_ru(z) - Tu(z)|,
\end{align}
which approaches $0$ as $r \to 0$.

Finally, in order to show that $Tu \in B^{1-\theta/p}_{p,\infty}(\dOm)$, we will next find an estimate for $E_p(Tu, R)$, which was defined in~\eqref{eq:Eput-def}, for $R>0$. If $z,w \in \dOm$ with $\dd(z,w) \le R$, then 
\begin{equation}
  \label{eq:TuTuTriangle}
  |Tu(z) - Tu(w)| \le |Tu(z) - T_Ru(z)| + |T_Ru(z) - T_{2R}u(w)| + |T_{2R}u(w) - Tu(w)|.
\end{equation}
Note also that the doubling condition for $\mu\lfloor_\Om$ leads to the estimate
\begin{align}
  \notag
  |T_Ru(z) - T_{2R}u(w)| & \le \fint_{B_{z,R}} |u(x) - T_{2R} u(w)|\,d\mu(x) \lesssim \fint_{B_{w,2R}} |u(x) - T_{2R} u(w)|\,d\mu(x) \\
  &  = \fint_{B_{w,2R}} |u -u_{B_{w,2R}}|\,d\mu \le 2R \biggl(\fint_{B_{w,2\lambda R}} g^p\,d\mu\biggr)^{1/p}\,.
\end{align}
Since $\Hcal$ is doubling and $\dd(z,w)\le R$, we have $\Hcal(B(w,R)\cap\dOm) \approx \Hcal(B(z,R)\cap\dOm)$ with constants independent of $z$, $w$, and $R$. Thus,
\begin{align}
\notag
  \int_{\dOm} \fint_{B(w,R)} |Tu(z) &- T_Ru(z)|^p\,d\Hcal(z) \,d\Hcal(w) \\
\notag 
 & = \int_{\dOm} |Tu(z) - T_Ru(z)|^p \int_{B(z,R)} \frac{d\Hcal(w)}{\Hcal(B(w,R) \cap \dOm)}\,d\Hcal(z) \\
 & \approx \int_{\dOm} |Tu(z) - T_Ru(z)|^p\,d\Hcal(z). \label{eq:TuTRu-est}
\end{align}
Combining \eqref{eq:TuTuTriangle}--\eqref{eq:TuTRu-est} with \eqref{eq:TRTr-est}, we obtain that
\begin{align}
\notag
E_p(Tu,R)^p & = \int_{\dOm} \fint_{B(w,R)} |Tu(z) - Tu(w)|^p\,d\Hcal(z)\,d\Hcal(w) \\
\notag
& \lesssim \int_{\dOm} \biggl(|Tu(z) - T_{2R} u(z)|^p + |Tu(z) - T_{R} u(z)|^p + (2R)^p \fint_{B_{z,2\lambda R}} g^p(x)\,d\mu(x) \biggr)\,d\Hcal(z) \\
\label{eq:Ep-TuR}
& \lesssim (2R)^{p-\theta} \|g\|_{L^p(\Om_{2\lambda R})}^p\,.
\end{align}
Hence, $\|Tu\|_{\dot B^{1-\theta/p}_{p,\infty}(\dOm)} = \sup_{R>0} E_p(Tu, R)/R^{1-\theta/p} \lesssim \|g\|_{L^p(\Om)}$, which concludes the proof.
\end{proof}
Comparing the situation with well-known trace theorems for domains in the Euclidean spaces, one can expect that the trace class of $P^{1,p}_*(\Om)$ should be the Besov space $B^{1-\theta/p}_{p,p}(\dOm)$ provided that the boundary is sufficiently regular. In Section~\ref{sec:tracesJohn} below, we will obtain the expected boundedness of the trace operator for $P^{1,p}_p(\Om)$ and $P^{1,p}_q(\Om)$ if $\Om$ is a uniform or a John domain, respectively.
\begin{pro}
\label{pro:TraceOp1}
Let $u \in P^{1,p}_p(\Om)$ for some $p \in (\theta, \infty)$. Then, there exist a set $E \subset \dOm$ with $\Hcal(E)=0$ and a function $h\in \wk L^p(\dOm)$ with  $\|h\|_{\wk L^p(\dOm)}\lesssim \|g\|_{L^p(\Om)}$, where $g\in L^p(\Om)$ is a $p$-PI gradient of $u$, such that
\[
  |Tu(z) - Tu(w)| \lesssim \dd(z,w)^{1-\theta/p} \bigl(h(z) + h(w)\bigr)\quad\text{for every }z,w \in \dOm \setminus E.
\]
\end{pro}
\begin{proof}
Let $0 < R \le 2 \diam(\Om)$ and $z\in \dOm$. Using the notation established in the proof of Theorem~\ref{thm:TraceOpB}, we have
\[
  |T_{R}u(z) - T_{R/2}u(z)| \lesssim R \biggl(\fint_{B_{\lambda R}} g^p\biggr)^{1/p} \le R^{1-\theta/p} M_{\theta,p}g(z)\,,
\]
where $M_{\theta,p}$ is the fractional maximal operator defined in \eqref{eq:FracMax}. Thus,
\[
 |T_R u(z) - Tu(z)| \le \sum_{k=1}^{\infty} |T_{2^{-k} R} u(z) - T_{2^{1-k} R}u(z)| \le \sum_{k=1}^{\infty} (2^{1-k}R)^{1-\theta/p} M_{\theta,p}g(z) \approx R^{1-\theta/p} M_{\theta,p}g(z). 
\]
Let now $z,w \in \dOm$ and $R = \dd(z,w)$. Then,
\begin{align}
  \notag
  |Tu(z) - Tu(w)| & \le |Tu(z) - T_Ru(z)| + |T_R u(z) - T_{2R} u(w)| + |T_{2R} u (w) - Tu(w)| \\
  \label{eq:TrHajlaszTypeIneq}
  & \le c R^{1-\theta/p} (M_{\theta,p} g(z) + M_{\theta,p} g(w))\,.
\end{align}
If we set $h = c M_{\theta,p} g$, then $\|h\|_{\wk L^p(\dOm)} \lesssim \|g\|_{L^p(\Om)}$ by Lemma~\ref{lem:fracMax-bdd} and the proof is complete.
\end{proof}
If the couple $(u,g)$ satisfies a Poincar\'e inequality in the form
\[
  \biggl(\fint_{B\cap \Om} |u-u_B|^{p^\maltese}\,d\mu\biggr)^{1/p^\maltese} \le C_\PI \rad(B) \biggl(\fint_{\lambda B\cap\Om} g^p\,d\mu\biggr)^{1/p}
\]
for some $p^\maltese>1$, then the trace satisfies not only \eqref{eq:defoftrace}, but also a stronger relation
\begin{equation}
  \label{eq:betterTconv}
  \lim_{R \to 0} \fint_{B(z,R)\cap \Om} |u(x) - Tu(z)|^{p^\maltese}\,d\mu(x) = 0 \quad\text{for $\Hcal$-a.e.\@ $z\in\dOm$.}
\end{equation}
Indeed, similarly as in \eqref{eq:traceconvergence}, one obtains for $\Hcal$-a.e.\@ $z\in\dOm$ that
\begin{align*}
  \biggl(\fint_{B_{z,r}} |u(x) - Tu(z)|^{p^\maltese}\,d\mu\biggr)^{1/p^\maltese} & \le \biggl(\fint_{B_{z,r}} |u (x) - T_ru(z)|^{p^\maltese}\,d\mu(x)\biggr)^{1/{p^\maltese}} + |T_r u(z) - Tu(z)| \\ 
  & \lesssim r^{1-\theta/p} M_{\theta,p}g(z)+ |T_r u(z) - Tu(z)| \to 0 \qquad \text{as }r\to 0.
\end{align*}
In view of the Sobolev-type embeddings \cite[Theorem~5.1]{HajKos}, \eqref{eq:betterTconv} holds for every $p^\maltese < (s-p)/sp$ if $p<s$ and for every $p^\maltese<\infty$ if $p\ge s$. If $u \in P^{1,p}_q$ for some $q<p<s$, then $p^\maltese = (s-p)/sp$ is also possible by Corollary~\ref{cor:Hajlasz-infPI} and \cite[Theorem~6]{Haj96}.
\begin{cor}
\label{cor:Tr-openendedtargets}
Suppose that $u \in P^{1,p}_p(\Om)$ for some $p>\theta$. Then:
\begin{enumerate}
  \item \label{it:tr1} $Tu \in B^{\alpha}_{m,q}(\dOm)$ for every $0 \le \alpha < 1-\frac{\theta}{p}$, $1 \le m < p_\alpha$, and $0<q\le \infty$, where $p_\alpha = \frac{p(s-\theta)}{s-p(1-\alpha)}$ if $p(1-\alpha)<s$ and $p_\alpha = \infty$ otherwise.
  \item \label{it:tr2} $Tu \in L^m(\dOm)$ for every $m \in [1, p^*)$, where $p^* = \frac{p(s-\theta)}{s-p}$ if $p<s$, while $p^*=\infty$ otherwise.
  \item \label{it:tr3} If $p>s$, then $Tu \in \Ccal^{0, \kappa}(\dOm)$, where $\kappa = 1-\frac{s}{p}$\,.
\end{enumerate}
\end{cor}
\begin{proof}
\ref{it:tr1} Given $\alpha \in [0, 1-\frac{\theta}{p})$ and $m \in [1, p_\alpha)$, we can find $\tilde{\alpha} \in (\alpha, 1-\frac{\theta}{p})$ such that $m < p_{\tilde{\alpha}}$. A~simple modification of the proof of Proposition~\ref{pro:TraceOp1} yields an analog of \eqref{eq:TrHajlaszTypeIneq}, viz.\@ $|Tu(z) - Tu(w)| \lesssim \dd(z,w)^{\tilde{\alpha}} (M_{p-\tilde\alpha p,p}g(z) + M_{p-\tilde\alpha p,p}g(w))$ while $|Tu(z) - u_\Om| \lesssim (\diam \Om)^{\tilde{\alpha}} M_{p-\tilde\alpha p,p}g(z)$. By Lemma~\ref{lem:fracMax-bdd}, $M_{p-\tilde\alpha p,p}g \in \wk L^{p_{\tilde{\alpha}}}(\dOm) \subset L^m(\dOm)$. Then, Lemmata~\ref{lem:HajlaszIsBesov} and~\ref{lem:decreasingsmoothness} give the desired result.

\ref{it:tr2} The assertion follows immediately from \ref{it:tr1} by choosing $\alpha = 0$.

\ref{it:tr3} We use another analog of~\eqref{eq:TrHajlaszTypeIneq}, viz.\@ $|Tu(z) - Tu(w)| \le C \dd(z,w)^{1-s/p} (M_{s,p}g(z) + M_{s,p}g(w))$. Since $M_{s,p}: L^p(\Om) \to L^\infty(\dOm)$ is bounded by Lemma~\ref{lem:fracMax-bdd}, we obtain that $Tu$ has a $(1-\frac{s}{p})$-H\"older continuous representative.
\end{proof}
\begin{cor}
\label{cor:Tr-generic-codimReg}
Assume that $\Hcal$ is Ahlfors codimension-$\theta$ regular. Suppose that $u \in P^{1,p}_p(\Om)$ for some $p>\theta$. Then:
\begin{enumerate}
  \item $Tu \in B^{\alpha}_{p_\alpha,\infty}(\dOm)$ for every $1-\frac{s}{p} < \alpha \le 1-\frac{\theta}{p}$, where $p_\alpha = \frac{p(s-\theta)}{s-p(1-\alpha)}$.
  \item $Tu \in \wk L^{p^*}(\dOm) \subset \bigcap_{m<p^*} L^m(\dOm)$, where $p^* = \frac{p(s-\theta)}{s-p}$ provided that $p<s$.
\end{enumerate}
\end{cor}
\begin{proof}
Follows from Theorem~\ref{thm:TraceOpB}, Remark~\ref{rem:sharp-besov-emb}, and Theorem~\ref{thm:Besov-embedding-closedend}.
\end{proof}
\begin{pro}
\label{pro:Trace-compact-operator}
The trace $T: M^{1,p}(\Om) \to L^m(\dOm)$ is a compact operator for every $m < p^*$, where $p^* = \frac{p(s-\theta)}{s-p}$ if $p<s$ while $p^* = \infty$ otherwise. Moreover, if $\Om$ admits a $p$-Poincar\'e inequality, then $T: N^{1,p}(\Om) \to L^m(\dOm)$ is compact for every $m<p^*$.
\end{pro}
\begin{proof}
Follows from Corollaries~\ref{cor:Hajlasz-infPI}, \ref{cor:Tr-openendedtargets}\,\ref{it:tr1}, and~\ref{cor:Besov-cpt-embedding}; see also commentary after Definition~\ref{df:PI}. 
\end{proof}
The proofs of statements about the trace operator above can be easily adapted for functions in the Besov class $B^\alpha_{p,q}(\Om)$ provided that $\alpha p>\theta$, $q\in (0,\infty]$. If $u \in B^\alpha_{p,q}(\Om)$ for some $q\le p$, then Lemma~\ref{lem:BesovHaj} provides us with a function $g\in L^p(\Om)$ so that $(u,g)$ satisfies the fractional Haj\l asz gradient inequality \eqref{eq:BesovHaj}. By \cite[Theorem~9.3]{Haj}, we have $u \in P^{\alpha}_{p,q}(\Om)$ for some $q<p$, i.e.,
\[
  \fint_{B} |u - u_B|\,d\mu \le \rad(B)^\alpha \biggl(\fint_{\lambda B} g^q \,d\mu\biggr)^{1/q},
\]
where $g\in L^p(\Om)$. One obtains that $Tu \in B^{\alpha-\theta/p}_{p,\infty}(\dOm)$ by following the steps of the proof of Theorem~\ref{thm:TraceOpB}. Similarly as in~\eqref{eq:TrHajlaszTypeIneq}, we have $|Tu(z) - Tu(w)| \lesssim \dd(z,w)^{\alpha - \theta/p} (M_{\theta,p}g(z) + M_{\theta,p}g(w))$. Open-ended boundedness and compactness of the trace operator can then be obtained similarly as in Corollary~\ref{cor:Tr-openendedtargets} and Proposition~\ref{pro:Trace-compact-operator}.

In case $u\in B^\alpha_{p,q}(\Om)$ with $q \in (p, \infty]$, one can first use Lemma~\ref{lem:decreasingsmoothness} to see that $u\in B^{\alpha-\eps}_{p,p}(\Om)$ for every $\eps>0$ and then proceed as above to show that $Tu \in B^{\alpha-\eps-\theta/p}_{p,\infty}(\dOm)$ and thus obtain the corresponding open-ended results.

Next, we will show a positive result on existence of a trace when a function has a $\theta$-PI gradient in a weighted $L^\theta(\Om)$. In Proposition~\ref{pro:T-intoLp-sharpness} below, it is shown that the hypothesis here is essentially sharp for $\theta = p > 1$. On the other hand, if $\theta = p = 1$ and $\dOm$ is Ahlfors codimension-1 regular, then $T: P^{1,1}_1(\Om) \to L^1(\dOm)$ without any extra assumptions on weighted $L^1$-integrability of the 1-PI gradient by \cite{LahSha}.
\begin{thm}
\label{thm:TraceOp-theta=p}
Assume that $\theta = p$. Let $w: (0, \infty) \to [1, \infty)$ be a decreasing function such that $\int_0^1 \frac{dt}{t w(t)} < \infty$. Suppose that $u \in P^{1,p}_p(\Om)$ and $\tilde{g}(x) \coloneq g(x) w(\dist(x,\dOm)) \in L^p(\Om)$, where $g\in L^p(\Om)$ is a $p$-PI gradient of $u$. Then, $Tu$ satisfying \eqref{eq:defoftrace} exists and $\|Tu - u_\Om\|_{L^p(\dOm)} \lesssim \|\tilde{g}\|_{L^p(\Om)}$.
\end{thm}
In particular, the proposition can be applied for the weight function
\[
  w(t) = \begin{cases}
    \bigl(\log \frac et\bigr)^{1+\eps} & \text{for } 0<t<1,\\
    1 & \text{for }1\le t < \infty,
  \end{cases}
  \quad\text{for which}
  \quad
  \tilde{g}(x) \approx g(x) \biggl(\log \frac{2 \diam \Om}{\dist(x, \dOm)}\biggr)^{1+\eps}, \quad x\in\Om,
\]
where $\eps>0$ is arbitrary.
\begin{proof}[Sketch of proof]
We can follow the proof of Theorem~\ref{thm:TraceOpB} without any modification until \eqref{eq:TRTr-est}. In what follows, we will be using the notation of the aforementioned proof.

Let $0<\frac{R}{2} \le r < R < 2 \diam \Om$. Then,
\[
  \|T_R u - T_r u\|_{L^p(\dOm)} \lesssim \biggl(\int_{\Om_{\lambda R}} g(x)^p\,d\mu(x)\biggr)^{1/p}
  \le \frac{\|\tilde g\|_{L^p(\Om_{\lambda R})}}{w(\lambda R)}\,.
\]
Let now $0 < r < R < 2\diam\Om$. Then, there is $N \in \Nbb$ such that $2^{-N}R \le r < 2^{1-N} R$. Thus,
\begin{align*}
  \| T_R u - T_r u\|_{L^p(\dOm)} & \le \|T_r u - T_{2^{-N} R}u\|_{L^p(\dOm)}+ \sum_{k=1}^N \|T_{2^{1-k} R} - T_{2^{-k} R}\|_{L^p(\dOm)} \\
  & \lesssim \sum_{k=1}^N \frac{\|\tilde{g}\|_{L^p(\Om)}}{w(2^{1-k} \lambda R)} \approx  \sum_{k=1}^N \frac{\|\tilde{g}\|_{L^p(\Om)}}{w(2^{1-k} \lambda R)} \int_{2^{-k}\lambda R}^{2^{1-k}\lambda R} \frac{dt}{t} \le \|\tilde{g}\|_{L^p(\Om)} \int_0^{2\lambda R} \frac{dt}{t w(t)}.
\end{align*}
Since $W(\rho) \coloneq \int_0^\rho \frac{dt}{tw(t)} \to 0$ as $\rho \to 0$, we have hereby shown that $\{T_{r_k} u\}_k$ is a Cauchy sequence in $L^p(\dOm)$ whenever $r_k \to 0$. Thus, the limit function $Tu$ lies in $L^p(\dOm)$ and $T_R u(z) \to Tu(z)$ as $R\to 0$ for $\Hcal$-a.e.\@ $z\in\dOm$.

If $R = 2\diam \Om$, then $T_R u \equiv u_\Om$. Hence, $\|Tu - u_\Om\|_{L^p(\dOm)} 
\lesssim \|\tilde g\|_{L^p(\Om)} W(2 \lambda \diam \Om)$.

Let $E=\{z \in \dOm: M_{\theta,p}\tilde{g}(z)<\infty\text{ and }T_r u(z) \to Tu(z)\text{ as }r\to0\}$. Then, $\Hcal(E) = 0$ and for every $z \in \dOm \setminus E$, we obtain that
\[
  \fint_{B_{z,r}} |u(x)  - Tu(z)|\,d\mu(x) 
  \le r \biggl(\fint_{B_{z,r}} g^p\,d\mu\biggr)^{1/p} + |T_ru(z) - Tu(z)| \le \frac{M_{\theta,p} \tilde{g}(z)}{w(r)} + |T_ru(z) - Tu(z)|,
\]
which approaches $0$ as $r \to 0$.
\end{proof}
\begin{cor}
\label{cor:Tr-LpLogLp}
Let $p\ge 1$. Assume that both \eqref{eq:H-lower-massbound} and \eqref{eq:H-upper-massbound} are satisfied for some $0 < \vartheta \le \theta = p$. Suppose that $u \in P^{1,p}_{p}(\Om)$ has a $p$-PI gradient $g\in L^p(\Om)$ such that $\breve{g} \coloneq g \log (e+g)^{1+\eps} \in L^p(\Om)$ for some $\eps>0$. Then, $\|Tu-u_\Om\|_{L^p(\dOm)} \lesssim (1+ \|\breve g\|_{L^p(\Om)})$. 
\end{cor}
\begin{proof}
The claim follows from Theorem~\ref{thm:TraceOp-theta=p} with the weight $w(t) = \bigl(1+\log_+(\Delta/t)\bigr)^{1+\eps}$, where $\Delta = \diam\Om$. Let us therefore verify that $\tilde{g} \in L^p(\Om)$, where $\tilde{g}$ is as in the theorem above.

For the sake of brevity, let $\delta(x) = \dist(x,\dOm)$.  Define $\Om_k = \{x\in \Om:  2^{-k}\Delta < \delta(x) \le 2^{1-k}\Delta\}$, $k=1,2,\ldots$. Then, $\mu(\Om_k) \lesssim \Hcal(\dOm) (2^{-k}\Delta)^\vartheta$ by Lemma~\ref{lem:shell-measure}. Let $\alpha \in [0, \vartheta)$, $\beta > 0$.
\begin{align*}
  \int_\Om \biggl(\frac{\Delta}{\delta(x)}\biggr)^\alpha \log^\beta\biggl(\frac{\Delta}{\delta(x)}\biggr) \,d\mu & \le \sum_{k=1}^\infty \int_{\Om_k} \biggl(\frac{1}{2^{-k}}\biggr)^\alpha \log^\beta \biggl(\frac{1}{2^{-k}}\biggr)\,d\mu \\
  & \approx \sum_{k=1}^\infty 2^{k\alpha} k^{\beta} \mu(\Om_k) \lesssim \Hcal(\dOm) \sum_{k=1}^\infty k^{\beta} 2^{k(\alpha-\vartheta)}<\infty.
\end{align*}
We can split $\Om$ into two parts, namely $E=\bigl\{x\in \Om: g(x) \le \bigl(\Delta/\delta(x)\bigr)^{\vartheta/2p} \bigr\}$ and $F = \Om \setminus E$. Then,
\begin{align*}
   \int_\Om g(x)^p \log^{p(1+\eps)}\biggl(\frac{\Delta}{\delta(x)}\biggr)\,d\mu & \le \int_E \biggl(\frac{\Delta}{\delta(x)}\biggr)^{\vartheta/2} \log^{p(1+\eps)}\biggl(\frac{\Delta}{\delta(x)}\biggr) \, d\mu \\
  & \qquad + \int_F g(x)^p \biggl(\frac{2p }{\vartheta} \log\bigl(g(x)\bigr)\biggr)^{p(1+\eps)}\,d\mu \lesssim 1 + \|\breve{g}\|_{L^p(\Om)}^p.
\qedhere
\end{align*}
\end{proof}
\begin{pro}
\label{pro:TraceOp-theta>p-l}
Assume that $\theta > p$. Suppose that $\tilde{g}(x) \coloneq g(x) \log(\frac{2\diam(\Om)}{\dist(x,\dOm)})^{1/p} \in L^p(\Om)$, where $g\in L^p(\Om)$ is a $p$-PI gradient of  $u \in P^{1,p}_p(\Om)$. Then, $Tu \in B^{1-\theta/p}_{p,p}(\dOm)$.

In particular, in order to obtain $Tu \in B^{1-\theta/p}_{p,p}(\dOm)$, it suffices to assume that both \eqref{eq:H-lower-massbound} and \eqref{eq:H-upper-massbound} are satisfied for some $0 < \vartheta \le \theta<p$ and that $\breve g \coloneq g \log (e+g)^{1/p} \in L^p(\Om)$ .
\end{pro}
\begin{proof}[Sketch of proof]
Let \eqref{eq:Ep-TuR} from the proof of Theorem~\ref{thm:TraceOpB} be the starting point of the current proof. Let $K = 2 \diam\Om$. Then,
\begin{align*}
  \|Tu\|_{\dot B^{1-\theta/p}_{p,p}(\dOm)}^p & = \int_0^K \frac{E_p(Tu, R)^p}{R^{p-\theta}} \frac{dR}{R} \lesssim \int_0^K \|g\|_{L^p(\Om_{2\lambda R})}^p \frac{dR}{R} \\
  & = \int_\Om g(x)^p \int_{\dist\{x, \dOm\}/2\lambda}^K \frac{dR}{R}\,d\mu(x) = \int_\Om g(x)^p \log\biggl(\frac{2\lambda K}{\dist\{x, \dOm\}}\biggr)\,d\mu(x) \approx \|\tilde g\|_{L^p(\Om)}^p.
\end{align*}
If $\breve g \in L^p(\Om)$, then $\tilde g\in L^p(\Om)$, which can be proven similarly as in Corollary~\ref{cor:Tr-LpLogLp}.
\end{proof}
\begin{rem}
Instead of assuming that $\tilde{g}\in L^p(\Om)$ in Theorem~\ref{thm:TraceOp-theta=p} and~\ref{pro:TraceOp-theta>p-l}, it is possible to obtain the desired boundedness of the trace operator by logarithmically refining the codimension relation between $\Hcal\lfloor_{\dOm}$ and $\mu\lfloor_\Om$. Namely, if there is $\eta > 1$ such that
\[
  \Hcal(B(z,r) \cap \dOm) \lesssim \frac{\mu(B(z,r)\cap \Om)}{r^\theta  \log\bigl(e+\frac{1}{r}\bigr) \log\bigl(e+\log \bigl(1 + \frac{1}{r}\bigr)\bigr)^{\eta}},\quad z\in \dOm,\ 0< r < 2 \diam \Om,
\]
then $T: P^{1,p}_p(\Om) \to B^{1-\theta/p}_{p,p}(\dOm)$ is bounded whenever $p>\theta$.
Similarly, if there is $\eta > 1$ such that
\[
  \Hcal(B(z,r) \cap \dOm) \lesssim \frac{\mu(B(z,r)\cap \Om)}{\bigl(r \log\bigl(e+\frac{1}{r}\bigr) \log\bigl(e+\log \bigl(1 + \frac{1}{r}\bigr)\bigr)^{\eta}\bigr)^\theta},\quad z\in \dOm,\ 0< r < 2 \diam \Om,
\]
then $T: P^{1,\theta}_\theta(\Om) \to L^\theta(\dOm)$ is bounded. Both these statements can be proven as Theorem~\ref{thm:TraceOpB}, \emph{mutatis mutandis}.
\end{rem}
This section will now be concluded by establishing  local estimates for the trace operator. Such estimates can be applied when studying regularity of solutions of the Neumann problem for the $p$-Laplace equation using the de\,Giorgi method, cf.\@ \cite{MalSha}.

If a domain $\Om \subset X$ is quasiconvex, then a biLipschitz change of the metric will turn it into a length space. Recall that $\Om$ is called \emph{quasiconvex} if every two distinct points $x,y \in \Om$ can be connected by a curve inside $\Om$ whose length is at most $L \dd(x,y)$ for some universal constant $L\ge 1$. If $L=1$, then we say that $\Om$ a \emph{geodesic space}. If $\Om$ is quasiconvex for every $L>1$, then it is called a \emph{length space}.

A biLipschitz modification of the metric does not change the Haj\l asz space $M^{1,p}(\Om)$ and hence the $P^{1,p}_q(\Om)$ space also remains unchanged. Thus, we do not lose generality by assuming that $\Om$ is a length space (rather than merely quasiconvex) in the proposition below.
\begin{pro}
\label{pro:TraceLocEst}
Assume that a domain $\Om \subset X$ is a length space. Suppose that $\theta < p <s$ while $p \ge 1$. Let $\tilde{p} \in (p, p^*)$, where $p^* = p(s-\theta)/(s-p)$. Then, the trace operator $T$ constructed in Theorem~\ref{thm:TraceOpB} satisfies the local estimate
\[
  \| Tu - u_{B \cap \Om} \|_{L^{\tilde{p}}(B \cap \dOm)} \le C \rad(B)^{(\frac{1}{\tilde{p}} - \frac{1}{p^*})(s-\theta)} \|g\|_{L^p(B \cap \Om)}
\]
for every $u \in P^{1,p}_q(\Om)$ with $q<p$ and every ball $B$ centered at an arbitrary boundary point, where $g\in L^p(\Om)$ is a $q$-PI gradient of $u$.
\end{pro}
In particular, the conclusion of the proposition holds true for every $u \in N^{1,p}(\Om)$ with $g\in L^p(\Om)$ being its $p$-weak upper gradient provided that $\Om$ admits a $p$-Poincar\'e inequality and $p>1$. 
\begin{proof}
Fix a couple of $L^p$ functions $(u,g)$ that satisfy \eqref{eq:def-P1pq} with $q < p$ and $\lambda = 1$. By the H\"older inequality, we may assume that $q$ is sufficiently close to $p$ so that $\frac sp - \frac \theta q > \frac{s-\theta}{\tilde{p}}$. Fix also a ball $B = B(z, R)$.
For every point $x \in B \cap \dOm$, define $r_x = \frac12(R - \dd(x,z)) \le \frac12 \dist(\{x\}, \dOm \setminus B)$. Since $\Om$ is a length space, we can find an arc-length parametrized curve $\gamma_x: [0, l_x] \to \overline\Om$ such that $\gamma_x(0) = z$, $\gamma_x(l_x) = x$, $\gamma_x\bigl( (0, l_x) \bigr) \subset \Om$, and $l_x \le (1+\delta) \dd(x,z)$, where the constant $\delta=\delta_x \in (0, 1)$ is chosen such that $(1+\delta)l_x < R$.

Next, we construct a finite decreasing sequence of balls whose centers lie on $\gamma_x$ and all the balls contain the point $x$ and are contained in $B(z, R)$.
Let $N = \lceil \log_2 (2R/r_x) \rceil$. For each $k=0, 1, \ldots, N$, let $r_k = (\delta^{k+1} + 2^{-k}) l_x$ and $x_k = \gamma_x\bigl((1-2^{-k}) l_x\bigr)$. Then, we define $B_k = B(x_k, r_k)$. It follows from the triangle inequality that $B_{k+1} \subset B_k \subset B(z,R)$, and $x \in B_k$ for all $k=0,1,\ldots N$. For the sake of clarity, let us see detailed calculations:
\begin{align*}
  \dd(y,z) & \le \dd(y, x_k) + \dd(x_k, z) \\
  & < (\delta^{k+1} + 2^{-k}) l_x + (1-2^{-k}) l_x = (1+\delta^{k+1}) l_x < R\quad \text{whenever }y \in B_k; \\
  \dd(w, x_k) & \le \dd(w, x_{k+1}) + \dd(x_k, x_{k+1}) < r_{k+1} + 2^{-(k+1)} l_x \\
  & \le (\delta^{k+2} + 2\cdot 2^{-(k+1)}) l_x < (\delta^{k+1} + 2^{-k}) l_x = r_k \quad\text{whenever }w \in B_{k+1};\\
 \dd(x, x_k) & \le l_x - (1-2^{-k} l_x) = 2^{-k}l_x < r_k.
\end{align*}
For $k>N$, we define $B_k = B(x, 2^{-k} l_x) \subset B(x, r_x)$. The difference $|Tu(x) - u_{B\cap \Om}|$ can be estimated using the chain of balls $\{B_k\}_k$ since $Tu(x) = \lim_{k\to\infty} u_{B_k \cap \Om}$ as follows:
\begin{align*}
  |u_{B\cap \Om}  - Tu(x)| &\le |u_{B\cap\Om} - u_{{B_0}\cap\Om}| + \sum_{k=1}^\infty |u_{B_k\cap\Om} - u_{{B_{k-1}}\cap\Om}| \\ 
  & \lesssim \fint_{B\cap\Om} |u-u_{B\cap\Om}|\,d\mu + \sum_{k=0}^\infty \fint_{B_k\cap\Om} |u - u_{B_k\cap\Om}|\,d\mu \\
  & \lesssim R \biggl(\fint_{ B\cap\Om} g^q \,d\mu\biggr)^{1/q} + \sum_{k=0}^\infty 2^{-k}R \biggl(\fint_{  B_k\cap\Om} g^q \,d\mu\biggr)^{1/q} 
    \lesssim R^{(\frac{1}{\tilde p} - \frac{1}{p^*})(s-\theta)} M^*_{\alpha, q} g(x),
\end{align*}
where $\alpha = q-(\frac{1}{\tilde p} - \frac{1}{p^*})(s-\theta)q$ and $M^*_{\alpha, q}$ denotes a restricted non-centered fractional maximal operator, defined by
\[
  M_{\alpha, q}^* f(x) = \sup_{\substack{x \ni U\\ \text{ball } U \subset  B}} \biggl(\rad(U)^\alpha \fint_{U\cap \Om} |f|^q\,d\mu\biggr)^{1/q}, \quad x\in B \cap \dOm, f\in L^q( B \cap \Om).
\]
The condition $\frac sp - \frac \theta q > \frac{s-\theta}{\tilde{p}}$ established at the beginning of the proof yields that $\alpha > \theta$. Boundedness of the restricted non-centered fractional maximal operator can be proven similarly as for the non-restricted centered operator in Lemma~\ref{lem:fracMax-bdd}. In particular, $M_{\alpha, q}: L^p(B \cap \Om) \to L^{\tilde p}(B\cap \dOm)$ as in Remark~\ref{rem:fracMaxbdd} (note however that the roles of $q$ and $p$ are switched in that remark). Thus,
\[
   \| u_{B\cap \Om} - Tu\|_{L^{\tilde{p}}(B\cap \dOm)}  \lesssim R^{(\frac{1}{\tilde p} - \frac{1}{p^*})(s-\theta)} \| M^*_{\alpha, q}g \|_{L^{\tilde p}(B \cap \dOm)} \lesssim R^{(\frac{1}{\tilde p} - \frac{1}{p^*})(s-\theta)} \|g\|_{L^p(B\cap \Om)}\,.
  \qedhere
\]
\end{proof}
\begin{rem}
\label{rem:TraceLocEst}
Under the assumptions of Proposition~\ref{pro:TraceLocEst}, one can show that
\[
  \| Tu - u_{B \cap \Om} \|_{L^{\tilde{p}}(B \cap \dOm)} \le C_\eps \rad(B)^{(\frac{1}{\tilde p} - \frac{1}{p^*}-\eps) (s-\theta)} \Hcal(B\cap \dOm)^{\eps} \|g\|_{L^p(\lambda B \cap \Om)}
\]
for every $u \in P^{1,p}_p(\Om)$, $\eps \in \bigl(0, \frac{1}{\tilde p} - \frac{1}{p^*}\bigr]$, and every ball $B$ centered at an arbitrary boundary point. Indeed, one can follow the steps of the proof above, but pick $\breve p \in (\tilde p, p^*)$ and use the boundedness of $M^*_{\breve{\alpha},p}: L^p(B\cap\Om) \to \wk L^{\breve p}(B\cap\dOm) \emb L^{\tilde p}(B\cap\dOm)$, where $\breve \alpha = p - p (\frac{1}{\tilde p} - \frac{1}{p^*}-\eps) (s-\theta)$.
\end{rem}
\section{Traces of functions with a Poincar\'e inequality: John and uniform domains}
\label{sec:tracesJohn}
In this section, we will focus on showing that it is in fact possible to obtain a closed-ended smoothness result for the trace, namely that $T: P^{1,p}_q(\Om) \to B^{1-\theta/p}_{p,p}(\dOm)$, where $q<p$ and $p\ge 1$, under the additional assumption that $\Om$ is a John domain with compact closure. If $\Om\Subset X$ is a uniform domain, then also $T: P^{1,p}_p(\Om) \to B^{1-\theta/p}_{p,p}(\dOm)$ is bounded.
\begin{lem}
\label{lem:trace-int-est1}
Let $z \in \dOm$ be fixed. Suppose that $a \in  (0, \diam(\dOm))$ and $\alpha \in (0, \infty)$.
Then,
\[
  \int_{\dOm \setminus B(z,a)} \frac{d\Hcal(w)}{\dd(z,w)^\alpha \Hcal(B(z, \dd(z,w))\cap \dOm)} \lesssim
  \frac{1}{a^\alpha}
\]
\end{lem}
\begin{proof}
Let us split the integration domain dyadically with respect to the distance from $z$. For $k=0,1,\ldots,$ let $A_k = B(z,2^{k+1}a) \setminus B(z,2^k a)$. Then,
\begin{align*}
 \int_{\dOm \setminus B(z,a)}  \frac{d\Hcal(w)}{\dd(z,w)^\alpha  \Hcal(B(z, \dd(z,w))\cap \dOm)}
  = \sum_{k=0}^{K_a} \int_{\dOm \cap A_k} \frac{d\Hcal(w)}{\dd(z,w)^\alpha \Hcal(B(z, \dd(z,w))\cap \dOm)}&  \\
  \approx \sum_{k=0}^{K_a} \frac{1}{(2^k a)^\alpha} \cdot
  \frac{\Hcal(B(z, 2^{k+1} a)\cap \dOm)}{\Hcal(B(z, 2^k a)\cap \dOm)}
  \lesssim \frac{1}{a^\alpha} \sum_{k=0}^{K_a} 2^{-k\alpha} & \approx \frac{1}{a^\alpha},
\end{align*}
where $K_a \coloneq \lceil \log_2 \diam(\dOm)/a \rceil -1$. 
\end{proof}
\begin{lem}
\label{lem:trace-int-est2}
Fix $x \in \Om$ and let $\delta = \dist (x, \dOm)$. Suppose that $a \in ( \delta, 2 \diam(\Om))$ and $\alpha \in (0, \infty)$. Then,
\[
  \int_{\dOm \cap B(x,a)} \frac{\dd(x,z)^\alpha }{ \Hcal(B(z, \dd(x,z))\cap \dOm)} \,d\Hcal(z) \lesssim
    a^\alpha\,.
\]
\end{lem}
\begin{proof}
Let $x\in \Om$ and $\delta = \dist (x, \dOm)$. Then, there is $\tilde{x} \in \dOm$ with $\dd(x, \tilde{x}) < \min\{a, 3\delta/2\}$ such that $B(x,a) \subset B(\tilde{x}, 2a)$.
On one hand, we can estimate $\dd(x,z) \in [\delta, 7\delta/2)$ whenever $z \in \dOm \cap B(\tilde{x}, 2 \delta)$. On the other hand, for every $z \in \dOm \setminus B(\tilde{x}, 2 \delta)$, we have
\[
  \bigl| \dd(z,x) - \dd(z,\tilde{x}) \bigr| \le \dd(x, \tilde{x}) < \frac{3\delta}{2} < \frac{3}{4} \dd(z, \tilde{x}),
\]
which yields that
\[
  \frac14 \dd(z, \tilde{x}) \le \dd(z,x) \le \frac74 \dd(z, \tilde{x})\,.
\]
Consequently,
\begin{align*}
   & \int_{\dOm \cap B(x,a)} \frac{\dd(x,z)^\alpha }{ \Hcal(B(z, \dd(x,z))\cap \dOm)} \,d\Hcal(z) \\
  & \quad \lesssim \int_{\dOm \cap B(\tilde{x},2\delta)} \frac{\delta^\alpha}{ \Hcal(B(\tilde{x}, \delta)\cap \dOm)} d\Hcal(z)
  + \int_{\dOm \cap B(\tilde{x},2a) \setminus B(\tilde{x},2\delta)} \frac{\dd(z,\tilde{x})^\alpha}{ \Hcal(B(\tilde{x}, \dd(z,\tilde{x}))\cap \dOm)} d\Hcal(z).
\end{align*}
We can find $k_a, k_\delta \in \Zbb$ such that $\delta < 2^{k_\delta} \le 2\delta$ and $a\le 2^{k_a} < 2a$. For the sake of brevity, let $A_k = B(\tilde{x},2^{k+1}) \setminus B(\tilde{x},2^k)$, $k\in \Zbb$. Then,
\begin{align*}
  & \int_{\dOm \cap B(\tilde{x},2a) \setminus B(\tilde{x},2\delta)}  \frac{\dd(z,\tilde{x})^\alpha}{ \Hcal(B(\tilde{x}, \dd(z,\tilde{x}))\cap \dOm)} d\Hcal(z) \\
  &\qquad \lesssim \sum_{k=k_\delta}^{k_a} \int_{\dOm \cap A_k} \frac{2^{k\alpha}}{ \Hcal(B(z, \dd(x,2^k))\cap \dOm)} d\Hcal(z) \lesssim
  \sum_{k=k_\delta}^{k_a} 2^{k\alpha}\,.
\end{align*}
Altogether,
\[
  \int_{\dOm \cap B(x,a)} \frac{\dd(x,z)^\alpha }{ \Hcal(B(z, \dd(x,z))\cap \dOm)} \,d\Hcal(z) \lesssim \delta^\alpha + a^\alpha \lesssim a^\alpha\,.
  \qedhere
\]
\end{proof}
\begin{df}[{cf.\@ \cite[Section~9.1]{HajKos}}]
\label{df:john}
A domain $\Om \subset X$ is called a \emph{John domain} with \emph{John constant} $c_J\in (0, 1]$ and \emph{John center} at $y \in \Om$ if every $x \in \Om$ can be joined to $y$ by a rectifiable curve $\gamma: [0, l_\gamma] \to \Om$ parametrized by arc-length such that $\gamma(0)=x$, $\gamma(l_\gamma) = y$, and
\begin{equation}
  \label{eq:john}
  \dist(\gamma(t), X \setminus \Om) \ge c_J t \quad \text{for all }t \in [0, l_\gamma].
\end{equation}
\end{df}
Suppose that $\Om$ is a John domain with compact closure. Then, we can invoke the Ascoli theorem (cf.\@ \cite[p.\@ 169]{Roy}) to see that every $z \in \dOm$ can also be joined to the John center by a rectifiable curve such that \eqref{eq:john} is satisfied.
\begin{lem}
\label{lem:LebDifThm-diffBalls}
Let $u \in L^1_\loc(\Om)$, $z \in \dOm$, and $c>0$. Assume that there is $a \in \Rbb$ such that
\[
  \lim_{R \to 0} \fint_{B(z, R)\cap \Om} |u - a|\,d\mu = 0.
\]
For every sequence of points $\{x_k\}_{k=1}^\infty \subset \Om$ with $x_k \to z$ as $k\to\infty$ and every sequence of radii $\{r_k\}_{k=1}^\infty$ such that $r_k \to 0$ as $k\to \infty$ and $r_k \ge c \dd(x_k, z)$ for all $k=1,2,\ldots$, we then have
\[
  \lim_{k\to\infty} \fint_{B(x_k, r_k) \cap \Om} |u - a|\,d\mu = 0.
\]
\end{lem}
\begin{proof}
Using the doubling property of $\mu\lfloor_{\dOm}$, we can estimate
\begin{align*}
  \fint_{B(x_k, r_k) \cap \Om} | u-a|\,d\mu 
  & \le \frac{\mu\bigl(B(z, r_k+ \dd(x_k,z))\cap \Om\bigr)}{\mu\bigl(B(x_k, r_k)\cap \Om\bigr)} \fint_{B(z, r_k + \dd(x_k,z))\cap \Om} |u-a|\,d\mu \\
  & \lesssim \fint_{B(z, (1+1/c)r_k)\cap \Om} |u-a|\,d\mu \to 0 \quad\text{as }k\to\infty. \qedhere
\end{align*}
\end{proof}
\begin{thm}
\label{thm:TraceJohn}
Let $\Om$ be a John domain whose closure is compact. Suppose that $u \in P^{1,p}_q(\Om)$ for some $0<q<p$, where $p\ge 1$ and $p>\theta$. Then, the trace $Tu$ defined by \eqref{eq:defoftrace}, whose existence was established in Proposition~\ref{pro:TraceOp1}, lies in $\Bcal_p^{1-\theta/p}(\dOm) = B_{p,p}^{1-\theta/p}(\dOm)$.

Moreover, $\|Tu\|_{\Bcal_p^{1-\theta/p}(\dOm)} \le C (\|u\|_{L^p(\Om)} + \|{g}\|_{L^p(\Om)})$, where $g \in L^p(\Om)$ is a $q$-PI gradient of $u$.
\end{thm}
The restriction $p>\theta$ is not an artifact of the proof of the theorem above. In fact, we will show in Section \ref{sec:lp-extension} below that any $L^\theta(\dOm)$ boundary data can be extended to an $N^{1,\theta}(\Om)$ function. In particular, if $\Om$ admits a $\theta$-Poincar\'e inequality with some $1\le\theta$, then $N^{1,\theta}(\Om) = P^{1,\theta}_1(\Om)$, which shows that there is no chance to obtain any kind of smoothness of the trace in case $p=\theta>1$. The case of $p=\theta=1$ for $\Om$ admitting $1$-Poincar\'e inequality was covered in \cite{LahSha,MalShaSni}, where it was proven that the trace class of $N^{1,1}(\Om) = P^{1,1}_1(\Om)$ and of $BV(\Om)$ is $L^1(\dOm)$, which again does not allow for any Besov-type smoothness of the trace function.

The proof below would fail at one step if we allowed $q=p$. Namely, we will use the non-centered Hardy--Littlewood maximal operator $M$ to define an auxilliary function $h \coloneq (Mg^q)^{1/q}$. If $q<p$, then $\|h\|_{L^p} \approx \|g\|_{L^p}$, but we would have only a weak estimate $\|h\|_{L^{p,\infty}} \lesssim \|g\|_{L^p}$ in case $q=p$. However, if $\Om$ is a uniform domain, then usage of the maximal function can be circumvented, see Theorem~\ref{thm:traceUniform} below.
\begin{proof}
Let $u \in P^{1,p}_q(\Om)$ with a $q$-PI gradient $g \in L^p(\Om)$ and a dilation factor $\lambda \ge 1$ on the right-hand side of \eqref{eq:def-P1pq} be given.  
Let $\delta = \dist\{a, \dOm\} > 0$, where $a\in\Om$ denotes the John center of $\Om$. For a fixed pair of points $y, z \in \dOm$ with $\dd(y,z) \le \delta$, let $\gamma_y$ and $\gamma_z$ be the arc-length parametrized curves that connect the points $y$ and $z$, respectively, to the John center $a \in \Om$ and satisfy \eqref{eq:john}. 

Let $t_k = \dd(y,z) (1-\frac{c_J}{2\lambda})^k$ and $r_k = \frac{c_J}{2\lambda} t_k$ for $k=1,2,\ldots$. Next, we define a chain of balls $\{B_k\}_{k\in\Zbb}$ by setting
\[
   B_0=B(z, 3\dd(y,z)), \quad B_{-k} = B(\gamma_y(t_k), r_k),\text{ and }B_k = B(\gamma_z(t_k), r_k), \quad \text{for all }k=1,2,\ldots.
\]
For the sake of simpler notation, let $r_0 = 3\dd(y,z)$ and $r_{-k}=r_k$. The subchains $\{B_k\}_{k=1}^\infty$ and $\{B_{-k}\}_{k=1}^\infty$ consist of balls of bounded overlap where the upper bound on the number of overlapping balls depends on the John constant $c_J$. Moreover, abbreviating $\alpha \coloneq \bigl( 2 - \frac{c_J}{2\lambda} \bigr)$, we have $B_{k+1} \subset \alpha B_k$, which follows from the triangle inequality. Indeed, if $x \in B_{k+1}$, then
\[
  \dd\bigl(x, \gamma_z(t_k)\bigr) \le \dd\bigl(x, \gamma_z(t_{k+1})\bigr) + \dd\bigl(\gamma_z(t_{k+1}), \gamma_z(t_{k})\bigr) \le r_{k+1} + t_k - t_{k+1} = r_{k+1} + r_k = r_k \biggl(2-\frac{c_J}{2 \lambda}\biggr)\,.
\]
The triangle inequality also implies for all $k\ge 1$ and $x\in \alpha \lambda B_k$ that 
\[
  r_k = \frac{c_J t_k}{2\lambda } \le \frac{\dist(\gamma(t_k), \dOm)}{2\lambda} \le \frac{\dd(\gamma(t_k), z)}{2\lambda}  \le \frac{\dd(\gamma(t_k), x) + \dd(x, z)}{2\lambda} \le \frac{\alpha r_k}{2} + \frac{\dd(x,z)}{2\lambda}\,.
\]
Conversely,
\[
  \dd(x,z) \le \dd(x, \gamma(t_k)) + \dd(\gamma(t_k), z) \le \alpha \lambda r_k + t_k = \biggl(\alpha \lambda + \frac{2\lambda}{c_J}\biggr) r_k\,.
\]
Hence,
\begin{equation}
\label{eq:rk-vs-dxz}
\frac{\dd(x,z)}{\lambda(\alpha+2/c_J)}  \le r_k \le \frac{2\dd(x,z)}{c_J} \quad\text{if }x \in \alpha\lambda B_k.
\end{equation}
We can also use the triangle inequality to show that if $x\in \alpha \lambda B_k$, $k\ge 1$, then
\begin{align*}
  \dist(x, \dOm) & \ge \dist(\gamma(t_k), \dOm) - \alpha \lambda r_k \\ & \ge c_J t_k - \alpha \lambda r_k  = (2\lambda-\alpha\lambda) r_k 
   = \frac{c_J}{2} r_k \ge \frac{c_J\,\dd(x,z) }{2 \lambda(\alpha+2/c_J)} \eqcolon \beta \,\dd(x,z)
\end{align*}
and conversely,
\[
  \dist(x, \dOm) \le \dd(x,z) \le \lambda \biggl(\alpha+ \frac{2}{c_J}\biggr) r_k = \biggl(\frac{c_J \alpha}{2} + 1 \biggr) t_k \le 2 \dd(y,z)\,.
\]
An analogous argument can be carried out for the balls $B_k$ with $k\le -1$. Consequently, the chains of inflated balls $\{\alpha \lambda B_k\}_{k\ge1}$ and $\{\alpha \lambda B_{-k}\}_{k\ge1}$ are contained in ``John carrots''
\begin{align}
\label{eq:carrot}
  \bigcup_{k=1}^\infty \alpha \lambda B_k & \subset \biggl\{x\in\Om: \beta \dd(x,z) \le \dist(x,\dOm) \le 2\dd(z,y)\biggr\} \eqcolon C_{z,y} \quad \text{and} \\
  \notag
  \bigcup_{k=1}^\infty \alpha \lambda B_{-k} & \subset \biggl\{x\in\Om: \beta \dd(x,y) \le \dist(x,\dOm) \le 2 \dd(y,z)\biggr\} \eqcolon C_{y,z}
\end{align}
Lemma~\ref{lem:LebDifThm-diffBalls} yields that
\[
  Tu(y) = \lim_{k\to \infty} \fint_{B_{-k} \cap \Om} u\,d\mu \quad\text{and}\quad
  Tu(z) = \lim_{k\to \infty} \fint_{B_{k} \cap \Om} u\,d\mu.
\]
Since $p>\theta$, we can find $\eps > 0$ such that $p-\theta-\eps>0$. Using a chaining argument and the doubling property of $\mu$ similarly as before, we obtain that
\begin{equation}
  \label{eq:chaining0}
  |Tu(y) - Tu(z)| \lesssim \sum_{k\in\Zbb} \fint_{\Om \cap \alpha B_k} |u-u_{\alpha B_k}|\,d\mu \lesssim \sum_{k\in\Zbb} \alpha r_k \biggl(\fint_{\Om \cap \alpha \lambda B_k} g^q\,d\mu\biggr)^{1/q}\,.
\end{equation}
Note that $\alpha r_0 = C(c_J) \cdot r_1$. Let
\[
  h(x) = \sup_{B \ni x} \biggl(\fint_{\Om \cap  B} g^q\,d\mu\biggr)^{1/q} \eqcolon M_q g(x), \quad x\in \Om,
\]
be a non-centered maximal function of the Hardy--Littlewood type in $\Om$ for $g$. Since $\mu\lfloor_\Om$ is doubling, the maximal operator $M_q$ maps $L^q(\Om) \to \wk L^q(\Om)$, cf.\@ \cite[Theorem~III.2.1]{CoiWei}, as well as $L^\infty(\Om) \to L^\infty(\Om)$. Hence, $\|h\|_{L^p(\Om)} \approx \|g\|_{L^p(\Om)}$ by the Marcinkiewicz interpolation theorem (with a slightly more involved, yet fairly straightforward, reasoning in case $q<1$). Moreover,
\[
  \biggl(\fint_{\Om \cap \alpha \lambda B_0} g^q \,d\mu \biggr)^{1/q} \le \inf_{x\in \Om \cap \alpha \lambda B_0} h(x) \le \biggl(\fint_{\alpha \lambda B_1} h^q\,d\mu\biggr)^{1/q}
\]
since $\alpha \lambda B_1 \subset \Om \cap B_0 \subset \Om \cap \alpha \lambda B_0$. For the sake of brevity, let $\bar g = (g^q + h^q)^{1/q}$, in which case $\|\bar g\|_{L^p(\Om)} \approx \| g\|_{L^p(\Om)}$.
Then, \eqref{eq:chaining0} and the H\"older inequality yield that
\begin{align}
  \label{eq:chaining}
  |Tu(y) - Tu(z)|  \lesssim \sum_{k\in\Zbb\setminus\{0\}} r_k 
  \biggl(\fint_{\alpha \lambda B_k} \bar{g}^q\,d\mu\biggr)^{1/q}   
  \le \sum_{k\in\Zbb\setminus\{0\}} r_k^{1-(\theta+\eps)/p} r_k^{(\theta+\eps)/p} \biggl(\fint_{\alpha \lambda B_k} \bar{g}^p\,d\mu\biggr)^{1/p}&  \\
  \notag
   \lesssim \biggl(\sum_{k\in\Zbb\setminus\{0\}} r_k^{(p-\theta-\eps)/(p-1)}\biggr)^{1/p'} \biggl(\sum_{k\in\Zbb\setminus\{0\}} \frac{r_k^{\theta+\eps}}{\mu(B_k)} \int_{\alpha \lambda B_k} \bar{g}^p\,d\mu \biggr)^{1/p}&\,.
\end{align}
As the radii $r_k$ form a geometric sequence, we can compute the sum 
\[
  \biggl(\sum_{k\in\Zbb\setminus\{0\}} r_k^{(p-\theta-\eps)/(p-1)}\biggr)^{1/p'} \approx \bigl( \dd(y,z)^{(p-\theta-\eps)/(p-1)} \bigr)^{1/p'} = \dd(y,z)^{1-(\theta+\eps)/p}\,,
\]
where the constants depend on $p$, $\theta$, $\eps$, and $c_J$. It follows from \eqref{eq:rk-vs-dxz} that $\dd(\gamma_z(t_k), z) \approx r_k$ for $k\ge 1$ and hence $\mu(B_k) = \mu(B(\gamma_z(t_k), r_k)) \approx \mu(B(z, r_k) \cap \Om)$ by the doubling condition of $\mu\lfloor_{\dOm}$. Similarly, $\mu(B_{-k}) = \mu(B(\gamma_y(t_k), r_k)) \approx \mu(B(y, r_k)\cap\Om)$. Consequently, \eqref{eq:H-upper-massbound} yields that
\[
  \frac{r_k^\theta}{\mu(B_k)} 
  \lesssim \frac{1}{\Hcal(B(z, r_k)\cap \dOm)} \quad\text{and}\quad \frac{r_{-k}^\theta}{\mu(B_{-k})}\lesssim \frac{1}{\Hcal(B(y, r_{-k})\cap \dOm)}, \quad k\ge 1\,.
\]
As each of the subchains of balls $\{\alpha \lambda B_k\}_{k=1}^\infty$ and $\{\alpha \lambda B_{-k}\}^\infty_{k=1}$ has bounded overlap, we have
\begin{align*}
  \sum_{k=1}^\infty \frac{r_k^{\theta+\eps}}{\mu(B_k)} \int_{\alpha \lambda B_k} \bar{g}(x)^p\,d\mu(x)
  & \lesssim \sum_{k=1}^\infty  \int_{\alpha \lambda B_k} \frac{r_k^\eps \bar{g}(x)^p}{\Hcal(B(z,r_k)\cap \dOm)}\,d\mu(x) \\
  & \lesssim \int_{C_{z,y}} \frac{\dd(x,z)^\eps \bar{g}(x)^p}{\Hcal(B(z, \dd(x,z))\cap \dOm)}\,d\mu(x),
\end{align*}
where $C_{z,y}$ is the ``carrot'' \eqref{eq:carrot} whose tip is at $z \in \dOm$.

The trace $Tu$ lies  in $L^p(\dOm)$ by Theorem~\ref{thm:TraceOpB}. It suffices to estimate the Besov seminorm considering only pairs of points $(y,z)\in\dOm^2_\delta \coloneq \{(y,z) \in\dOm\times\dOm: \dd(y,z)<\delta\}$, where $\delta>0$ is the distance of the John center of $\Om$ from $\dOm$. For each such couple of points, we will apply \eqref{eq:chaining} and the subsequent estimates. Hence
\begin{align*}
   \iint_{\dOm^2_\delta} & \frac{|Tu(y) - Tu(z)|^p}{\dd(y,z)^{p-\theta} \Hcal(B(z, \dd(y,z))\cap \dOm)} \,d\Hcal^2(y,z)
 \\ 
 \tag{I} \lesssim& \iint_{\dOm^2_\delta} \frac{\dd(y,z)^{p-\theta-\eps}}{\dd(y,z)^{p-\theta} \Hcal(B(z, \dd(y,z))\cap \dOm)} \int_{C_{z,y}} \frac{\dd(x,z)^\eps \bar{g}(x)^p}{\Hcal(B(z, \dd(x,z))\cap \dOm)}\,d\mu(x)\,d\Hcal^2(y,z) \\
 \tag{II}  &\ + \iint_{\dOm^2_\delta} \frac{\dd(y,z)^{p-\theta-\eps}}{\dd(y,z)^{p-\theta} \Hcal(B(y, \dd(y,z))\cap \dOm)}  \int_{C_{y,z}} \frac{\dd(x,y)^\eps \bar{g}(x)^p}{\Hcal(B(y, \dd(x,y))\cap \dOm)}\,d\mu(x)\,d\Hcal^2(y,z) \\
\end{align*}
Let us now examine the first of the two integrals. The integration domain is contained in the set
\[
  \bigl\{(x,y,z) \in \Om\times \dOm \times \dOm: \beta \dd(x,z) \le  \dist(x,\dOm) \le 2\dd(y,z) < \delta\bigr\}.
\]
By the Fubini theorem, we change the order of integration and then extend the integration domain so that
\[
  I \le \int_{\{x\in\Om: \dist(x, \dOm) < \delta\}} \int_{\{z\in \dOm \cap B(x, \dist(x, \dOm)/\beta\}} \int_{\{y \in \dOm \setminus B(z, \dist(x, \dOm)/2)\}}\bigl(\cdots\bigr) d\Hcal(y)\,d\Hcal(z)\,d\mu(x)\,.
\]
We apply Lemma~\ref{lem:trace-int-est1} for the innermost integral and then we proceed by Lemma~\ref{lem:trace-int-est2}, which yields
\begin{align*}
  I & \le \int_{\Om_\delta} \bar{g}(x)^p \biggl(\int_{\dOm \cap B(x, \dist(x,\dOm)/c)} \frac{\dd(x,z)^\eps }{\Hcal(B(z, \dd(x,z))\cap \dOm)}\cdot \frac{1}{\dist(x,\dOm)^\eps} \,d\Hcal(z)\biggr)\,d\mu(x) \\
  & \le \int_{\Om_\delta} \bar{g}(x)^p \frac{\dist(x,\dOm)^\eps}{\dist(x,\dOm)^\eps}\,d\mu(x) \le \|\bar{g}\|_{L^p(\Om)}^p,
\end{align*}
where $\Om_\delta = \{ x\in\Om: \dist(x, \dOm) < \delta\}$ is the $\delta$-strip in $\Om$ near the boundary. Integral (II) can be processed analogously as (I) by mere flipping order of integration with respect to $y$ and $z$.
\end{proof}
As noted after the theorem's statement above, the Hardy--Littlewood maximal function we used in the proof fails to be in $L^p(\Om)$ if we allow $q=p$. However, we will obtain the desired critical Besov smoothness of the trace under the additional assumption that $\Om$ is a uniform domain. 
\begin{df}
\label{df:uniform}
A domain $\Om \subset X$ is called \emph{uniform} if there is a constant $c_U \in (0, 1]$ such that every pair of distinct points $x,y\in\Om$ can be connected by a rectifiable curve $\gamma: [0, l_\gamma] \to \Om$ parametrized by arc-length such that $\gamma(0) = x$, $\gamma(l_\gamma) = y$, $l_\gamma \le c_U^{-1} \dd(x,y)$, and
\begin{equation}
  \label{eq:uniform}
  \dist(\gamma(t), X \setminus \Om) \ge c_U \min\{ t, l_\gamma-t\} \quad \text{for all }t\in[0, l_\gamma].
\end{equation}
\end{df}
Similarly as for John domains, one can use the Ascoli theorem (cf.\@ \cite[p.\@ 169]{Roy}) to see that if $\Om$ is uniform with compact closure, then every pair of distinct boundary points $x,y \in \dOm$ can be connected by a curve $\gamma$ of length $l_\gamma \le c_U^{-1} \dd(x,y)$ such that \eqref{eq:uniform} is satisfied.
\begin{thm}
\label{thm:traceUniform}
Let $\Om$ be a uniform domain whose closure is compact. Suppose that $u \in P^{1,p}_p(\Om)$ for some $p \ge 1$ while $\theta < p$. Then, the trace $Tu \in L^p(\dOm)$ lies in $\Bcal_p^{1-\theta/p}(\dOm) = B_{p,p}^{1-\theta/p}(\dOm)$. Moreover, $\|Tu\|_{\dot\Bcal_p^{1-\theta/p}(\dOm)} \le C \|g\|_{L^p(\Om)}$, where $g\in L^p(\Om)$ is a $p$-PI gradient of $u$.
\end{thm}
\begin{proof}[Sketch of proof]
The proof can be led along the same lines as for Theorem~\ref{thm:TraceJohn} with a slightly modified construction of chains of balls connecting pairs of boundary points, using the fact that $\Om$ is uniform, which will yield that a Hardy--Littlewood-type maximal function is no longer needed. 

For a fixed couple of points $y,z \in \dOm$, let $\gamma: [0, l_\gamma] \to \overline{\Om}$ be a curve that connects the two points, i.e., $\gamma(0)=z$ and $\gamma(l_\gamma) = y$, and satisfies the conditions from the definition of a uniform domain.
Let
\[
  t_k = \frac{l_\gamma}{2} \Bigl(1-\frac{c_U}{2\lambda}\Bigr)^k\quad \text{ and }\quad r_k = \frac{c_U t_k}{2\lambda}, \quad k=0,1,2,\ldots.
\]
Next, we define a chain of balls $\{B_k\}_{k\in\Zbb}$ by setting
\[
   B_{k} = B(\gamma(t_k), r_k)\quad\text{ and }\quad B_{-k} = B(\gamma(l_\gamma - t_k), r_k), \quad \text{for all }k=0,1,2,\ldots.
\]
With this choice of balls, we obtain that the chains of inflated balls (including $\alpha \lambda B_0$) are contained in the carrots
\begin{align*}
  \bigcup_{k=0}^\infty \alpha \lambda B_k & \subset \biggl\{x\in\Om: \beta \dd(x,z) \le \dist(x,\dOm) \le 2\dd(z,y)\biggr\} \eqcolon C_{z,y} \quad \text{and} \\
  \bigcup_{k=0}^\infty \alpha \lambda B_{-k} & \subset \biggl\{x\in\Om: \beta \dd(x,y) \le \dist(x,\dOm) \le 2 \dd(y,z)\biggr\} \eqcolon C_{y,z}\,.
\end{align*}
Hence $\alpha \lambda B_0$ does not need any special treatment via the maximal function and one can proceed as in the proof of Theorem~\ref{thm:TraceJohn}.
\end{proof}
\begin{cor}
Under the hypothesis of Theorems~\ref{thm:TraceJohn} or~\ref{thm:traceUniform} and assuming that $\Hcal$ is Ahlfors codimension-$\theta$ regular, we have that
\begin{align*}
  \|Tu\|_{\dot{B}^\alpha_{p_\alpha, q}(\dOm)} & \le C \|g\|_{L^p(\Om)}\quad \text{for every $0 < \alpha \le 1-\theta/p$, $q\in[p,\infty]$, and} \\
  \inf_{a\in\Rbb} \|Tu - a\|_{L^{p^*,p}(\dOm)} & \le \inf_{a\in\Rbb} \|Tu - a\|_{L^{p^*}(\dOm)} \le C \|g\|_{L^p(\Om)},
\end{align*}
where $p_\alpha = \frac{p(s-\theta)}{s-p(1-\alpha)}$ and 
$p^* = \frac{p(s-\theta)}{s-p}$.
\end{cor}
\begin{proof}
Both estimates follow from the embeddings in Theorem~\ref{thm:Besov-embedding-closedend} and Corollary~\ref{cor:Besov-standardembeddings}. 
\end{proof}
\section{\texorpdfstring{Bounded linear extension from the Besov class $B^{1-\vartheta/p}_{p,p}(\dOm)$\\to the Newtonian class $N^{1,p}(\Om)$}{Bounded linear extension from the Besov class B\^{}\{1-\unichar{"03D1}/p\}\_{}\{p,p\}(\unichar{"2202}\unichar{"03A9}) to the Newtonian class N\^{}\{1,p\}(\unichar{"03A9})}}
\label{sec:Besov-extension}
In this section, we will find a linear extension operator that maps the Besov class $B^{1-\vartheta/p}(\dOm)$ into the class of locally Lipschitz Newtonian functions $N^{1,p}(\Om)$. The extension will be defined using a Whitney-type cover of $\Om$ and the corresponding partition of unity.

For the extension theorem, we will assume that the relation between $\mu\lfloor_\Om$ and $\Hcal\lfloor_{\dOm}$ satisfies both \eqref{eq:H-lower-massbound} and \eqref{eq:H-upper-massbound} with $0 < \vartheta \le \theta \le p$. In fact, it suffices to assume that \eqref{eq:H-upper-massbound} is localized, i.e., 
\[
  \Hcal(B(z,r) \cap \dOm) \le C(z) \frac{\mu(B(z,r)\cap \Om)}{r^\theta}
\]
holds for $\Hcal$-a.e.\@ $z\in \dOm$ and for all $r<R(z)$, where $C(z)>0$ and $R(z)>0$. Such a localization would affect the proofs of Lemmata~\ref{lem:layer-est-Fn-boundaryball}, \ref{lem:trace-Lebesgue}, and~\ref{lem:trace-Besov}, making the argument somewhat obscured by technical details, whence their formulation and proof will use the non-localized form \eqref{eq:H-upper-massbound}.

Let us first establish the \emph{Whitney covering} of an open set $\Om$ as in \cite[Section 4.1]{HKST}, see also \cite{BjoBjoSha07}.
\begin{pro}[{\cite[Proposition~4.1.15]{HKST}}]
Let $\Omega\subsetneq X$ be open. Then there exists a countable collection $\mathcal{W}_\Omega=\{B(p_{j, i},r_{j,i})=B_{j,i}\}$ of balls in $\Omega$ so that
\begin{itemize}
  \setlength{\parskip}{0pt}
  \setlength{\itemsep}{2pt plus 1pt minus 2pt}
 \item $\bigcup_{j,i} B_{j,i} = \Omega$,
 \item $\sum_{j,i} \chi_{B(p_{j,i}, 2r_{j,i})}\leq C$, 
 \item $2^{j-1}<r_{j,i}\leq 2^{j}$ for all $i$, and
 \item $r_{j,i}=\frac{1}{8}\dist(p_{j,i}, X\setminus \Omega)$,
\end{itemize}
where $C>0$ depends only on the doubling constant of the measure $\mu$.
\end{pro}
Since the radii of the balls are sufficiently small, we have that $2B_{j,i}\subset\Omega$ for every $B_{j,i}\in\Wcal_\Om$. By the boundedness of $\Omega$ there is a largest exponent $j$ that occurs in the cover; we denote this exponent by $j_0$. Note that $2^{j_0}$ is comparable to $\diam(\Om)$. 

Note also that no ball in level $j$ intersects a ball in level $j+2$.  This follows by the reverse triangle inequality
$\dd(p_{j,i},p_{j+2,k})\geq 2^{j+4}-2^{j+3}=2^{j+3}$
and the bounds on the radii: $2^{j-1}<r_{j,i}\leq 2^j$ and $2^{j+1}<r_{j+2,k}\leq 2^{j+2}$.

As in~\cite[Theorem~4.1.21]{HKST}, there is a Lipschitz partition of unity $\{\vphi_{j,i}\}$ subordinate to the Whitney decomposition $\Wcal_\Om$, that is, $\sum_{j,i} \vphi_{j,i} \equiv \chi_\Om$ and $0 \le \vphi_{j,i} \le \chi_{2B_{j,i}}$ for every ball $B_{j,i}\in\Wcal_\Om$ while $\vphi_{j,i}$ is $C/r_{j,i}$-Lipschitz continuous.
\subsection{Extending boundary data via Whitney-type partition of unity}
Given $f \in L^1(\dOm)$, we will construct a function $F:\Omega\to\Rbb$ whose trace will (under certain hypothesis, cf.\@ Lemmata~\ref{lem:trace-Lebesgue} and~\ref{lem:trace-Besov}) be the original function $f$ on $\partial\Omega$. Moreover, we will show that $F \in N^{1,p}(\Om)$ provided that $f\in B_{p,p}^{1-\vartheta/p}(\partial\Omega)$.

Consider the center of the Whitney ball $p_{j,i}\in\Omega$ and choose a closest point $q_{j,i}\in\partial\Omega$.
Define $U_{j,i}:=B(q_{j,i}, r_{j,i})\cap \partial\Omega$.
We set $a_{j,i}:=\fint_{U_{j,i}}f(y)\,d\Hcal(y)$.
Then, set
\[
  F(x) = Ef(x) \coloneq \sum_{j,i} a_{j,i} \vphi_{j,i}(x), \quad x\in \Om.
\]
We will show that $E: L^p(\dOm) \to L^p(\Om)$ is bounded for every $p\ge 1$ in Lemma~\ref{lem:Lp-est_Whitney}. In Proposition~\ref{pro:extnBounds}, we will see that $E: B^{1-\vartheta/p}_{p,p}(\dOm) \to N^{1,p}(\Om) \cap \Lip_\loc(\Om)$ provided that $p\ge \max\{1,\vartheta\}$. Finally, Lemmata~\ref{lem:trace-Lebesgue} and~\ref{lem:trace-Besov} show that $Ef$ has a trace on $\dOm$ in the sense of \eqref{eq:defoftrace} and the trace coincides with the given function $f$ if either $\vartheta = \theta$ and $f \in L^p(\dOm)$ with any $p\ge 1$, or $f\in B^{1-\vartheta/p}_{p,p}(\dOm)$ and $p>\theta$.

Let us first establish the $L^p$-estimates for $F=Ef$.
\begin{lem}
\label{lem:Lp-est_Whitney}
Let $p\ge 1$. Then, 
\[
\|F\|_{L^p(\Om)} \lesssim \diam(\Om)^{\vartheta/p} \|f\|_{L^p(\dOm)}.\]
\end{lem}
\begin{proof}
We first consider a fixed ball $B_{\ell, m}$ from the Whitney cover. Recall that if $2B_{j,i}\cap B_{\ell,m}\neq \emptyset$, then
$|j-\ell|\le 1$. Thus, $U_{j,i}\subset U_{\ell,m}^*\coloneq B(q_{\ell,m}, 2^{6}r_{\ell,m})\cap\partial\Omega$, whence $\Hcal(U_{j,i})\approx \Hcal(U_{\ell,m}^*)$. Then,
\begin{align*}
\int_{B_{\ell, m}}|F(x)|^p\,d\mu(x) &=\int_{B_{\ell, m}}\biggl|\sum_{\substack{j,i \, \textrm{ s.t.}\\
    2B_{j,i}\cap B_{\ell,m}\neq \emptyset}} \fint_{U_{j,i}}f(y)\,d\Hcal(y)\vphi_{j,i}(x)\biggr|^p\,d\mu(x)\\
&\leq\int_{B_{\ell, m}}\biggl(\sum_{\substack{j,i \, \textrm{ s.t.}\\
    2B_{j,i}\cap B_{\ell,m}\neq \emptyset}}\fint_{U_{j,i}}|f(y)|\,d\Hcal(y)\,\vphi_{j,i}(x)\biggr)^p\,d\mu(x)\\
&\lesssim \int_{B_{\ell, m}}\biggl(\fint_{U^*_{\ell,m}}|f(y)|\,d\Hcal(y) \sum_{j,i}\vphi_{j,i}(x)\biggr)^p\,d\mu(x) \\
& = \int_{B_{\ell, m}}\biggl(\fint_{U^*_{\ell,m}}|f(y)|\,d\Hcal(y)\biggr)^p\,d\mu(x).  
\end{align*}
From the construction of the Whitney balls, the doubling condition of $\mu\lfloor_\Om$, and \eqref{eq:H-lower-massbound}, we deduce that
\begin{equation}
  \label{eq:muBlm-HUlm}
  \mu(B_{\ell,m}) \approx \mu(B(q_{\ell,m},r_{\ell,m}) \cap \Om) \lesssim r_{\ell,m}^\vartheta \Hcal(U_{\ell,m}).
\end{equation}
Hence,
\begin{equation}\label{eq:Ball-F-est}
\int_{B_{\ell, m}}|F(x)|^p\,d\mu(x) \lesssim \mu(B_{\ell,m}) \fint_{U^*_{\ell,m}} |f(y)|^p\,d\Hcal(y) \lesssim \,r^\vartheta_{\ell,m} \int_{U_{\ell,m}^*}|f(y)|^p\,d\Hcal(y).
\end{equation}
Recall that $r_{\ell, m}\approx 2^{\ell}$. As $\Omega=\bigcup_{\ell,m}B_{\ell,m}$ and the balls have uniformly bounded overlap, we have
\begin{align*}
  \int_\Omega |F|^p \,d\mu 
      & \lesssim \sum_{\ell,m} r_{\ell,m}^\vartheta \int_{U_{\ell,m}^*}\left|f\right|^p\,d\Hcal
      \lesssim \sum_{\ell=-\infty}^{j_0}2^{\ell\,\vartheta} \sum_{m} \int_{U_{\ell,m}^*}\left|f\right|^p\,d\Hcal \\
      & \lesssim \sum_{\ell=-\infty}^{j_0}2^{\ell\,\vartheta}  \int_{\partial\Omega}\left|f\right|^p\,d\Hcal
      \approx \diam(\Om)^\vartheta \|f\|_{L^p(\partial\Omega)}^p.
  \qedhere  
\end{align*}
\end{proof}
We will use the extension constructed in this section to find a nonlinear bounded extension from $L^p(\dOm)$
to $N^{1,p}(\Om)$ in the subsequent section in case $p=\vartheta=\theta \ge 1$. There, we will need the following estimates for the integral of the function $F$ and its gradient (Lemma~\ref{lem:Lip-layer-est-grad} below) on boundary shells of $\Omega$.
\begin{lem}
\label{lem:layer-est-Fn-boundaryball}
Let $z \in \dOm$, $r > 0$, and $0< \rho <\diam(\Om)/2$. Set $\Om(\rho) = \{x\in\Om : \dist(x,X\setminus\Om)<\rho\}$. Then,
\[
\int_{B(z,r)\cap \Om(\rho)}|F|^p\, d\mu\lesssim \min\{r, \rho\}^\vartheta \int_{B(z, 2^8r) \cap \dOm}|f|^p\, d\Hcal\,.
\]
\end{lem}
\begin{proof}
Since $B(z,r)\cap \Om(\rho) = B(z,r)\cap \Om(\min\{r,\rho\})$, we do not lose any generality by 
assuming that $\rho \le r$. Let  
$\ell_\rho$ be the greatest value of $\ell \in \Zbb$ for which there exists a ball $B_{\ell,m}$ that intersects $\Omega(\rho)\cap B(z,r)$. For each $\ell \le \ell_\rho$, we define $\mathcal{I}(\ell)$ to be the collection of all $m\in\Nbb$ for which $B_{\ell,m}\cap\Omega(\rho) \cap B(z,r)$ is non-empty. Then by~\eqref{eq:Ball-F-est},
\[
\int_{B(z,r) \cap \Om(\rho)} |F|^p\, d\mu
   \le \sum_{\ell=-\infty}^{\ell_\rho} \sum_{m\in\mathcal{I}(\ell)}\int_{B_{\ell,m}}|F|^p\, d\mu
   \lesssim \sum_{\ell=-\infty}^{\ell_\rho} \sum_{m\in\mathcal{I}(\ell)} r_{\ell,m}^\vartheta\int_{U^*_{\ell,m}}|f|^p\, d\Hcal.
\]
The triangle inequality yields that
\[
  \dd(z, q_{\ell, m}) \le \dd(z, p_{\ell,m}) + \dd(p_{\ell,m}, q_{\ell, m}) \le 2 \dd(z, p_{\ell,m}) \le 2(r + r_{\ell,m}),
\]
where $B_{\ell, m} = B(p_{\ell,m}, r_{\ell,m})$ and $U_{\ell, m} = B(q_{\ell,m}, r_{\ell,m}) \cap \dOm$ with 
$q_{\ell,m} \in \dOm$ being a boundary point lying closest to $p_{\ell,m}$. Moreover, 
$8 r_{\ell,m} = \dist(p_{\ell,m}, X\setminus \Om) \le \dd(p_{\ell,m}, z) \le r+ r_{\ell,m}$. Hence, 
$r_{\ell,m} \le \tfrac17 r$. Consequently, $\dd(z, q_{\ell,m}) \le \tfrac{16}{7} r$ and 
$U_{\ell, m} \subset B(z, (\tfrac{16}{7} + \tfrac17)r)$. Thus, $U^*_{\ell, m} \subset B(z, 2^8 r)$ and
\[
  \int_{B(z,r) \cap\Om(\rho) } |F|^p\, d\mu 
  \lesssim \sum_{\ell=-\infty}^{\ell_\rho} 2^{\ell\vartheta} \int_{B(z, 2^8 r) \cap \dOm}|f|^p\, d\Hcal 
  \approx \rho^\vartheta \int_{B(z, 2^8 r) \cap \dOm}|f|^p\, d\Hcal,
\]
where we used that $r_{\ell,m}\approx 2^\ell$ and $2^{\ell_\rho} \approx \rho$, which follow from the construction of $\Wcal_\Om$.
\end{proof}
By choosing any $r>\diam \Om$ in the preceding lemma, we obtain the following corollary.
\begin{cor}
\label{cor:layer-est-Fn}
With the notation of Lemma~\ref{lem:layer-est-Fn-boundaryball}, we have
\[
\int_{\Om(\rho)}|F|^p\, d\mu\lesssim \rho^\vartheta\, \int_{\dOm}|f|^p\, d\Hcal.
\]
\end{cor}
Since $\Lip F$ is an upper gradient of locally Lipschitz functions, cf.~\cite[Proposition~1.14]{BjoBjo}, the following proposition provides us with the desired $L^p$-norm bound for an upper gradient of $F = Ef$.
\begin{pro}
\label{pro:extnBounds}
Given $\Omega\subset X$ and $p\ge \max\{1, \vartheta\}$, then
\[
  \| \Lip F \|_{L^p(\Om)} \lesssim \|f\|_{B_{p,p}^{1-\vartheta/p}(\partial\Omega)} \quad \text{whenever }f\in B_{p,p}^{1-\vartheta/p}(\partial\Omega).
\]
\end{pro}
\begin{proof}
  Fix a ball $B_{\ell,m}\in \Wcal_\Om$, and fix a point $x \in B_{\ell,m}$.  Then, for all $y\in B_{\ell,m}$,
  \begin{align*}
    |F(y)-F(x)| &= \biggl|\sum_{j,i}a_{j,i}(\vphi_{j,i}(y)-\vphi_{j,i}(x))\biggr| \\
     & = \biggl|\sum_{j,i}(a_{j,i}-a_{\ell,m})(\vphi_{j,i}(y)-\vphi_{j,i}(x))\biggr| 
    \lesssim \sum_{\substack{j,i\, \textrm{ s.t.}\\ 2B_{j,i}\cap B_{\ell,m}\ne\emptyset}}|a_{j,i}-a_{\ell,m}|\frac{\dd(y,x)}{r_{j,i}},
  \end{align*}
  where the last inequality follows from the Lipschitz constant of $\vphi_{j,i}$.  Hence, 
  \[
  \frac{|F(y)-F(x)|}{\dd(y,x)}\leq \frac{C}{r_{\ell,m}} \sum_{\substack{j,i \,\textrm{ s.t.} \\
    2B_{j,i}\cap B_{\ell,m}\neq \emptyset}}|a_{j,i}-a_{\ell,m}|.
  \]
  Thus, we want to bound terms of the form $|a_{j,i}-a_{\ell,m}|$. Note also that if $2B_{j,i} \cap B_{\ell,m} \neq \emptyset$, then $|j-\ell|\le 1$. Therefore,
\begin{align}
  |a_{j,i} -a_{\ell,m}| = \left|\fint_{U_{j,i}}f(z)\,d\Hcal(z)-\fint_{U_{\ell,m}}f(w)\,d\Hcal(w)\right|&\notag \\
  \label{eq:expandballs1}
  \le \fint_{U_{j,i}}\fint_{U_{\ell,m}}  \left|f(z)-f(w)\right|\,d\Hcal(w)\,d\Hcal(z)\lesssim & \fint_{U^*_{\ell,m}}\fint_{U^*_{\ell,m}}  \left|f(z)-f(w)\right|\,d\Hcal(w)\,d\Hcal(z),
\end{align}
where $U^*_{\ell,m} \coloneq B(q_{\ell,m}, 2^{6}r_{\ell,m}) \cap \dOm$ denotes the expanded subset of the boundary. Recall that $\Hcal$ is doubling by \eqref{eq:H-upper-massbound} and by the assumption that $\mu\lfloor_\Om$ is doubling, and that is why $\Hcal(U^*_{\ell,m})\lesssim \Hcal(U_{\ell,m})$ and $\Hcal(U^*_{\ell,m})\lesssim \Hcal(U_{j,i})$,
which was used in~\eqref{eq:expandballs1}. The above estimates together with the bounded overlap of the
Whitney balls yield the following inequality:
\begin{equation}
\label{eq:pointwise-Lip-est}
\Lip F(x) =\limsup_{y\to x}\frac{|F(y)-F(x)|}{\dd(y,x)}   \lesssim \frac{1}{r_{\ell,m}}  \fint_{U^*_{\ell,m}}\fint_{U^*_{\ell,m}}  \left|f(z)-f(w)\right|\,d\Hcal(w)\,d\Hcal(z)
\end{equation}
whenever $x\in B_{\ell,m}$.
Therefore, $\|\Lip F\|_{L^p(\Om)}^p \le \sum_{\ell,m} \|\Lip F\|_{L^p(B_{\ell,m})}^p$ and hence \eqref{eq:pointwise-Lip-est} and \eqref{eq:muBlm-HUlm} yield that
\begin{align*}
 \int_\Omega (\Lip F)^p\,d\mu  
      &\lesssim \sum_{\ell, m}\, \frac{\mu(B_{\ell,m})}{r^p_{\ell,m}}\,
          \biggl(\fint_{U^*_{\ell,m}}\fint_{U^*_{\ell,m}}  \left|f(z)-f(w)\right|\,d\Hcal(w)\,d\Hcal(z)\biggr)^p\notag\\
  &\lesssim \sum_{\ell,m} \frac{\Hcal(U_{\ell,m})}{r^{p-\vartheta}_{\ell,m}}
 \fint_{U^*_{\ell,m}}\fint_{U^*_{\ell,m}}  \left|f(z)-f(w)\right|^p\,d\Hcal(w)\,d\Hcal(z)\\
 &\lesssim\sum_{\ell=-\infty}^{j_0}\sum_{m} \frac{1}{r^{p-\vartheta}_{\ell,m}}
 \int_{U^*_{\ell,m}}\fint_{U^*_{\ell,m}}  \left|f(z)-f(w)\right|^p\,d\Hcal(w)\,d\Hcal(z)\\
 &\lesssim \sum_{\ell=-\infty}^{j_0} \frac{1}{(2^{7+\ell})^{p-\vartheta}}
 \int_{\partial\Omega}\fint_{B(z,2^{7+\ell})}  \left|f(z)-f(w)\right|^p\,d\Hcal(w)\,d\Hcal(z)\,,
\end{align*}
where the last inequality follows from the uniformly bounded overlap of the balls $U_{\ell,m}^*$ for each $\ell$.
Without loss of generality, we may choose $R = 2^{j_0+7}$ in the definition of the Besov norm \eqref{eq:Besov}, in which case $R\approx \diam(\Omega)$. Finally,
\begin{align*} 
\sum_{\ell=-\infty}^{j_0}& \frac{1}{(2^{7+\ell})^{p-\vartheta}}
 \int_{\partial\Omega}\fint_{B(z,2^{7+\ell})}  \left|f(z)-f(w)\right|^p\,d\Hcal(w)\,d\Hcal(z)\\
 & \approx\int_{t=0}^{R} \frac{1}{t^{p-\vartheta}}
   \int_{\partial\Omega}\fint_{B(z,t)}  \left|f(z)-f(w)\right|^p\,d\Hcal(w)\,d\Hcal(z)\frac{dt}{t} = \|f\|_{\dot{B}_{p,p}^{1-\vartheta/p}(\partial\Omega)}\,.  \\[-6.3ex] & \phantom{\frac{1}{t}}\qedhere
\end{align*}
\end{proof}
Next, we obtain a localized estimate for the $L^p$-norm of the $p$-weak upper gradient $\Lip F$ on the layer $\Om(\rho)$ if the boundary function $f$ Lipschitz is continuous. The Lipschitz constant of $f$ on $\dOm$ will be denoted by $\LIP(f, \dOm)$, i.e., we define
\[
  \LIP (f, \dOm) = \sup_{x,y\in \dOm : x\neq y} \frac{|f(x)-f(y)|}{\dd(x,y)}.
\]
\begin{lem}
\label{lem:Lip-layer-est-grad}
For $0< \rho <\diam(\Om)/2$, set $\Om(\rho) = \{x\in\Om : \dist(x,X\setminus\Om)<\rho\}$. If $f$ is Lipschitz continuous on $\dOm$, then
\[
  \int_{\Om(\rho)}(\Lip F(x))^p\, d\mu(x) \lesssim \mu(\Om(\rho)) \LIP(f, \dOm)^p\,.
\]
\end{lem}
\begin{proof}
It follows from \eqref{eq:pointwise-Lip-est} that $\Lip F(x) \lesssim \LIP(f, \dOm)$ for all $x\in \Om$.
\end{proof}
So far, we have seen that if $f\in B^{1-\vartheta/p}_{p,p}(\dOm)$, the corresponding extension $F$ is in the Newtonian space $N^{1,p}(\Om)$. This extension is linear by construction.
We will now prove that the trace of $F$ returns the original function $f$, i.e., $T\circ E$ is the identity function on $B^{1-\vartheta/p}_{p,p}(\dOm)$ whenever $\vartheta = \theta$ or $p> \theta$. 
For the sake of clarity, let us explicitly point out that the following lemma shows that the $N^{1,p}$ extension has a well-defined trace even though no Poincar\'e inequality for $\Om$ is assumed.
\begin{lem}\label{lem:trace-Lebesgue}
  Assume that $\vartheta = \theta$. Let $f \in L^p_\loc(\dOm)$, $p\ge 1$. Then,
\[
 \lim_{r\to 0^+} \fint_{B(z,r)\cap\Omega}|Ef(x)-f(z)|^p\,d\mu(x)=0\quad \text{for $\Hcal$-a.e.\@ $z\in \partial\Omega$.}
\]
  In particular, the trace $T(Ef)(z)$ exists and equals $f(z)$ for $\Hcal$-a.e.~$z\in\dOm$.
\end{lem}
\begin{proof}
Since $\Hcal\lfloor_{\dOm}$ is doubling, we know that
$\Hcal$-a.e.\@ $z\in\dOm$ is a Lebesgue point of $f$.
Let $z\in\dOm$ be such a point. Let $f_z(w) = f(w) - f(z)$, $w\in \dOm$. Note that $E f_z(x) = Ef(x) - f(z)$ for every $x\in\dOm$ by linearity of the extension operator $E$.
From Lemma~\ref{lem:layer-est-Fn-boundaryball} applied to $E f_z$ with $r=\rho$, we obtain that
\[
  \int_{B(z,r)\cap\Om} |Ef(x) - f(z)|^p \,d\mu(x) = \int_{B(z,2^8 r) \cap \Om} |Ef_z(x)|^p \,d\mu(x) \lesssim r^{\vartheta} \int_{B(z,2^8 r) \cap \dOm} |f_z(w)|^p \,d\Hcal(w).
\]
In particular,
\begin{align*}
  \fint_{B(z,r)\cap\Om} |Ef(x) - f(z)|^p \,d\mu(x) & \lesssim \frac{r^{\vartheta}}{\mu(B(z,r)\cap\Om)} \int_{B(z,r) \cap \dOm} |f_z(w)|^p \,d\Hcal(w)\\
  &\approx \fint_{B(z,r) \cap \dOm} |f(w) - f(z)|^p \,d\Hcal(w)
\end{align*}
by \eqref{eq:H-upper-massbound}. As $z\in\dOm$ is a Lebesgue point of $f$, letting $r\to 0$ concludes the proof.
\end{proof}
\begin{lem}\label{lem:trace-Besov}
 Let $f \in B^{\alpha}_{p,p}(\dOm)$ with $\alpha \in (0, 1)$ and $p\ge 1$ such that $\alpha p > \theta - \vartheta$. Then,
\[
 \lim_{r\to 0^+} \fint_{B(z,r)\cap\Omega}|Ef(x)-f(z)|^p\,d\mu(x)=0\quad \text{for $\Hcal$-a.e.\@ $z\in \partial\Omega$.}
\]
In particular, the trace $T(Ef)(z)$ exists and equals $f(z)$ for $\Hcal$-a.e.~$z\in\dOm$ if $f \in B^{1-\vartheta/p}_{p,p}(\dOm)$ and $0 < \vartheta \le \theta < p$.
\end{lem}
\begin{proof}
By Lemma~\ref{lem:BesovHaj}, there is $g \in L^p(\dOm)$ and $E \subset \dOm$ with $\Hcal(E) = 0$ such that $|f(z) - f(w)| \le \dd(z,w)^{\alpha} (g(z) + g(w))$ for all $z,w \in \dOm \setminus E$. As $\Hcal\lfloor_{\dOm}$ is doubling, the Hardy--Littlewood maximal operator $M: L^1(\dOm) \to \wk L^1(\dOm)$ is bounded by~\cite[Theorem~III.2.1]{CoiWei}.

Pick $z \in \dOm \setminus E$ such that $g(z)<\infty$ and $M g^p (z) < \infty$, which is satisfied for $\Hcal$-a.e.\@ $z\in\dOm$ as $Mg^p \in \wk L^1(\dOm)$. Then, following the steps of proof of Lemma~\ref{lem:trace-Lebesgue} to its ultimate displayed formula, we arrive at
\begin{align*}
  \fint_{B(z,r)\cap\Om} |Ef(x) - f(z)|^p \,d\mu(x) & \lesssim r^{\vartheta - \theta} \fint_{B(z,r) \cap \dOm} |f(w) - f(z)|^p \,d\Hcal(w) \\
  & \lesssim r^{\vartheta - \theta} \fint_{B(z,r) \cap \dOm} \dd(z,w)^{\alpha p} (g(z)^p + g(w)^p)\,d\Hcal(w) \\
  & \le r^{\vartheta - \theta + \alpha p} (g(z)^p + Mg^p(z)),
\end{align*}
which approaches $0$ as $r\to 0$.
\end{proof}
\section{\texorpdfstring{Extension theorem for $L^p$ data on a boundary of codimension $p$}{Extension theorem for L\^{}p data on a boundary of codimension p}}
\label{sec:lp-extension}
In this section, we assume that $\dOm$ is Ahlfors codimension-$p$ regular, i.e., both \eqref{eq:H-lower-massbound} and \eqref{eq:H-upper-massbound} are satisfied and $\vartheta = \theta = p \ge 1$.
Given an $L^p$-function on $\dOm$, we will construct its $N^{1,p}$ extension in $\Om$ using the linear extension operator for Besov boundary data. However, the mapping 
$f\in L^p(\dOm) \mapsto \Ext f\in N^{1,p}(\Om)$ will be nonlinear, which is to be expected in view of \cite{Pee,PelWoj}.

Instead of constructing the extension using a Whitney decomposition of $\Omega$, we will set up a 
sequence of layers inside $\Om$ whose widths depend not only on their distance from $X\setminus \Om$,
but also on the function itself. Using a partition of unity subordinate to these layers, we will glue together $N^{1,p}$ extensions of Lipschitz approximations of the boundary data. 
The core idea of such a construction can be traced back to Gagliardo~\cite{Gag} who discussed
extending $L^1(\Rbb^{n-1})$ functions to $W^{1,1}(\Rbb^n_+)$, and the construction was successfully generalized later in \cite{MalShaSni,Vit}. 

We start by approximating $f$ in $L^p(\dOm)$ by a sequence of Lipschitz functions $\{f_k\}_{k=1}^\infty$
such that $\|f_{k+1} - f_k\|_{L^p(\dOm)} \le 2^{2-k} \|f\|_{L^p(\dOm)}$.
Note that this rate of convergence of $f_k$ to $f$ also ensures that $f_k\to f$ pointwise
$\Hcal$-a.e.\@ in $\dOm$. For technical reasons, we choose
$f_1 \equiv 0$.

Next, we choose a decreasing sequence of real numbers $\{\rho_k\}_{k=1}^\infty$ such that:
\begin{itemize}
  \setlength{\parskip}{0pt}
  \setlength{\itemsep}{2pt plus 1pt minus 2pt}
	\item $\rho_1 \le \diam(\Omega)/2$;
  \item $0<\rho_{k+1} \le \rho_k /2$;
  \item $\sum_k \rho_k \LIP(f_{k+1}, \dOm) \le C \|f\|_{L^p(\dOm)}$.
\end{itemize}
These will now be used to define layers in $\Omega$. Let
\[
  \psi_k (x) = \max\biggl\{ 0, \min\biggl\{1, \frac{\rho_k - \dist(x, X\setminus \Om)}{\rho_k - \rho_{k+1}} \biggr\} \biggr\},
  \quad x \in \Om.
\]
Recall that $\Om(\rho)$ denotes the shell $\{x\in \Om: \dist(x, X\setminus \Om) < \rho\}$. The sequence of functions $\{\psi_{k-1} - \psi_k: k=2,3,\ldots\}$ serves as a partition of unity in 
$\Om(\rho_2)$ subordinate to the system of layers given by $\{ \Om(\rho_{k-1}) \setminus \Om(\rho_{k+1}): k=2,3,\ldots\}$.

Since Lipschitz continuous functions lie in the Besov class $B_{p,p}^0$, we can apply the linear extension 
operator $E: B_{p,p}^0(\dOm) \to N^{1,p}(\Om)$, whose properties were established in
Section~\ref{sec:Besov-extension}, to define the extension of $f \in L^p(\dOm)$ by extending its Lipschitz 
approximations in layers, i.e.,
\begin{equation}
  F(x) = \Ext f(x) := \sum_{k=2}^\infty \bigl(\psi_{k-1}(x) - \psi_{k}(x)\bigr) Ef_k(x) = \sum_{k=1}^\infty \psi_k(x) \bigl(Ef_{k+1}(x) - Ef_{k}(x)\bigr), \quad x\in\Om.
  \label{eq:Lp-ext}
\end{equation}

The following result shows that the above extension is in the class $N^{1,p}(\Om)$ with appropriate
norm bounds since $\Lip F$ is an upper gradient of $F$, which is locally Lipschitz by construction.
\begin{pro}\label{pro:trace-Lp}
Given $f \in L^p(\dOm)$, the extension defined by \eqref{eq:Lp-ext} satisfies
\begin{align*}
  \| F\|_{L^p(\Omega)} & \lesssim \diam(\Omega)^p \|f\|_{L^p(\dOm)} \quad\text{and}\\
  \|\Lip F\|_{L^p(\Omega)} & \lesssim (1+\Hcal(\dOm)^{1/p}) \|f\|_{L^p(\dOm)}\,.
\end{align*}
\end{pro}
\begin{proof}
Corollary~\ref{cor:layer-est-Fn} allows us to obtain the desired $L^p$ estimate for $F$. Since the
extension operator for Besov boundary data $E$ is linear, we have that $Ef_{k+1}-Ef_{k}=E(f_{k+1}-f_{k})$. Therefore,
\begin{align*}
  \| F\|_{L^p(\Omega)} &\le \sum_{k=1}^\infty \| \psi_k E(f_{k+1}-f_{k})\|_{L^p(\Omega)}
   \le \sum_{k=1}^\infty \| E(f_{k+1}-f_{k})\|_{L^p(\Omega(\rho_{k}))} \\
   & \lesssim \sum_{k=1}^\infty \rho_{k}^p \|f_{k+1}-f_{k}\|_{L^p(\dOm)}
  \lesssim \rho_1^p \|f\|_{L^p(\dOm)}
    \lesssim \diam(\Om)^p  \|f\|_{L^p(\dOm)}\,.
\end{align*}
In order to obtain the $L^p$ estimate  for $\Lip F$, we first apply the product rule for locally Lipschitz functions, which yields that
\begin{align*}
  \Lip F & = \sum_{k=1}^\infty\bigl( |E(f_{k+1}-f_{k})|\Lip\psi_k + \psi_k \Lip(E(f_{k+1}-f_{k})) \bigr)\\
  & \le \sum_{k=1}^\infty \biggl(\frac{|E(f_{k+1}-f_{k})| \chi_{\Omega(\rho_k)}}{\rho_k-\rho_{k+1}}
  + \chi_{\Omega(\rho_k)} \Lip(E(f_{k+1}-f_{k}))\biggr)
\end{align*}
Thus,
\[
  \|\Lip F\|_{L^p(\Omega(\rho_1))} \le \sum_{k=1}^{\infty} \biggl(\biggl\| \frac{E(f_{k+1}-f_{k})}{\rho_k-\rho_{k+1}} \biggr\|_{L^p(\Omega(\rho_k))} + \bigl\| \Lip E(f_{k+1}-f_{k}) \bigr\|_{L^p(\Omega(\rho_k))}\biggr)\,.
\]
It follows from Corollary~\ref{cor:layer-est-Fn} that
\begin{align*}
 \sum_{k=1}^{\infty} \biggl\| \frac{E(f_{k+1}-f_{k})}{\rho_k-\rho_{k+1}} \biggr\|_{L^p(\Omega(\rho_k))} & \lesssim \sum_{k=1}^{\infty}  \frac{\rho_k}{\rho_k-\rho_{k+1}} \|f_{k+1}-f_{k}\|_{L^p(\dOm)} \\
  &  \lesssim  \sum_{k=1}^{\infty}  \|f_{k+1}-f_{k}\|_{L^p(\dOm)} \lesssim \|f\|_{L^p(\dOm)} \,.
\end{align*}
Next, we apply Lemma~\ref{lem:Lip-layer-est-grad} to see that
\begin{align*}
  \sum_{k=1}^\infty \bigl\| \Lip E(f_{k+1}-f_{k}) \bigr\|_{L^p(\Omega(\rho_k))} & \lesssim  \sum_{k=1}^\infty \rho_k \Hcal(\dOm)^{1/p} \LIP (f_{k+1}-f_{k},\dOm) \\
  & \lesssim \Hcal(\dOm)^{1/p}  \sum_{k=1}^\infty \rho_k \bigl(\LIP (f_{k+1},\dOm)+\LIP(f_{k},\dOm)\bigr) \\
  &\lesssim \Hcal(\dOm)^{1/p}  \|f\|_{L^p(\dOm)},
\end{align*}
where we used the defining properties of $\{\rho_k\}_{k=1}^\infty$ to obtain the ultimate inequality. 
Altogether, we have shown that $\|\Lip F\|_{L^p(\Omega(\rho_1))} \lesssim (1+ \Hcal(\dOm)^{1/p}) \|f\|_{L^p(\dOm)}$. 
\end{proof}
Finally, we will show that the trace of the
extended function yields the original function back.
\begin{pro}
Let $F \in N^{1,p}(\Om)$ be the extension of $f\in L^p(\dOm)$ as constructed in \eqref{eq:Lp-ext}. Then,
\[
  \lim_{r\to 0} \fint_{B(z,r)\cap \Om} |F - f(z)| \,d\mu = 0
\]
for $\Hcal$-a.e.\@ $z\in \dOm$.
\end{pro}
\begin{proof}
Let $E_0$ be the collection of all $z\in\dOm$ for which $\lim_kf_k(z)=f(z)$, and for $k\in\Nbb$ let $E_k$ be the collection of all $z\in\dOm$ for which $TEf_k(z)=f_k(z)$ exists.
Lemma~\ref{lem:trace-Lebesgue} yields that $\Hcal(\dOm \setminus \bigcap_{k=0}^\infty E_k) = 0$.
We define also an auxiliary sequence $\{F_n\}_{n=1}^\infty$ of functions approximating $F$ by
\[
  F_n = \sum_{k=2}^n (\psi_{k-1}-\psi_k) E{f_k} + \sum_{k=n+1}^\infty (\psi_{k-1}-\psi_k) E{f_n}, \quad n\in\Nbb.
\]
Since $F_n = Ef_n$ in $\Om(\rho_n)$, the trace of $F_n$ exists on $\dOm$ and coincides with $T Ef_n = f_n$.

Fix a point $z \in \bigcap_{k=0}^\infty E_k$ and let $\eps>0$.
Then, we can find $j \in \Nbb$ such that $|f_k(z) - f(z)| < \eps$ for every $k\ge j$.
Next, we choose $k_0>j$ such that $R := \rho_{k_0}$ satisfies:
\begin{itemize}
  \setlength{\parskip}{0pt}
  \setlength{\itemsep}{2pt plus 1pt minus 2pt}
	\item $R \LIP(f_j, \dOm) < \eps$;
  \item $\fint_{B(z,r)\cap\Om} |F_j - f_j(z)|\,d\mu < \eps$ for every $r < R$;
  \item $\sum_{k=k_0}^\infty \rho_k \LIP(f_{k+1}, \dOm) < \eps$.
\end{itemize}
For every $r \in (0, \rho_{k_0+1}) \subset (0, R/2)$, we can then estimate 
\begin{align}
  \notag
  \fint_{B(z,r)\cap\Om} |F-f(z)|\,d\mu \le \fint_{B(z,r)\cap\Om} |F-F_j|\,d\mu + \fint_{B(z,r)\cap\Om} |F_j - f_j(z)|\,d\mu + |f_j(z) - f(z)|& \\
  \label{eq:Lp-trace}
  \le \fint_{B(z,r)\cap\Om} |F-F_j|\,d\mu + 2\eps \le \frac{\|F-F_j\|_{L^p(B(z,r)\cap \Om)}}{\mu(B(z,r)\cap\Om)^{1/p}} &+ 2\eps.
\end{align}
For such $r$, choose
$k_r> k_0$ such that $\rho_{k_r+1} \le r < \rho_{k_r}$. Then,
\begin{align}
  \notag
  \|F&-F_j\|_{L^p(B(z,r)\cap \Om)} \le \sum_{k=k_r}^\infty \bigl\|(\psi_{k-1} - \psi_k) |E(f_k - f_j)|\bigr\|_{L^p(B(z,r)\cap \Om)} \\
  & \le \sum_{k=k_r}^\infty \bigl\||E(f_k - f_j)|\bigr\|_{L^p(B(z,r)\cap \Om(\rho_{k-1}))}
 \label{eq:Lp-trace-A}
 \lesssim \sum_{k=k_r}^\infty \min\{r,\rho_{k-1}\} \|f_k - f_j\|_{L^p(B(z,2^8 r) \cap \dOm)}
\end{align}
by Lemma~\ref{lem:layer-est-Fn-boundaryball}. 
For $k=k_r$ in the last inequality above, we used the fact that we have chosen $r<\rho_{k_r}$ and hence
$B(z,r)\cap\Om(\rho_{k_r-1})=B(z,r)\cap\Om(r)$.

Let us, for the sake of brevity, write $U_r = B(z, 2^8 r) \cap \dOm$. As $f_k - f_j$ is Lipschitz continuous and due to the choice of $j$, we have for $k\ge j$ that
\begin{align}
  \notag
  \|f_k - f_j\|_{L^p(U_r)}
  & \le \bigl\|f_k - f_j - (f_k(z) - f_j(z))\bigr\|_{L^p(U_r)} + \bigl|f_k(z) - f_j(z)\bigr| \Hcal(U_r)^{1/p} \\
 \label{eq:Lp-trace-B}
  & \lesssim r \Hcal(U_r)^{1/p} \LIP(f_k - f_j, U_r)  + 2 \eps \Hcal(U_r)^{1/p}.
\end{align}
Observe that $r^p \Hcal(U_r) \lesssim \mu(B(z, r) \cap \Om)$ by \eqref{eq:H-upper-massbound} and the doubling condition for $\mu\lfloor_\Om$. 
Note that $\sum_{k=k_r}^\infty\rho_{k-1}\approx \rho_{k_r-1} \le R$.
Combining this with~\eqref{eq:Lp-trace-A} and \eqref{eq:Lp-trace-B} gives us that
\begin{align*}
  \|F-F_j\|_{L^p(B(z,r)\cap \Om)} & \lesssim \sum_{k=k_r}^\infty \rho_{k-1} \mu(B(z,r)\cap \Om)^{1/p} \bigl(\LIP(f_k, \dOm) + \LIP(f_j, \dOm)\bigr) \\
 & \quad + 2 \eps \mu(B(z,r)\cap \Om)^{1/p} \sum_{k=k_r}^\infty \frac{\min\{r, \rho_{k-1}\}}{r} \\
 &\lesssim \mu(B(z,r)\cap \Om)^{1/p} \biggl(\sum_{k=k_0}^\infty \bigl(\rho_k \LIP(f_{k+1}, \dOm)\bigr) +  R \LIP(f_{j}, \dOm) + \eps\biggr)\\
 & \lesssim \mu(B(z,r)\cap \Om)^{1/p} \eps.
\end{align*}
Plugging this estimate into~\eqref{eq:Lp-trace} completes the proof.
\end{proof}
\section{Counterexamples for the trace theorems}
\label{sec:examples}
In Sections~\ref{sec:tracesGen}--\ref{sec:lp-extension}, we witnessed a certain mismatch of the target space of the trace operator and the domain of the extension operators for domains whose boundary has unequal lower and upper codimension bounds, i.e., $\vartheta<\theta$ in \eqref{eq:H-lower-massbound} and \eqref{eq:H-upper-massbound}. We have seen that the trace operator maps $P^{1,p}_p(\Om)$ into $B^{1-\theta/p}_{p, \infty}(\dOm)$ but we were able to extend functions of $B^{1-\vartheta/p}_{p, p}(\dOm) \subsetneq B^{1-\theta/p}_{p, p}(\dOm)$ to a Sobolev-type function in $\Om$ with $p$-integrable gradient. 

The following example shows that the trace operator, in general, does not map into a Besov space of smoothness any higher than $1-\theta/p$ with summability of the $p$th power. 
\begin{exa}
\label{exa:T-not-surjective}
Let $X=\Rbb^2$ be endowed with the Euclidean distance and $2$-dimensional Lebesgue measure. Let $\Om = \{(x_1, x_2) \in X : 0<x_2<x_1^2 < x_1\}$. Since the boundary consists of two line segments connected by a parabola, i.e., $\dOm = \{(x_1,x_2)\in [0,1]^2: x_2=x_1^2\ \text{or}\ x_1=1\ \text{or}\ x_2=0\}$, it is natural to choose $\Hcal$ to be the $1$-dimensional Hausdorff measure on $\dOm$.

Note that for every $x \in \Om$, there is $R(x)>0$ such that $\mu(B(x,r) \cap \Om) \approx r^2$ whenever $r \in (0, R(z))$. However, $\mu(B(0, r) \cap \Om) \approx r^3$ for every $r \in (0, 2\diam \Om)$. Therefore,
\[
  \frac{\mu(B(z,r) \cap \Om)}{r} \lesssim \Hcal(B(z,r) \cap \dOm) \lesssim \frac{\mu(B(z,r) \cap \Om)}{r^2}, \quad z\in \dOm, r \in (0, 2\diam \Om).
\]
Given $p\in (2, 3)$, let $\alpha = \frac{3}{p}-1 \in (0,\frac12)$ and define a function
\[
  u(x_1, x_2) = \frac{1}{x_1^\alpha \log (\frac{e}{x_1})}, \quad (x_1, x_2) \in \Om.
\]
Since $u$ is really a function of one variable only, it is easy to verify (e.g., by the mean value theorem) that $g(x_1, x_2) \coloneq C_\alpha \frac{u(x_1, x_2)}{x_1}$ is a Haj\l asz gradient of $u$ in $\Om$. Then,
\[
  \int_\Om g(x)^p\,dx \approx \int_0^1 \int_0^{x_1^2} \frac{1}{x_1^{(\alpha + 1)p} \log(\frac{e}{x_1})^p}\,dx_2\, dx_1 \approx \int_0^1 \frac{1}{x_1 \log(\frac{e}{x_1})^p}\,dx_1 < \infty.
\]
Hence, $u \in M^{1,p}(\Om) \subset P^{1,p}_p(\Om)$. Theorem~\ref{thm:TraceOpB} yields that $Tu \in B^{1-2/p}_{p,\infty}(\dOm)$. If we restrict our attention to the horizontal line segment $\dOm$, we see that
\[
  \int_{[0,1]\times\{0\}} (Tu)^q \,d\Hcal = \int_0^1 \frac{dt}{t^{\alpha q} \log(\frac{e}{t})^q},
\]
which is convergent if and only if $q \le 1/\alpha$. In particular, $Tu \notin L^q(\dOm)$ for any $q>1/\alpha$. From Remark~\ref{rem:sharp-besov-emb}, we can deduce that $B^{1-2/p + \eps}_{p, \infty}(\dOm) \subset \wk L^{1/(\alpha - \eps)}(\dOm)$ for any $\eps \ge 0$. Consequently, $Tu \notin B^{1-2/p + \eps}_{p, \infty}(\dOm)$ for any $\eps > 0$ since otherwise we would have $Tu \in L^q(\dOm)$ for some $q>1/\alpha$.
\end{exa}
Note that $Tu$ in the example above actually lies in $B^{1-\theta/p}_{p,p}(\dOm) \subsetneq B^{1-\theta/p}_{p,\infty}(\dOm)$. In Theorem~\ref{thm:TraceJohn}, the trace operator was shown to map $M^{1,p} (\Om)$ into $B^{1-\theta/p}_{p,p}(\dOm)$ rather than $B^{1-\theta/p}_{p,\infty}(\dOm)$ provided that $\Om$ is a John domain. However, it is currently unclear whether the target space is $B^{1-\theta/p}_{p,p}(\dOm)$ also for general non-John domains.

Considering the example above, let us now argue that the trace operator there is \emph{not surjective}, be it for the target space $B^{1-\theta/p}_{p,\infty}(\dOm)$, or for $B^{1-\theta/p}_{p,p}(\dOm)$. Recall that $\vartheta = 1 < \theta = 2$. The measure $\Hcal$ restricted to the vertical segment $\dOm_V \coloneq \{1\} \times [0,1] \subset \dOm$, is codimension $1$ regular with respect to $\mu\lfloor_\Om$. Hence, $Tu|_{\dOm_V} \in B^{1-1/p}_{p,\infty}(\dOm_V) \subsetneq B^{1-2/p}_{p,\infty}(\dOm_V)$ whenever $u \in P^{1,p}_p(\Om)$.

The following example shows that the issue with $T$ not being surjective can arise when $\vartheta < \theta$ even if $\Om$ is a uniform domain that admits a $1$-Poincar\'e inequality.
\begin{exa}
\label{exa:T-not-surjective-uniform}
Let $X = \Rbb^2$ be endowed with the Euclidean distance and measure $d\mu(x) = |x|\,dx$. Then, $X$ supports a $1$-Poincar\'e inequality by \cite[Example~1.6]{HeKiMa}. Let $\Om = [0,1]^2$. Since $\Om$ is a uniform domain, it also supports a $1$-Poincar\'e inequality by \cite[Theorem~4.4]{BjoSha}, whence $N^{1,p}(\Om) = M^{1,p}(\Om) = P^{1,p}_p(\Om)$ for every $p>1$. Suppose that $\dOm$ is endowed with $1$-dimensional Hausdorff measure. Then, $1=\vartheta < \theta = 2$. Given $p \in (2,3)$, let $\alpha = \frac 3p - 1$. Thus, $u(x) = |x|^{-\alpha} \log(e/|x|)^{-1} \in N^{1,p}(\Om)$. By Theorem~\ref{thm:TraceJohn}, $Tu \in B^{1-2/p}_{p,p}(\dOm)$, but it lacks any higher smoothness (for the exponent $p$) similarly as in Example~\ref{exa:T-not-surjective}.
On the other hand, $T: N^{1,p}(\Om) \to B^{1-2/p}_{p,p}(\dOm)$ is not surjective since the restriction of the trace to any compact set $K\subset \dOm \setminus \{0\}$ lies in $B^{1-1/p}_{p,p}(K)$.
\end{exa}

Let us now also look at the assumptions of Theorem~\ref{thm:TraceOp-theta=p}, where the trace of a $P^{1, \theta}_{\theta}(\Om)$ function was established to be $L^\theta(\dOm)$ provided that its $\theta$-PI gradient lies in a weighted $L^\theta(\Om)$ space. The following example shows that the weight in the integrability condition for the PI gradient cannot be removed even in uniform domains with smooth boundary. 
\begin{exa}
\label{exa:T-to-Lp-fail}
Let $\Om = B(0,1) \subset \Rbb^2$ be endowed with Euclidean distance. Let the measure in $\Om$ be given by $d\mu(x) = \dist(x, \dOm)\,dx$.
Let $\Hcal$ be the 1-dimensional Hausdorff measure on $\dOm$. Then, $\vartheta = \theta = 2$ in \eqref{eq:H-lower-massbound} and \eqref{eq:H-upper-massbound}.
Let $\eps \in (0, \frac12)$ and define
\[
  u(x) = \log \biggl(\frac{e}{\dist(x,\dOm)}\biggr)^{\eps}, \quad g(x) = \frac{C}{\dist(x,\dOm) \log (\frac{e}{\dist(x,\dOm)})^{1-\eps}}, \quad x\in \Om.
\]
Then, $g$ is a Haj\l asz gradient of $u$ provided that  the constant $C>0$ if sufficiently large. Moreover,
\[
  \int_\Om g(x)^2\,d\mu(x) = \int_\Om \frac{\dist(x, \dOm)}{ \dist(x,\dOm)^2 \log (\frac{e}{\dist(x,\dOm)})^{2-2\eps}}\,dx \approx \int_0^1 \frac{r (1-r)}{(1-r)^2 \log(\frac{e}{1-r})^{2-2\eps}} \,dr < \infty.
\]
Hence, $u \in M^{1,2}(\Om) \subset P^{1,2}_{2}(\Om)$. However, for every $z \in \dOm$,
\[
  T_R u(z) = \fint_{B(z, R) \cap \Om} \log \biggl(\frac{e}{\dist(x,\dOm)}\biggr)^{\eps} \,d\mu(x) \ge \log \biggl(\frac{e}{R}\biggr)^{\eps} \to \infty\quad \text{as }R \to 0.
\]
It follows from the definition of the trace~\eqref{eq:defoftrace} that necessarily $Tu(z) = \lim_{R\to 0} T_R u(z)$ for $\Hcal$-a.e.\@ $z\in\dOm$. In this case, we would obtain that $T u \equiv \infty$ on $\dOm$, which would violate \eqref{eq:defoftrace}.
\end{exa}
Moreover, the weight condition of Theorem~\ref{thm:TraceOp-theta=p} is essentially sharp whenever $\theta > 1$, which is a consequence of the following proposition. On the other hand, if $\vartheta = \theta = 1$ and $\Om$ admits a $1$-Poincar\'e inequality, then $T: BV(\Om) \to L^1(\dOm)$ by \cite{LahSha}, where $BV(\Om) \supset N^{1,1}(\Om) = P^{1,1}_1(\Om)$ denotes the space of functions of bounded variation (see~\cite{Mir}).
\begin{pro}
\label{pro:T-intoLp-sharpness}
Let $\eps \in (0, 1)$. Then, there are:
\begin{itemize}
	\item a decreasing function $w: (0, \infty) \to [1, \infty)$ such that $\int_0^1 \frac{dt}{t w(t)} < \infty$;
  \item a uniform domain $\Om \Subset X$ with a doubling measure $\mu\lfloor_\Om$ that admits $1$-Poincar\'e inequality;
  \item a measure $\Hcal$ on $\dOm$ that satisfies both \eqref{eq:H-lower-massbound} and \eqref{eq:H-upper-massbound} with $\vartheta = \theta>1$; and
  \item a function $u \in M^{1,\theta}(\Om)$ such that $\tilde{g}_\eps(x) \coloneq g(x) w(\dist(x, \dOm))^{1-\eps} \in L^\theta(\Om)$, where $g \in L^\theta(\Om)$ is a Haj\l asz gradient (and hence a $1$-PI gradient) of $u$
\end{itemize}
such that $u$ does not have a trace in the sense of \eqref{eq:defoftrace}. Moreover, $\fint_{B(z,R)\cap \Om} u(x) \, d\mu(x) \to \infty$ as $R\to 0$ for every $z\in\dOm$.
\end{pro}
Recall that if $w$, $\Om$, $\mu$, and $\Hcal$ are as in the proposition above and $\tilde{g}_0 \in L^\theta(\Om)$, then $Tu \in L^\theta(\dOm)$ exists by Theorem~\ref{thm:TraceOp-theta=p}.
\begin{proof}
Let $\eps>0$ be given. Let $X = \Rbb^2$ be endowed with the Euclidean distance and set $\Om = B(0,1)$. Pick $n \ge 2/\eps$. Define $d\mu(x) = (\dist(x, X \setminus \Om)^{n-1} + \chi_{X \setminus \Om}(x))\,dx$. By \cite[Theorem~3.4\,(i)]{Hur}, cf.\@ \cite[Lemmata~8.3 and 8.4]{HurT}, $\Om$ admits a $1$-Poincar\'e inequality.
Let $\Hcal$ be the $1$-dimensional Hausdorff measure on $\dOm$, in which case $\vartheta = \theta = n$. Let $w(t) = \max\{1, \log (e/t)^{1+\eps/4}\}$, $t\in (0, \infty)$. 

Then, the function $u(x) = \log(\frac{e}{\dist(x, \dOm)})^{\eps/4}$ lies in $N^{1,\theta}(\Om) = M^{1, \theta}(\Om)$, but lacks a trace since $\lim_{R\to 0} T_R u(z) = \infty$ for every $z \in \dOm$. Note that
\[
  g(x) = \frac{C}{\dist(x,\dOm) \log (\frac{e}{\dist(x,\dOm)})^{1-\eps/4}}
\]
is a Haj\l asz gradient of $u$ (provided that $C>0$ is sufficiently large) and hence
\[
  \tilde{g}_\eps = \frac{C \log (\frac{e}{\dist(x,\dOm)})^{(1+\eps/4)(1-\eps)}}{\dist(x,\dOm) \log (\frac{e}{\dist(x,\dOm)})^{1-\eps/4}} = \frac{C }{\dist(x,\dOm) \log (\frac{e}{\dist(x,\dOm)})^{\eps/2 + \eps^2/4}}
\]
Let us now verify that $\tilde{g}_\eps \in L^\theta(\Om)$. Recall that $\theta = n \ge 2/\eps$.
\begin{align*}
  \| \tilde{g}_\eps \|_{L^n(\Om)}^n & \approx \int_\Om \frac{\dist(x,\dOm)^{n-1}\,dx}{\dist(x,\dOm)^n \log (\frac{e}{\dist(x,\dOm)})^{n\eps/2 + n\eps^2/4}} \\
  & \le \int_\Om \frac{dx}{\dist(x,\dOm) \log (\frac{e}{\dist(x,\dOm)})^{1 + \eps/2}} \approx \int_0^1 \frac{r\,dr}{(1-r) \log(\frac{e}{1-r})^{1+\eps/2}} < \infty.
  \qedhere
\end{align*}
\end{proof}

\end{document}